\documentclass[11pt,reqno]{amsart}
\usepackage{amssymb,amscd,amsfonts,amsbsy,enumerate}
\usepackage{latexsym}
\usepackage{exscale}
\usepackage{amsmath,amsthm,amsfonts}
\usepackage{mathrsfs}

\usepackage[colorlinks=true, pdfstartview=FitV, linkcolor=blue, citecolor=red, urlcolor=blue]{hyperref}

\usepackage{color}

\numberwithin{equation}{section}

\def\Xint#1{\mathchoice
   {\XXint\displaystyle\textstyle{#1}}%
   {\XXint\textstyle\scriptstyle{#1}}%
   {\XXint\scriptstyle\scriptscriptstyle{#1}}%
   {\XXint\scriptscriptstyle\scriptscriptstyle{#1}}%
   \!\int}
\def\XXint#1#2#3{{\setbox0=\hbox{$#1{#2#3}{\int}$}
     \vcenter{\hbox{$#2#3$}}\kern-.5\wd0}}
\def\aver#1{\Xint-_{#1}}

\newcommand{\meanint}{{\int{\mkern-19mu}-}}

\allowdisplaybreaks

\newtheorem{theorem}{Theorem}[section]
\newtheorem{lemma}[theorem]{Lemma}
\newtheorem{corollary}[theorem]{Corollary}
\newtheorem{proposition}[theorem]{Proposition}

\theoremstyle{remark}

\begin{document}
\allowdisplaybreaks

\title[Poisson semigroups for systems in the upper-half space]
{On the $L^p$-Poisson semigroup associated with elliptic systems}

\author{Jos\'e Mar{\'\i}a Martell}
\address{Jos\'e Mar{\'\i}a Martell
\\
Instituto de Ciencias Matem\'aticas CSIC-UAM-UC3M-UCM
\\
Consejo Superior de Investigaciones Cient{\'\i}ficas
\\
C/ Nicol\'as Cabrera, 13-15
\\
E-28049 Madrid, Spain} \email{chema.martell@icmat.es}

\author{Dorina Mitrea}
\address{Dorina Mitrea
\\
Department of Mathematics
\\
University of Missouri
\\
Columbia, MO 65211, USA} \email{mitread@missouri.edu}

\author{Irina Mitrea}
\address{Irina Mitrea
\\
Department of Mathematics
\\
Temple University\!
\\
1805\,N.\,Broad\,Street
\\
Philadelphia, PA 19122, USA} \email{imitrea@temple.edu}

\author{Marius Mitrea}
\address{Marius Mitrea
\\
Department of Mathematics
\\
University of Missouri
\\
Columbia, MO 65211, USA} \email{mitream@missouri.edu}

\thanks{The first author has been supported in part by MINECO Grant MTM2010-16518,
ICMAT Severo Ochoa project SEV-2011-0087. He also acknowledges that
the research leading to these results has received funding from the European Research
Council under the European Union's Seventh Framework Programme (FP7/2007-2013)/ ERC
agreement no. 615112 HAPDEGMT. The second author has been supported in part by a
Simons Foundation grant $\#\,$200750, the third author has
been supported in part by US NSF grant $\#\,$0547944, while the fourth author has been
supported in part by the Simons Foundation grant $\#\,$281566, and by a University of
Missouri Research Leave grant. This work has been possible thanks to the support and hospitality
of \textit{ICMAT, Consejo Superior de Investigaciones
Cient{\'\i}ficas} (Spain) and the \textit{University of Missouri} (USA).
The authors express their gratitude to these institutions.}

\date{\today}

\subjclass[2010]{Primary: 35J47, 47D06, 47D60. Secondary: 35C15, 35J57, 42B37.}

\keywords{Poisson semigroup, second order elliptic system, infinitesimal generator, graph Lipschitz domain,
higher order system, Lam\'e system, Poisson kernel, nontangential maximal function, Whitney arrays,
Sobolev space, Dirichlet problem, Regularity problem, Dirichlet-to-Normal map}

\begin{abstract}
We study the infinitesimal generator of the Poisson semigroup in $L^p$ associated
with homogeneous, second-order, strongly elliptic systems with constant complex
coefficients in the upper-half space, which is proved to be the Dirichlet-to-Normal
mapping in this setting. Also, its domain is identified
as the linear subspace of the $L^p$-based Sobolev space of order one
on the boundary of the upper-half space consisting of functions for which the
Regularity problem is solvable. Moreover, for a class of systems
containing the Lam\'e system, as well as all second-order, scalar elliptic operators,
with constant complex coefficients, the action of the infinitesimal generator is
explicitly described in terms of singular integral operators whose kernels involve
first-order derivatives of the canonical fundamental solution of the given system.
Furthermore, arbitrary powers of the infinitesimal generator of the said Poisson semigroup
are also described in terms of higher order Sobolev spaces and a higher order Regularity
problem for the system in question. Finally, we indicate how our techniques may adapted
to treat the case of higher order systems in graph Lipschitz domains.
\end{abstract}

\maketitle

\allowdisplaybreaks

\section{Introduction}
\setcounter{equation}{0}
\label{S-1}

Let $M\in{\mathbb{N}}$ and consider the second-order, homogeneous, $M\times M$ system,
with constant complex coefficients, written (with the usual convention of summation
over repeated indices in place) as
\begin{equation}\label{L-def}
Lu:=\Bigl(\partial_r(a^{\alpha\beta}_{rs}\partial_s u_\beta)
\Bigr)_{1\leq\alpha\leq M},
\end{equation}
when acting on a ${\mathscr{C}}^2$ vector-valued function $u=(u_\beta)_{1\leq\beta\leq M}$
defined in an open subset of ${\mathbb{R}}^n$. Assume that $L$ is {\tt strongly}
{\tt elliptic} in the sense that there exists $\kappa_o>0$ such that
\begin{equation}\label{L-ell.X}
\begin{array}{c}
{\rm Re}\,\bigl[a^{\alpha\beta}_{rs}\xi_r\xi_s\overline{\eta_\alpha}
\eta_\beta\,\bigr]\geq\kappa_o|\xi|^2|\eta|^2\,\,\mbox{ for every}
\\[8pt]
\xi=(\xi_r)_{1\leq r\leq n}\in{\mathbb{R}}^n\,\,\mbox{ and }\,\,
\eta=(\eta_\alpha)_{1\leq\alpha\leq M}\in{\mathbb{C}}^M.
\end{array}
\end{equation}

Every $L$ as in \eqref{L-def}-\eqref{L-ell.X} has
a Poisson kernel $P^L$, described in detail in Theorem~\ref{kkjbhV}.
This Poisson kernel has played a basic role in the treatment of the $L^p$-Dirichlet
boundary value problem for $L$ in the upper-half space from \cite{K-MMMM}.
To formally state this result, we shall need a little more notation.
Specifically, we agree to identify the boundary of the upper-half space
\begin{equation}\label{RRR-UpHs}
{\mathbb{R}}^{n}_{+}:=\big\{x=(x',x_n)\in
{\mathbb{R}}^{n}={\mathbb{R}}^{n-1}\times{\mathbb{R}}:\,x_n>0\big\}
\end{equation}
with the horizontal hyperplane ${\mathbb{R}}^{n-1}$ via $(x',0)\equiv x'$.
Given $\kappa>0$, at each point $x'\in\partial{\mathbb{R}}^{n}_{+}$ define
the conical nontangential approach region with vertex at $x'$ as
\begin{equation}\label{NT-1}
\Gamma_\kappa(x'):=\big\{y=(y',t)\in{\mathbb{R}}^{n}_{+}:\,|x'-y'|<\kappa\,t\big\}.
\end{equation}
Also, given a vector-valued function $u:{\mathbb{R}}^{n}_{+}\to{\mathbb{C}}^M$,
define the nontangential maximal function of $u$ by
\begin{equation}\label{NT-Fct}
\big({\mathcal{N}}u\big)(x'):=
\big({\mathcal{N}}_\kappa u\big)(x')
:=\sup\big\{|u(y)|:\,y\in\Gamma_\kappa(x')\big\},\qquad\forall\,x'\in{\mathbb{R}}^{n-1},
\end{equation}
and, whenever meaningful, define the nontangential pointwise trace of $u$ on
$\partial{\mathbb{R}}^{n}_{+}\equiv{\mathbb{R}}^{n-1}$ by
\begin{equation}\label{nkc-EE-2}
\Big(u\big|^{{}^{\rm n.t.}}_{\partial{\mathbb{R}}^{n}_{+}}\Big)(x')
:=\lim_{\Gamma_{\kappa}(x')\ni y\to (x',0)}u(y)
\quad\mbox{for }\,x'\in\partial{\mathbb{R}}^{n}_{+}\equiv{\mathbb{R}}^{n-1}.
\end{equation}

Here is the well-posedness result alluded to earlier. Below and elsewhere,
we define $P^L_t(x'):=t^{1-n}P^L(x'/t)$ for every $x'\in{\mathbb{R}}^{n-1}$ and $t>0$.

\begin{theorem}[\cite{K-MMMM}]\label{Theorem-NiceDP}
Assume that $L$ is a system as in \eqref{L-def}-\eqref{L-ell.X}, and fix some $p\in(1,\infty)$.
Then the $L^p$-Dirichlet boundary value problem for $L$ in $\mathbb{R}^{n}_{+}$,
\begin{equation}\label{Dir-BVP-Xalone}
(D_p)\,\,
\left\{
\begin{array}{l}
u\in{\mathscr{C}}^\infty(\mathbb{R}^{n}_{+}),
\\[4pt]
Lu=0\,\,\mbox{ in }\,\,\mathbb{R}^{n}_{+},
\\[4pt]
\mathcal{N}u\in L^p({\mathbb{R}}^{n-1}),
\\[6pt]
u\big|_{\partial\mathbb{R}^{n}_{+}}^{{}^{\rm n.t.}}=f\in L^p({\mathbb{R}}^{n-1}),
\end{array}
\right.
\end{equation}
is well-posed. In addition, the solution $u$ of \eqref{Dir-BVP-Xalone}
is given by
\begin{equation}\label{eq:UUPP}
u(x',t)=(P^L_t\ast f)(x')\quad\mbox{for all }\,\,(x',t)\in{\mathbb{R}}^n_{+},
\end{equation}
where $P^L$ is the Poisson kernel for $L$ in $\mathbb{R}^{n}_{+}$ and, in fact,
\begin{equation}\label{Dir-BVP-X:b}
\|\mathcal{N} u\|_{L^p({\mathbb{R}}^{n-1})}\approx\|f\|_{L^p({\mathbb{R}}^{n-1})},
\end{equation}
where the proportionality constants involved depend only on $n$, $p$, and $L$.
\end{theorem}

As a consequence of Theorem~\ref{Theorem-NiceDP} it has been shown in \cite{K-MMMM}
that, for $p\in(1,\infty)$, the Poisson kernel $P^L$ induces a natural $C_0$-semigroup
on $L^p({\mathbb{R}}^{n-1})$ in the following precise sense
(for pertinent definitions see \S\ref{S-2}).

\begin{theorem}[\cite{K-MMMM}]\label{VCXga}
Let $L$ be a system with complex coefficients as in \eqref{L-def}-\eqref{L-ell.X}
and denote by $P^L$ the Poisson kernel associated with $L$ as in Theorem~\ref{kkjbhV}.
Also, fix $p\in(1,\infty)$. Then the family $T=\{T(t)\}_{t\geq 0}$ where
$T(0):=I$, the identity operator on $L^p({\mathbb{R}}^{n-1})$ and, for each $t>0$,
\begin{equation}\label{eq:Taghb8}
\begin{array}{c}
T(t):L^p({\mathbb{R}}^{n-1})\longrightarrow L^p({\mathbb{R}}^{n-1}),
\\[6pt]
\big(T(t)f\big)(x'):=(P^L_t\ast f)(x')
\,\,\mbox{ for each }\,f\in L^p({\mathbb{R}}^{n-1}),\,\,x'\in{\mathbb{R}}^{n-1},
\end{array}
\end{equation}
is a $C_0$-semigroup on $L^p({\mathbb{R}}^{n-1})$, which satisfies
\begin{equation}\label{eq:Taghb8.77}
\sup_{t\geq 0}\big\|T(t)\big\|_{{\mathcal{L}}(L^p({\mathbb{R}}^{n-1}))}<\infty.
\end{equation}
\end{theorem}

In this paper we are interested in continuing the work initiated in \cite{K-MMMM}
by addressing the question of identifying the infinitesimal generator of the
semigroup described in Theorem~\ref{VCXga}. Later on, in Theorem~\ref{V-Naa.11.HOR}
we shall actually succeed in identifying arbitrary powers of this generator, whose
nature turns out to be intimately connected with the solvability properties of a
version of the Dirichlet problem \eqref{Dir-BVP-Xalone} for more regular boundary data.
We first describe this problem below, in \eqref{Dir-BVvku}, at the most basic level
(compare with \eqref{Dir-BVP-p:l}, dealing with the general case).

For each $p\in(1,\infty)$ denote by $L^p_1(\mathbb{R}^{n-1})$
the classical Sobolev space of order one in $\mathbb{R}^{n-1}$, consisting of
functions from $L^p(\mathbb{R}^{n-1})$ whose distributional first-order derivatives
are in $L^p(\mathbb{R}^{n-1})$. Then the version of $(D_p)$, with data from the Sobolev space
$L^p_1({\mathbb{R}}^{n-1})$, reads
\begin{equation}\label{Dir-BVvku}
(R_p)\,\,
\left\{
\begin{array}{l}
u\in{\mathscr{C}}^\infty({\mathbb{R}}^n_{+}),
\\[4pt]
Lu=0\,\,\mbox{ in }\,\,\mathbb{R}^{n}_{+},
\\[4pt]
{\mathcal{N}}u,\,{\mathcal{N}}(\nabla u)\in L^p(\mathbb{R}^{n-1}),
\\[6pt]
u\bigl|_{\partial\mathbb{R}^{n}_{+}}^{{}^{\rm n.t.}}=f\in L^p_1(\mathbb{R}^{n-1}).
\end{array}
\right.
\end{equation}
We will refer to \eqref{Dir-BVvku} as the Regularity problem for $L$ in
${\mathbb{R}}^n_{+}$, and abbreviate it as $(R_p)$. Of course, the well-posedness
of $(D_p)$ from Theorem~\ref{Theorem-NiceDP} implies uniqueness for $(R_p)$,
though the solvability of the latter boundary value problem is not known to hold
in the general class of systems with complex coefficients as in \eqref{L-def}-\eqref{L-ell.X}.

One of our main results in this paper then reads as follows.

\begin{theorem}\label{V-Naa.11}
Let $L$ be a strongly elliptic, second-order, homogeneous, $M\times M$ system,
with constant complex coefficients. Fix some $p\in(1,\infty)$ and consider the
$C_0$-semigroup $T=\{T(t)\}_{t\geq 0}$ on $L^p({\mathbb{R}}^{n-1})$ associated
with $L$ as in Theorem~\ref{VCXga}. Denote by ${\mathbf{A}}$ the infinitesimal
generator of $T$, with domain $D({\mathbf{A}})$. Then
\begin{equation}\label{eq:tfc.1-new.RRR}
\mbox{$D({\mathbf{A}})$ is a dense linear subspace of $L^p_1({\mathbb{R}}^{n-1})$}
\end{equation}
and, in fact,
\begin{equation}\label{eq:tfc.1-new}
D({\mathbf{A}})=\big\{f\in L^p_1({\mathbb{R}}^{n-1}):\,
\mbox{$(R_p)$ with boundary datum $f$ is solvable}\big\}.
\end{equation}
In particular,
\begin{equation}\label{eq:tfc.2BB}
D({\mathbf{A}})=L^p_1({\mathbb{R}}^{n-1})
\,\Longleftrightarrow\,\,\mbox{$(R_p)$ is solvable
for arbitrary data in $L^p_1({\mathbb{R}}^{n-1})$}.
\end{equation}
Moreover, given any $f\in D({\mathbf{A}})$, if $u$ solves $(R_p)$ with boundary
datum $f$ then
\begin{equation}\label{eq:hJB}
(\partial_n u)\big|^{{}^{\rm n.t.}}_{\partial{\mathbb{R}}^n_{+}}\,\,
\mbox{ exists a.e.~in ${\mathbb{R}}^{n-1}$},
\end{equation}
and ${\mathbf{A}}$ acts on $f$ like the Dirichlet-to-Normal operator, i.e.,
\begin{equation}\label{mJBVV}
{\mathbf{A}}f=(\partial_n u)\big|^{{}^{\rm n.t.}}_{\partial{\mathbb{R}}^n_{+}}.
\end{equation}

Finally, let $E=\big(E_{\gamma\beta}\big)_{1\leq\gamma,\beta\leq M}$
be the canonical fundamental solution for $L$ {\rm (}from Theorem~\ref{FS-prop}{\rm )}, and
make the additional assumption that $L=\partial_r A_{rs}\partial_s$ for some family
of complex matrices
\begin{equation}\label{eq:Asd.444}
A_{rs}=\bigl(a_{rs}^{\,\alpha\beta}\bigr)_{1\leq\alpha,\beta\leq M},\qquad
1\leq r,s\leq n,
\end{equation}
satisfying
\begin{equation}\label{eq:Asd.FFF}
\begin{array}{l}
a^{\beta\alpha}_{rn}\big(\partial_r E_{\gamma\beta}\big)(x',0)=0
\,\,\mbox{ for each }\,\,x'\in{\mathbb{R}}^{n-1}\setminus\{0\},
\\[6pt]
\mbox{for every fixed multi-indices }\,\,\alpha,\gamma\in\{1,\dots,M\}.
\end{array}
\end{equation}
Then $D({\mathbf{A}})=L^p_1({\mathbb{R}}^{n-1})$ and,
for each $\gamma\in\{1,\dots,M\}$ and $f=(f_\alpha)_{1\leq\alpha\leq M}
\in L^p_1({\mathbb{R}}^{n-1})$,
\begin{align}\label{mdiab}
({\mathbf{A}}f)_\gamma(x')
&=
-\sum\limits_{s=1}^{n-1}
\Big(\big[\big(a^{\sigma\tau}_{nn}\big)_{1\leq\sigma,\tau\leq M}\big]^{-1}
\Big)_{\gamma\beta}\,a^{\beta\alpha}_{ns}\,(\partial_sf_{\alpha})(x')
\nonumber\\[4pt]
&\qquad
-2\sum\limits_{s=1}^{n-1}a^{\beta\alpha}_{rs}
\lim_{\varepsilon\to0^{+}}
\int\limits_{\substack{y'\in{\mathbb{R}}^{n-1}\\ |x'-y'|>\varepsilon}}
(\partial_r E_{\gamma\beta})(x'-y',0)(\partial_sf_\alpha)(y')\,dy',
\end{align}
for a.e. $x'\in{\mathbb{R}}^{n-1}$.
\end{theorem}

The name `Dirichlet-to-Normal map' used for the action of ${\mathbf{A}}$
is justified by \eqref{eq:tfc.1-new} and \eqref{mJBVV}, since for each
$f\in D({\mathbf{A}})$ there exists $u$ solution of $(R_p)$, so in particular $u$
is a solution of the Dirichlet problem $(D_p)$ with boundary datum $f$
in ${\mathbb{R}}^n_{+}$, and ${\mathbf{A}}f$ equals the (inner)
normal derivative on the boundary of ${\mathbb{R}}^n_{+}$ of this solution.

Regarding the nature of \eqref{eq:Asd.FFF}, the starting point is the observation that
there are multiple coefficient tensors inducing the same given system $L$. While
neither the strong ellipticity condition for $L$ nor the fundamental solution
$E=\big(E_{\gamma\beta}\big)_{1\leq\gamma,\beta\leq M}$ for the system
$L$, explicitly described in Theorem~\ref{FS-prop}, depend on
the choice of the coefficient tensor used to represent $L$, other entities
related to $L$ are affected by such a choice. For example, this is the case
for the conormal derivative, on which we wish to elaborate. Concretely, each
writing of $L$ as in \eqref{L-def}, corresponding to a coefficient tensor
\begin{equation}\label{eq:Asd}
A=\bigl(a_{rs}^{\,\alpha\beta}\bigr)
_{\substack{1\leq r,s\leq n\\[1pt] 1\leq\alpha,\beta\leq M}}
\in{\mathbb{C}}^{\,nM}\times{\mathbb{C}}^{\,nM},
\end{equation}
induces a conormal derivative associated with a domain $\Omega\subset{\mathbb{R}}^n$
with outward unit normal $\nu=(\nu_1,\dots,\nu_n)$ acting on a function
$u=(u_\beta)_{1\leq\beta\leq M}$ according to the formula
\begin{equation}\label{eq:Asd.5f56}
\partial^{A}_\nu u:=\big(a_{rs}^{\,\alpha\beta}\nu_r\partial_s
u_\beta\big)_{\substack{1\leq\alpha\leq M}}\,\,\mbox{ on }\,\,\partial\Omega,
\end{equation}
In this piece of terminology, \eqref{eq:Asd.FFF} expresses the fact that
\begin{equation}\label{eq:grr6}
\parbox{9.90cm}
{the conormal derivative of $L^\top$, the transposed of $L$, annihilates
$E_{\gamma\cdot}=(E_{\gamma\beta})_{1\leq\beta\leq M}$ on $\partial{\mathbb{R}}^n_{+}$,
for each $\gamma\in\{1,\dots,M\}$.}
\end{equation}
Sufficient conditions guaranteeing the existence of a coefficient tensor for $L$
satisfying \eqref{eq:grr6}, formulated exclusively in terms of (the inverse of)
the symbol of the system $L$, are given in Proposition~\ref{UaalpIKL-PP}.
Following \cite{MaMiMiMi}, \cite{EE-MMMM}, here we wish to note that examples
of systems $L$ having this property include all elliptic scalar operators
$L={\rm div}{A}\,\nabla$ with
${A}=(a_{rs})_{1\leq r,s\leq n}\in{\mathbb{C}}^{n\times n}$,
as well as the Lam\'e system
\begin{eqnarray}\label{TYd-YG-76g}
Lu:=\mu\Delta u+(\lambda+\mu)\nabla{\rm div}\,u,\qquad u=(u_1,\dots,u_n)\in{\mathscr{C}}^2,
\end{eqnarray}
where the constants $\lambda,\mu\in{\mathbb{R}}$ (typically called Lam\'e moduli)
are assumed to satisfy
\begin{eqnarray}\label{Yfhv-8yg}
\mu>0\quad\mbox{ and }\quad 2\mu+\lambda>0.
\end{eqnarray}

Since these special cases are of independent interest, we consider them separately
in the next corollary. To state it, we agree to let $\omega_{n-1}$ denote the area
of the unit sphere $S^{n-1}$ in ${\mathbb{R}}^n$, and to let $\delta_{jk}$
(defined to be one if $j=k$, and zero otherwise) denote the standard Kronecker symbol.
Also, let `dot' denote the canonical inner product in ${\mathbb{R}}^n$.

\begin{corollary}\label{yreeCCC}
Suppose $p\in(1,\infty)$.
\begin{enumerate}
\item [(i)] Assume that ${A}=(a_{rs})_{1\leq r,s\leq n}\in{\mathbb{C}}^{n\times n}$
satisfies the ellipticity condition
\begin{equation}\label{sec-or-aEEE}
{\rm Re}\big[a_{rs}\xi_r\xi_s\big]\geq c|\xi|^2,
\qquad\forall\,\xi=(\xi_1,\dots,\xi_n)\in{\mathbb{R}}^n,
\end{equation}
for some $c>0$, and consider the scalar operator
\begin{equation}\label{eq:Kbag}
L:={\rm div}{A}\,\nabla.
\end{equation}
Then the infinitesimal generator ${\mathbf{A}}$ for the Poisson semigroup in
$L^p({\mathbb{R}}^{n-1})$ {\rm (}associated with this $L$ as in Theorem~\ref{VCXga}{\rm )}
has domain $L^p_1({\mathbb{R}}^{n-1})$ and its action
on any function $f\in L^p_1({\mathbb{R}}^{n-1})$ is given by
\begin{align}\label{mdiab-scalar}
({\mathbf{A}}f)(x')
&=
-\sum\limits_{s=1}^{n-1}\frac{a_{ns}+a_{sn}}{2a_{nn}}\,(\partial_sf)(x')
\\[4pt]
&\,
-\frac{2}{\omega_{n-1}\sqrt{{\rm det}\,({A}_{\rm sym})}}
\lim_{\varepsilon\to0^{+}}
\int\limits_{\substack{y'\in{\mathbb{R}}^{n-1}\\ |x'-y'|>\varepsilon}}
\frac{\sum\limits_{s=1}^{n-1}(x_s-y_s)(\partial_sf)(y')}{
\Big[\big(({A}_{\rm sym})^{-1}(x'-y',0)\big)\cdot (x'-y',0)\Big]^{\frac{n}{2}}}
\,dy',
\nonumber
\end{align}
for a.e. $x'\in{\mathbb{R}}^{n-1}$, where
${A}_{\rm sym}:=\frac{1}{2}({A}+{A}^\top)$
stands for the symmetric part of the coefficient matrix ${A}$.

In particular, in the case of the Laplacian,
\begin{align}\label{mdiab-Laplacian}
({\mathbf{A}}f)(x')
=
-\frac{2}{\omega_{n-1}}\lim_{\varepsilon\to 0^{+}}
\int\limits_{\substack{y'\in{\mathbb{R}}^{n-1}\\ |x'-y'|>\varepsilon}}
\frac{\sum\limits_{s=1}^{n-1}(x_s-y_s)(\partial_sf)(y')}{|x'-y'|^n}\,dy',
\end{align}
for a.e. $x'\in{\mathbb{R}}^{n-1}$, i.e.,
\begin{align}\label{mdiab-Laplacian.2}
{\mathbf{A}}=-\sum\limits_{s=1}^{n-1}R_s\partial_s
\end{align}
where $\{R_s\}_{1\leq s\leq n-1}$ are the Riesz transforms in ${\mathbb{R}}^{n-1}$
(as defined in \cite[p.\,57]{St70}) or, alternatively,
\begin{align}\label{mdiab-Lap:LL}
{\mathbf{A}}=-\sqrt{-\Delta_{n-1}}
\end{align}
where $\Delta_{n-1}:=\sum\limits_{s=1}^{n-1}\partial_s^2$ is the Laplacian
in ${\mathbb{R}}^{n-1}$ {\rm (}and $\sqrt{-\Delta_{n-1}}$ is defined as a
Fourier multiplier with symbol $|\xi'|${\rm )}.
\item[(ii)] Let $L$ be the Lam\'e system \eqref{TYd-YG-76g} with Lam\'e moduli
as in \eqref{Yfhv-8yg}. Then the domain of the infinitesimal generator ${\mathbf{A}}$ for
the Poisson semigroup in $L^p({\mathbb{R}}^{n-1})$ associated with the Lam\'e system
as in Theorem~\ref{VCXga} is the Sobolev space $L^p_1({\mathbb{R}}^{n-1})$ and for each
$f=(f_1,\dots,f_n)\in L^p_1({\mathbb{R}}^{n-1})$
\begin{align}\label{mdiab-Lame}
({\mathbf{A}}f)_\gamma(x')
&=
-
\sum_{s=1}^{n-1}\frac{\mu+\lambda}{3\mu+\lambda}\left[
\delta_{\gamma n}\delta_{s\alpha}+\delta_{n\alpha}\delta_{\gamma s}\right]
(\partial_sf_\alpha)(x')
\\[4pt]
&\,\,\,
-\frac{4\mu}{(3\mu+\lambda)\omega_{n-1}}\sum\limits_{s=1}^{n-1}
\lim_{\varepsilon\to 0^{+}}
\int\limits_{\substack{y'\in{\mathbb{R}}^{n-1}\\ |x'-y'|>\varepsilon}}
\frac{x_s-y_s}{|x'-y'|^n}\,(\partial_sf_\gamma)(y')\,dy'
\nonumber\\[4pt]
&\,\,\,
-\frac{2n(\mu+\lambda)}{(3\mu+\lambda)\omega_{n-1}}
\sum\limits_{s=1}^{n-1}\lim_{\varepsilon\to 0^{+}}
\int\limits_{\substack{y'\in{\mathbb{R}}^{n-1}\\ |x'-y'|>\varepsilon}}
\frac{(x_s-y_s)(x'-y',0)_\alpha(x'-y',0)_\gamma}{|x'-y'|^{n+2}}\,
(\partial_sf_\alpha)(y')\,dy',
\nonumber\end{align}
for each $\gamma\in\{1,\dots,n\}$ and a.e. $x'=(x_1,\dots,x_{n-1})\in{\mathbb{R}}^{n-1}$.
\end{enumerate}
\end{corollary}

We note that part {\it (i)} in Corollary~\ref{yreeCCC} corresponding to
$L=\Delta$ and $p=2$, but with ${\mathbb{R}}^n_{+}$ replaced by the domain
above the graph of a real-valued Lipschitz function defined in ${\mathbb{R}}^{n-1}$,
has been dealt with by B.~Dahlberg in \cite{Dah}. In the same geometric setting,
a result of a similar nature for the Clifford-Cauchy operator may be found in \cite[p.\,86]{LNM}.
Remarkably, the techniques used to establish Theorem~\ref{V-Naa.11} are robust enough
to adapt to the more general geometric setting just described as well as to higher order systems;
see Theorem~\ref{yenbcxgu} and Corollary~\ref{yenbcxgu.CCC} in this regard, generalizing
the aforementioned result of Dahlberg.

The layout of the remainder of the paper is as follows. In \S\ref{S-2} we collect a number
of analytical tools which are going to be useful for the goals we pursue here.
In \S\ref{S-3} we prove Theorem~\ref{V-Naa.11} and Corollary~\ref{yreeCCC} by relying
on Calder\'on-Zygmund theory for singular integral operators and square-function estimates.
The novel feature of the latter ingredient is the fact that the systems in question are
allowed to have complex coefficients, and no symmetry conditions are imposed (compare
with \cite{DKPV} where higher-order symmetric systems with real coefficients
in Lipschitz domains have been considered). Next, in \S\ref{S-4}, we discuss a refinement
of the vanishing conormal condition \eqref{eq:Asd.FFF} and prove a version of \eqref{mdiab}
emphasizing the role of the so-called conjugate kernel functions (in place of the fundamental solution).
In \S\ref{S-5} we extend the scope of Theorem~\ref{V-Naa.11} by providing
a description of arbitrary powers ${\mathbf{A}}^k$, with $k\in{\mathbb{N}}$,
of the infinitesimal generator of the Poisson semigroup from Theorem~\ref{VCXga}.
This task is accomplished in Theorem~\ref{V-Naa.11.HOR} and a significant feature of
this result is the somewhat surprising connection with a $k$-th order
Regularity problem for the system in question. Finally, in \S\ref{S-6} we treat the
case of higher order systems in graph Lipschitz domains.

\section{Analytical tools}
\setcounter{equation}{0}
\label{S-2}

Given a complex Banach space $X$, let
${\mathcal{L}}(X)$ stand for the space of all linear and bounded operators from $X$ into
itself, and denote by ${\mathscr{C}}\big([0,\infty),X\big)$ the space of all continuous
functions defined on $[0,\infty)$ with values in $X$. Then a $C_0$-{\tt semigroup} on $X$ is a
mapping $T:[0,\infty)\to{\mathcal{L}}(X)$ such that:
\begin{enumerate}
\item[(i)] $T(\cdot)x\in{\mathscr{C}}\big([0,\infty),X\big)$ for each $x\in X$;
\item[(ii)] $T(0)=I$, the identity on $X$;
\item[(iii)] $T(s+t)=T(s)T(t)$ for all $s,t\in[0,\infty)$.
\end{enumerate}
Given a $C_0$-semigroup $T$ on $X$, one defines the {\tt infinitesimal} {\tt generator}
${\mathbf{A}}$ of $T$ as an unbounded operator on $X$ by
\begin{equation}\label{eq:Dav}
{\mathbf{A}}x:=\lim_{t\to 0^{+}}\frac{T(t)x-x}{t}
\end{equation}
with domain
\begin{equation}\label{eq:Cabnm}
D({\mathbf{A}}):=\Big\{x\in X:\,\lim_{t\to 0^{+}}\frac{T(t)x-x}{t}\,\,\mbox{ exists in $X$}\Big\}.
\end{equation}
Then $D({\mathbf{A}})$ is dense in $X$, and $A$ is closed and linear.
Moreover,
\begin{equation}\label{eq:ADCV}
T(t)x=e^{t{\mathbf{A}}}x,\quad\forall\,t\in[0,\infty),\,\,\,
\forall\,x\in D({\mathbf{A}}),
\end{equation}
and
\begin{equation}\label{eq:ADCV.ttt}
\begin{array}{c}
\mbox{for each $t\in[0,\infty)$ the operator $T(t)$ maps $D({\mathbf{A}})$ into $D({\mathbf{A}})$,}
\\[6pt]
\mbox{and }\,\,\displaystyle
\frac{d}{dt}\big[T(t)x\big]={\mathbf{A}}T(t)x=T(t){\mathbf{A}}x,\quad\forall\,t\in(0,\infty),\,\,\,
\forall\,x\in D({\mathbf{A}}),
\end{array}
\end{equation}
where the derivative is taken in the strong sense on $X$. Recall that for each $k\in{\mathbb{N}}$,
\begin{align}\label{Dkdfidi}
D({\mathbf{A}}^k)&=\big\{x\in D({\mathbf{A}}^{k-1}):\,{\mathbf{A}}^{k-1}x\in D({\mathbf{A}})\big\}
\nonumber\\[4pt]
&=\big\{x\in D({\mathbf{A}}):\,{\mathbf{A}}x,\,{\mathbf{A}}^2x,\,{\mathbf{A}}^3x,
\dots, {\mathbf{A}}^{k-1}x\in D({\mathbf{A}}) \big\}.
\end{align}
For further use, let us point out that iterating \eqref{eq:ADCV.ttt}
shows that given any $k\in{\mathbb{N}}$,
\begin{equation}\label{eq:ADCV.ttt.79}
\begin{array}{c}
\mbox{for each $t\in(0,\infty)$ the operator $T(t)$ maps $D({\mathbf{A}}^k)$ into $D({\mathbf{A}}^k)$}
\\[6pt]
\mbox{and }\,\,\displaystyle
\frac{d^k}{dt^k}\big[T(t)x\big]={\mathbf{A}}^k T(t)x=T(t){\mathbf{A}}^k x
\,\,\mbox{ for all }\,\,x\in D({\mathbf{A}}^k),
\end{array}
\end{equation}
where the $k$-th order derivative in $t$ is taken in the strong sense on $X$.
For more on these matters, the interested reader is referred to, e.g.,
\cite{Arendt}, \cite{Ou05}, \cite{Pa83}, \cite{Yo}.

We continue by describing the main properties of a special fundamental solution
for constant, complex coefficient, homogeneous systems of arbitrary order.
A proof of the present formulation may be found in \cite[Theorem~11.1, pp.\,347-348]{DM}
and \cite[Theorem~7.54, pp.\,270-271]{DM} (cf.~also \cite{IMM} and the references therein).
Recall that $S^{n-1}$ is the unit sphere centered at the origin in ${\mathbb{R}}^n$,
$\sigma$ is its canonical surface measure, and $\omega_{n-1}=\sigma(S^{n-1})$
denotes its area. Also, throughout, we let `hat' denote the Fourier transform in ${\mathbb{R}}^n$,
given by $\widehat{f}(\xi)=\int_{{\mathbb{R}}^n}e^{-i\xi\cdot x}f(x)\,dx$ for $\xi\in{\mathbb{R}}^n$.

\begin{theorem}\label{FS-prop}
Fix $n,m,M\in\mathbb{N}$ with $n\geq 2$, and consider an $M\times M$
system of homogeneous differential operators of order $2m$,
\begin{equation}\label{op-Liii}
{\mathfrak{L}}:=\sum_{|\alpha|=2m}{A}_{\alpha}\partial^\alpha,
\end{equation}
with matrix coefficients ${A}_{\alpha}\in{\mathbb{C}}^{M\times M}$.
Assume that ${\mathfrak{L}}$ satisfies the weak ellipticity condition
\begin{equation}\label{LH.w}
{\rm det}\big[{\mathfrak{L}}(\xi)\big]\not=0,
\qquad\forall\,\xi\in{\mathbb{R}}^n\setminus\{0\},
\end{equation}
where
\begin{equation}\label{LH.2}
{\mathfrak{L}}(\xi):=\sum_{|\alpha|=2m}\xi^\alpha
{A}_{\alpha}\in{\mathbb{C}}^{M\times M},
\qquad\forall\,\xi\in{\mathbb{R}}^n.
\end{equation}
Then the $M\times M$ matrix $E$ defined at
each $x\in{\mathbb{R}}^{n}\setminus\{0\}$ by
\begin{equation}\label{Def-ES1-GLOB}
E(x):=\frac{1}{4(2\pi\,i)^{n-1}(2m-1)!}\Delta_x^{(n-1)/2}
\int_{S^{n-1}}(x\cdot\xi)^{2m-1}\,{\rm sgn}\,(x\cdot\xi)
\big[{\mathfrak{L}}(\xi)\big]^{-1}\,d\sigma(\xi)\quad
\end{equation}
if $n$ is odd, and
\begin{equation}\label{Def-ES2-GLOB}
E(x):=\frac{-1}{(2\pi\,i)^{n}(2m)!}\Delta_x^{n/2}\int_{S^{n-1}}
(x\cdot\xi)^{2m}\ln|x\cdot\xi|\big[{\mathfrak{L}}(\xi)\big]^{-1}\,d\sigma(\xi)
\end{equation}
if $n$ is even, satisfies the following properties.
\begin{enumerate}
\item[(1)] Each entry in $E$ is a tempered distribution in ${\mathbb{R}}^n$,
and a real-analytic function in $\mathbb{R}^n\setminus\{0\}$ {\rm (}hence, in particular,
it belongs to ${\mathscr{C}}^\infty(\mathbb{R}^n\setminus\{0\})${\rm )}. Moreover,
\begin{equation}\label{smmth-odd}
E(-x)=E(x)\quad\mbox{for all }\,\,x\in{\mathbb{R}}^n\setminus\{0\}.
\end{equation}
\item[(2)] If $I_{M\times M}$ is the $M\times M$ identity matrix, then for each $y\in{\mathbb{R}}^n$
\begin{equation}\label{fs-GLOB}
{\mathfrak{L}}_x\bigl[E(x-y)\bigr]=\delta_y(x)\,I_{M\times M}
\end{equation}
in the sense of tempered distributions in ${\mathbb{R}}^n$,
where the subscript $x$ denotes the fact that the operator ${\mathfrak{L}}$ in
\eqref{fs-GLOB} is applied to each column of $E$ in the variable $x$.
\item[(3)] Define the $M\times M$ matrix-valued function
\begin{equation}\label{Fm-PjkX}
{\mathcal{P}}(x):=\frac{-1}{(2\pi\,i)^n(2m-n)!}\int_{S^{n-1}}
(x\cdot\xi)^{2m-n}\big[{\mathfrak{L}}(\xi)\big]^{-1}\,d\sigma(\xi),\quad\forall\,x\in{\mathbb{R}}^n.
\end{equation}
Then the entries of ${\mathcal{P}}$ are identically zero when either $n$ is odd or $n>2m$,
and are homogeneous polynomials of degree $2m-n$ when $n\leq 2m$. Moreover,
there exists a ${\mathbb{C}}^{M\times M}$-valued function $\Phi$, with entries in
${\mathscr{C}}^\infty(\mathbb{R}^n\setminus\{0\})$, that
is positive homogeneous of degree $2m-n$ such that
\begin{equation}\label{fs-str}
E(x)=\Phi(x)+\bigl(\ln|x|\bigr){\mathcal{P}}(x),\qquad \forall\,x\in\mathbb{R}^n\setminus\{0\}.
\end{equation}
\item[(4)] For each $\beta\in\mathbb{N}_0^n$ with $|\beta|\geq 2m-1$,
the restriction to ${\mathbb{R}}^n\setminus\{0\}$ of the matrix distribution
$\partial^\beta E$ is of class ${\mathscr{C}}^\infty$ and positive homogeneous
of degree $2m-n-|\beta|$.
\item[(5)] For each $\beta\in\mathbb{N}_0^n$ there exists $C_\beta\in(0,\infty)$
such that the estimate
\begin{equation}\label{fs-est}
|\partial^\beta E(x)|\leq\left\{
\begin{array}{l}
\displaystyle\frac{C_\beta}{|x|^{n-2m+|\beta|}}
\quad\mbox{ if either $n$ is odd, or $n>2m$, or if $|\beta|>2m-n$},
\\[16pt]
\displaystyle\frac{C_{\beta}(1+|\ln|x||)}{|x|^{n-2m+|\beta|}}
\quad\mbox{ if }\,\,0\leq |\beta|\leq 2m-n,
\end{array}\right.
\end{equation}
holds for each $x\in{\mathbb{R}}^n\setminus\{0\}$.
\item[(6)] When restricted to ${\mathbb{R}}^n\setminus\{0\}$, the entries of
$\widehat{E}$ are ${\mathscr{C}}^\infty$ functions and, moreover,
\begin{equation}\label{E-ftXC}
\widehat{E}(\xi)=(-1)^m\bigl[{\mathfrak{L}}(\xi)\bigr]^{-1}
\quad\mbox{for each}\quad\xi\in{\mathbb{R}}^n\setminus\{0\}.
\end{equation}
\item[(7)] Writing $E^{\mathfrak{L}}$ in place of $E$ to emphasize
the dependence on ${\mathfrak{L}}$, the fundamental solution $E^{\mathfrak{L}}$ with
entries as in \eqref{Def-ES1-GLOB}-\eqref{Def-ES2-GLOB} satisfies
\begin{equation}\label{E-Trans}
\begin{array}{c}
\bigl(E^{\mathfrak{L}}\bigr)^\top=E^{{\mathfrak{L}}^\top},\quad
\overline{E^{\mathfrak{L}}}=E^{\overline{{\mathfrak{L}}}}\,,\quad
\bigl(E^{\mathfrak{L}}\bigr)^\ast=E^{{\mathfrak{L}}^\ast},
\\[8pt]
\mbox{and}\quad E^{\lambda{\mathfrak{L}}}=\lambda^{-1} E^{\mathfrak{L}}
\,\,\mbox{ for each }\,\,\lambda\in{\mathbb{C}}\setminus\{0\},
\end{array}
\end{equation}
where ${\mathfrak{L}}^\top$, $\overline{{\mathfrak{L}}}$, and
${\mathfrak{L}}^\ast=\overline{{\mathfrak{L}}}^\top$ denote the
transposed, the complex conjugate, and the Hermitian adjoint of ${\mathfrak{L}}$,
respectively.
\item[(8)] Any fundamental solution ${\mathbb{E}}$ of the system ${\mathfrak{L}}$ in ${\mathbb{R}}^n$,
whose entries are tempered distributions in ${\mathbb{R}}^n$, is of the form
${\mathbb{E}}=E+Q$ where $E$ is as in \eqref{Def-ES1-GLOB}-\eqref{Def-ES2-GLOB} and
$Q$ is an $M\times M$ matrix whose entries are polynomials in ${\mathbb{R}}^n$ and whose columns,
$Q_k$, $k\in\{1,\dots,M\}$, satisfy the pointwise equations
${\mathfrak{L}}\,Q_k=0\in{\mathbb{C}}^M$ in ${\mathbb{R}}^n$ for each $k\in\{1,\dots,M\}$.
\item[(9)] In the particular case when $M=1$ and $m=1$, i.e., in the situation when
${\mathfrak{L}}={\rm div}{A}\nabla$ for some matrix
${A}=(a_{jk})_{1\leq j,k\leq n}\in{\mathbb{C}}^{n\times n}$,
and when in place of \eqref{LH.w} the strong ellipticity condition
\begin{equation}\label{sec-or-a}
{\rm Re}\Bigg[\sum\limits_{j,k=1}^n a_{jk}\xi_j\xi_k\Bigg]\geq C|\xi|^2,
\qquad\forall\,\xi=(\xi_1,\dots,\xi_n)\in{\mathbb{R}}^n,
\end{equation}
is imposed, the fundamental solution $E$ of ${\mathfrak{L}}$ from
\eqref{Def-ES1-GLOB}-\eqref{Def-ES2-GLOB} takes the explicit form
\begin{equation}\label{YTcxb-ytSH}
E(x)=\left\{
\begin{array}{ll}
-\dfrac{1}{(n-2)\omega_{n-1}\sqrt{{\rm det}\,({A}_{\rm sym})}}
\Big[\big(({A}_{\rm sym})^{-1}x\big)\cdot x\Big]^{\frac{2-n}{2}}
& \mbox{ if }\,\,n\geq 3,
\\[18pt]
\dfrac{1}{4\pi\sqrt{{\rm det}\,({A}_{\rm sym})}}
\log\Big[\big(({A}_{\rm sym})^{-1}x\big)\cdot x\Big] & \mbox{ if }\,\,n=2.
\end{array}
\right.
\end{equation}
Here, ${A}_{\rm sym}:=\frac{1}{2}({A}+{A}^\top)$
stands for the symmetric part of
the coefficient matrix ${A}=(a_{rs})_{1\leq r,s\leq n}$
and $\log$ denotes the principal branch of the complex logarithm function
{\rm (}defined by the requirement that $z^t=e^{t\log z}$ for all
$z\in{\mathbb{C}}\setminus(-\infty,0]$ and all $t\in{\mathbb{R}}${\rm )}.
\end{enumerate}
\end{theorem}

Changing topics, Poisson kernels for elliptic operators in a half-space have
been studied in, e.g., \cite{ADNI}, \cite{ADNII},  \cite{Sol1}, \cite{Sol2}.
Here we record the following useful existence and uniqueness result.

\begin{theorem}\label{kkjbhV}
Let $L$ be an $M\times M$ system with constant complex coefficients as in
\eqref{L-def}-\eqref{L-ell.X}. Then there exists a matrix-valued function
$P^L=\big(P^L_{\alpha\beta}\big)_{1\leq\alpha,\beta\leq M}:
\mathbb{R}^{n-1}\to\mathbb{C}^{M\times M}$
{\rm (}called the Poisson kernel for $L$ in $\mathbb{R}^{n}_+${\rm )}
satisfying the following properties:
\begin{list}{$(\theenumi)$}{\usecounter{enumi}\leftmargin=.8cm
\labelwidth=.8cm\itemsep=0.2cm\topsep=.1cm
\renewcommand{\theenumi}{\alph{enumi}}}
\item[(1)] There exists $C\in(0,\infty)$ such that
\begin{equation}\label{eq:IG6gy}
|P^L(x')|\leq\frac{C}{(1+|x'|^2)^{\frac{n}2}}\quad\mbox{for each }\,\,x'\in\mathbb{R}^{n-1}.
\end{equation}
\item[(2)] The function $P^L$ is Lebesgue measurable and
\begin{equation}\label{eq:IG6gy.2}
\int_{\mathbb{R}^{n-1}}P^L(x')\,dx'=I_{M\times M},
\end{equation}
the $M\times M$ identity matrix.
\item[(3)] If one sets
\begin{equation}\label{eq:Gvav7g5}
K^L(x',t):=P^L_t(x'):=t^{1-n}P^L(x'/t)\quad\mbox{for each }\,\,x'\in\mathbb{R}^{n-1}
\,\,\,\mbox{ and }\,\,t>0,
\end{equation}
then the function $K^L=\big(K^L_{\alpha\beta}\big)_{1\leq\alpha,\beta\leq M}$
satisfies {\rm (}in the sense of distributions{\rm )}
\begin{equation}\label{uahgab-UBVCX}
LK^L_{\cdot\beta}=0\,\,\mbox{ in }\,\,\mathbb{R}^{n}_{+}
\,\,\mbox{ for each }\,\,\beta\in\{1,\dots,M\},
\end{equation}
where $K^L_{\cdot\beta}:=\big(K^L_{\alpha\beta}\big)_{1\leq\alpha\leq M}$.
\end{list}

Moreover, $P^L$ is unique in the class of $\mathbb{C}^{M\times M}$-valued
functions satisfying (1)-(3) above, and has the following additional properties:

\begin{list}{$(\theenumi)$}{\usecounter{enumi}\leftmargin=.8cm
\labelwidth=.8cm\itemsep=0.2cm\topsep=.1cm
\renewcommand{\theenumi}{\alph{enumi}}}
\item[(4)] One has $P^L\in{\mathscr{C}}^\infty(\mathbb{R}^{n-1})$ and
$K^L\in{\mathscr{C}}^\infty\big(\overline{{\mathbb{R}}^n_{+}}\setminus B(0,\varepsilon)\big)$
for every $\varepsilon>0$. Consequently, \eqref{uahgab-UBVCX}
holds in a pointwise sense.
\item[(5)] There holds $K^L(\lambda x)=\lambda^{1-n}K^L(x)$ for all $x\in{\mathbb{R}}^n_{+}$
and $\lambda>0$. Hence, in particular, for each multi-index $\alpha\in{\mathbb{N}}_0^n$
there exists $C_\alpha\in(0,\infty)$ with the property that
\begin{equation}\label{eq:Kest}
\big|(\partial^\alpha K^L)(x)\big|\leq C_\alpha\,|x|^{1-n-|\alpha|},\qquad
\forall\,x\in{\overline{{\mathbb{R}}^n_{+}}}\setminus\{0\}.
\end{equation}
\item[(6)] For each $\kappa>0$ there exists a finite constant $C_\kappa>0$ with the property that
for each $x'\in\mathbb{R}^{n-1}$,
\begin{equation}\label{exTGFVC}
\sup_{|x'-y'|<\kappa t}\big|(P^L_t\ast f)(y')\big|\leq C_\kappa\,\mathcal{M}f(x'),
\qquad\forall\,f\in L^1\Big({\mathbb{R}}^{n-1}\,,\,\frac{1}{1+|x'|^n}\,dx'\Big),
\end{equation}
where $\mathcal{M}$ is the Hardy-Littlewood maximal operator in ${\mathbb{R}}^{n-1}$.
\item[(7)] Given any $\kappa>0$ and any function
\begin{equation}\label{eq:aaAa}
f=(f_\beta)_{1\leq\beta\leq M}\in L^1\Big({\mathbb{R}}^{n-1}\,,\,\frac{1}{1+|x'|^n}dx'\Big)
\subset L^1_{\rm loc}({\mathbb{R}}^{n-1}),
\end{equation}
if $u(x',t):=(P^L_t\ast f)(x')$ for each $(x',t)\in{\mathbb{R}}^n_{+}$, then
\begin{equation}\label{exist:u2}
u\in\mathscr{C}^\infty(\mathbb{R}^n_{+}),\qquad
Lu=0\,\,\mbox{ in }\,\,\mathbb{R}^{n}_{+},
\end{equation}
and, at every Lebesgue point $x'_0\in{\mathbb{R}}^{n-1}$ of $f$,
\begin{equation}\label{exTGFVC.2s}
\Big(u\big|^{{}^{\rm n.t.}}_{\partial{\mathbb{R}}^{n}_{+}}\Big)(x'_0):=
\lim_{\substack{(x',\,t)\to(x'_0,0)\\ |x'-x'_0|<\kappa t}}(P^L_t\ast f)(x')=f(x'_0).
\end{equation}
\item[(8)] The function $P^L$ satisfies the semigroup property
\begin{equation}\label{eq:re4fd}
P^L_{t_1}\ast P^L_{t_2}=P^L_{t_1+t_2}\,\,\,\mbox{ for every }\,\,t_1,t_2>0.
\end{equation}
\end{list}
\end{theorem}

\noindent Concerning Theorem~\ref{kkjbhV}, we note that the existence part follows from
the classical work of S.\,Agmon, A.\,Douglis, and L.\,Nirenberg in \cite{ADNII}.
The uniqueness property has been recently proved in \cite{K-MMMM},
where \eqref{exTGFVC}, \eqref{exist:u2}, \eqref{exTGFVC.2s}, as well as the semigroup property \eqref{eq:re4fd} have also been established.
\vskip 0.08in

The following is a consequence of the classical Calder\'on-Zygmund theory of
singular integral operators (a result of this flavor in a much more general
geometric setting has been proved in \cite{HoMiTa10}). See also
\cite[Corollary~4.78, p.159]{DM} for the setting when the operator acts on
Schwartz functions. Throughout, let $\{{\bf e}_j\}_{1\leq j\leq n}$ denote the
standard orthonormal basis in ${\mathbb{R}}^n$; hence, in particular,
${\bf e}_n:=(0,\dots,0,1)\in{\mathbb{R}}^n$.

\begin{theorem}\label{Main-T2}
There exists a positive integer $N=N(n)$ with the following significance.
Consider a function
\begin{eqnarray}\label{ker}
\begin{array}{l}
K\in {\mathscr{C}}^N({\mathbb{R}}^{n}\setminus\{0\})\quad\mbox{with}\quad
K(-x)=-K(x)\quad\mbox{and}\quad
\\[4pt]
K(\lambda\,x)=\lambda^{-(n-1)}K(x)\,\,\,\,\,\,\,\,
\forall\,\lambda>0,\,\,\,\forall\,x\in{\mathbb{R}}^{n}\setminus\{0\},
\end{array}
\end{eqnarray}
and define the singular integral operator
\begin{eqnarray}\label{T-layer}
{\mathcal{T}}f(x):=\int_{{\mathbb{R}}^{n-1}}K\big(x-(y',0)\big)f(y')\,dy',
\qquad x\in{\mathbb{R}}^n_{+},
\end{eqnarray}
along with
\begin{eqnarray}\label{T-pv}
&& T_*f(x'):=\sup_{\varepsilon>0}|T_{\varepsilon}f(x')|,
\qquad x'\in{\mathbb{R}}^{n-1},
\quad\mbox{where}
\\[8pt]
&& T_{\varepsilon}f(x'):=
\int\limits_{\substack{y'\in{\mathbb{R}}^{n-1}\\ |x'-y'|>\varepsilon}}
K(x'-y',0)f(y')\,dy',\qquad x'\in{\mathbb{R}}^{n-1}.
\label{pv-layer2}
\end{eqnarray}
Then for each $p\in(1,\infty)$ there exists a constant
$C\in(0,\infty)$ depending only on $n$ and $p$ such that
\begin{eqnarray}\label{Tmax-bdd}
\|T_*f\|_{L^p({\mathbb{R}}^{n-1})}
\leq C\|K|_{S^{n-1}}\|_{{\mathscr{C}}^N(S^{n-1})}\|f\|_{L^p({\mathbb{R}}^{n-1})},
\qquad\forall\,f\in L^p({\mathbb{R}}^{n-1}).
\end{eqnarray}
Furthermore, for each $f\in L^p({\mathbb{R}}^{n-1})$, where $p\in [1,\infty)$, the limit
\begin{eqnarray}\label{main-lim}
Tf(x'):=\lim_{\varepsilon\to 0^+}T_{\varepsilon}f(x')
\end{eqnarray}
exists for a.e. $x'\in{\mathbb{R}}^{n-1}$ and the induced operator
\begin{eqnarray}\label{maKP}
T:L^p({\mathbb{R}}^{n-1})\longrightarrow L^p({\mathbb{R}}^{n-1}),
\quad p\in(1,\infty),
\end{eqnarray}
is well-defined, linear and bounded. In addition, the following Cotlar inequality holds:
\begin{eqnarray}\label{T-Har.YFDS}
{\mathcal{N}}({\mathcal{T}}f)(x')\leq C{\mathcal{M}}(T_\ast f)(x')+C{\mathcal{M}}f(x'),
\qquad\forall\,x'\in{\mathbb{R}}^{n-1}.
\end{eqnarray}
In particular, for each $p\in(1,\infty)$ there exists a finite constant
$C=C(n,p)>0$ such that
\begin{eqnarray}\label{T-Har}
\|{\mathcal{N}}({\mathcal{T}}f)\|_{L^p({\mathbb{R}}^{n-1})}
\leq C\|K|_{S^{n-1}}\|_{{\mathscr{C}}^N(S^{n-1})}\|f\|_{L^p({\mathbb{R}}^{n-1})}.
\end{eqnarray}
Finally, whenever $f\in L^p({\mathbb{R}}^{n-1})$ with $1\leq p<\infty$, the jump-formula
\begin{eqnarray}\label{main-jump}
\Big({\mathcal{T}}f\big|^{{}^{\rm n.t.}}_{\partial{\mathbb{R}}^n_{+}}\Big)(x')
=\tfrac{1}{2i}\,\widehat{K}(-{\bf e}_n)f(x')+Tf(x')
\end{eqnarray}
is valid at a.e. $x'\in{\mathbb{R}}^{n-1}$.
\end{theorem}

Recall next that, having fixed some $\kappa>0$, the area-function ${\mathcal{A}}_{\kappa}$
acts on a measurable function $u$ defined in ${\mathbb{R}}^n_{+}$ according to
\begin{eqnarray}\label{ga.12-hat}
({\mathcal{A}}_{\kappa}u)(x'):=\left(\int_{\Gamma_{\kappa}(x')}
|u(y',t)|^2\,\frac{dy'\,dt}{t^{n-2}}\right)^{1/2},
\qquad\forall\,x'\in{\mathbb{R}}^{n-1}.
\end{eqnarray}
The following square function estimate is a particular case of more general results of this nature
proved in \cite{HMMM} (for related work which is relevant to the present aims see also \cite{LNM},
\cite{M95}).

\begin{theorem}\label{prop:SFE}
Assume that $\theta:\mathbb{R}^{n}_+\times\mathbb{R}^{n-1}\rightarrow\mathbb{C}$
is a measurable function such that there exists a finite constant $c>0$ with the
property that, for all $(x',t)\in\mathbb{R}^{n}_{+}$ and $y'\in\mathbb{R}^{n-1}$,
\begin{equation}\label{SFE-est-theta}
|\theta(x',t;y')|+|(x'-y',t)|\,|\nabla_{y'}\theta(x',t;y')|
\leq\frac{c}{|(x'-y',t)|^n},
\end{equation}
and such that, for each $(x',t)\in\mathbb{R}^{n}_{+}$,
\begin{equation}\label{SFE-vanish-theta}
\int_{\mathbb{R}^{n-1}}\theta(x',t;y')\,dy'=0.
\end{equation}
Consider the integral operator
\begin{equation}\label{defi:Theta}
(\Theta f)(x',t):=\int_{\mathbb{R}^{n-1}}\theta(x',t;y')\,f(y')dy',
\qquad\forall\,(x',t)\in\mathbb{R}^{n}_{+}.
\end{equation}
Also, fix $\kappa>0$ and recall the area function from \eqref{ga.12-hat}.

Then for every $p\in(1,\infty)$ there exists a finite constant
$C=C(n,\kappa,p,c)>0$ with the property that
\begin{equation}\label{uagv.SS}
\|{\mathcal{A}}_\kappa(\Theta f)\|_{L^p(\mathbb{R}^{n-1})}\leq C\|f\|_{L^p(\mathbb{R}^{n-1})},
\qquad\forall\,f\in L^p(\mathbb{R}^{n-1}).
\end{equation}
\end{theorem}

Later, we shall need a suitable version of the divergence theorem,
recently obtained in \cite{Div-MMM}. To state it, for each
$k\in{\mathbb{N}}$ denote by ${\mathscr{L}}^{k}$ the $k$-dimensional
Lebesgue measure in ${\mathbb{R}}^k$.

\begin{theorem}[\cite{Div-MMM}]\label{theor:div-thm}
Assume that $\vec{F}\in L^1_{\rm loc}({\mathbb{R}}^n_{+},\mathbb{C}^n)$ is a vector
field satisfying the following conditions {\rm (}for some fixed $\kappa>0$,
and with the divergence taken in the sense of distributions{\rm )}:
\begin{equation}\label{ytead}
{\rm div}\,\vec{F}\in L^1({\mathbb{R}}^n_{+}),\quad
{\mathcal{N}}_\kappa\vec{F}\in L^1({\mathbb{R}}^{n-1}),\quad
\mbox{ there exists }\,\,\vec{F}\big|_{\partial{\mathbb{R}}^n_{+}}^{{}^{\rm n.t.}}
\,\,\mbox{ a.e.~in }\,\,{\mathbb{R}}^{n-1}.
\end{equation}
Then
\begin{equation}\label{eqn:div-form}
\int_{{\mathbb{R}}^n_{+}}{\rm div}\,\vec{F}\,d{\mathscr{L}}^{n}
=-\int_{{\mathbb{R}}^{n-1}}{\bf e}_n\cdot
\bigl(\vec F\,\big|^{{}^{\rm n.t.}}_{\partial{\mathbb{R}}^n_{+}}\bigr)\,d{\mathscr{L}}^{n-1}.
\end{equation}
\end{theorem}

We shall now record the following versatile version of interior estimates for
higher-order elliptic systems. A proof may be found in \cite[Theorem~11.9, p.\,364]{DM}.

\begin{theorem}\label{ker-sbav}
Fix $n,m,M\in\mathbb{N}$ with $n\geq 2$, and consider an $M\times M$
system of homogeneous differential operators of order $2m$,
${\mathfrak{L}}:=\sum_{|\alpha|=2m}{A}_{\alpha}\partial^\alpha$,
with matrix coefficients ${A}_{\alpha}\in{\mathbb{C}}^{M\times M}$,
satisfying the weak ellipticity condition \eqref{LH.w}.
Then for each null-solution $u$ of ${\mathfrak{L}}$ in a ball $B(x,R)$
{\rm (}where $x\in{\mathbb{R}}^n$ and $R>0${\rm )}, $0<p<\infty$, $\lambda\in(0,1)$,
$\ell\in{\mathbb{N}}_0$, and $0<r<R$, one has
\begin{equation}\label{detraz}
\sup_{z\in B(x,\lambda r)}|\nabla^\ell u(z)|
\leq\frac{C}{r^\ell}\left(\aver{B(x,r)}|u|^p\,d{\mathscr{L}}^n\right)^{1/p},
\end{equation}
where $C=C(L,p,\ell,\lambda,n)>0$ is a finite constant, and
$\nabla^\ell u$ denotes the vector with the components
$(\partial^\alpha u)_{|\alpha|=\ell}$.
\end{theorem}

We next state the following Fatou type result from \cite{K-MMMM}.

\begin{theorem}\label{tuFatou}
Let $L$ be a system with complex coefficients as in
\eqref{L-def}-\eqref{L-ell.X}. Then for each $p\in(1,\infty)$, $\kappa>0$, and
$u\in{\mathscr{C}}^\infty({\mathbb{R}}^n_{+})$,
\begin{eqnarray}\label{Tafva}
\left.
\begin{array}{r}
Lu=0\,\mbox{ in }\,{\mathbb{R}}^n_{+}
\\[6pt]
{\mathcal{N}}_{\kappa}u\in L^p({\mathbb{R}}^{n-1})
\end{array}
\right\}
\Longrightarrow
\left\{
\begin{array}{l}
u\big|^{{}^{\rm n.t.}}_{\partial{\mathbb{R}}^n_{+}}\,\mbox{ exists a.e.~in }\,
{\mathbb{R}}^{n-1},\, \mbox{ belongs to $L^p({\mathbb{R}}^{n-1})$,}
\\[12pt]
\mbox{and }\,u(x',t)=\Big(P^L_t\ast\big(u\big|^{{}^{\rm n.t.}}_{\partial{\mathbb{R}}^n_{+}}\big)\Big)(x'),\,\,
\forall\,(x',t)\in{\mathbb{R}}^n_{+},
\end{array}
\right.
\end{eqnarray}
where $P^L$ is the Poisson kernel for $L$ in ${\mathbb{R}}^n_{+}$.
\end{theorem}

Finally, it is possible to give sufficient conditions guaranteeing the validity of \eqref{eq:Asd.FFF}
formulated exclusively in terms of (the inverse of) the symbol of the system $L$.
Specifically, the strong ellipticity condition \eqref{L-ell.X} implies that
the symbol of $L$, given by ${\rm Symb}_L(\xi):=
-\Bigl(\xi_r\xi_s a^{\alpha\beta}_{rs}\Bigr)_{1\leq\alpha,\beta\leq M}$
for $\xi\in{\mathbb{R}}^n\setminus\{0\}$, satisfies the strict negativity condition
\begin{equation}\label{L-ell.Xfff}
{\rm Re}\,\Bigl\langle -\,{\rm Symb}_L(\xi)\eta,\eta\Big\rangle_{{\mathbb{C}}^M}
\geq\kappa_o|\xi|^2|\eta|^2\,\,\mbox{ for every }\,\,
\xi\in{\mathbb{R}}^n\,\,\mbox{ and }\,\,\eta\in{\mathbb{C}}^M.
\end{equation}
In particular, the inverse
\begin{equation}\label{yfg-yL4f}
\bigl(S_{\gamma\beta}(\xi)\bigr)_{1\leq\gamma,\beta\leq M}
:=\Bigl[{\rm Symb}_L(\xi)\Bigr]^{-1}\in{\mathbb{C}}^{\,M\times M}\,\,\,
\mbox{ exists for all }\,\,\xi\in{\mathbb{R}}^n\setminus\{0\}.
\end{equation}
We then have the following result.

\begin{proposition}[\cite{MaMiMi}]\label{UaalpIKL-PP}
Let $L$ be a strongly elliptic, second-order, homogeneous, $M\times M$ system,
with constant complex coefficients, and recall \eqref{yfg-yL4f}.
Assume that for each indices $s,s'\in\{1,...,n\}$
and $\alpha,\gamma\in\{1,...,M\}$ there holds
\begin{equation}\label{Ea4-fCii-n3}
\Bigl[a^{\beta\alpha}_{s's}-a^{\beta\alpha}_{ss'}
+\xi_r a^{\beta\alpha}_{rs}\partial_{\xi_{s'}}
-\xi_r a^{\beta\alpha}_{rs'}\partial_{\xi_{s}}\Bigr]
S_{\gamma\beta}(\xi)=0,\qquad\forall\,\xi\in{\mathbb{R}}^n\setminus\{0\},
\end{equation}
as well as {\rm (}with $\sigma$ denoting the arc-length
measure on the unit circle $S^1${\rm )}
\begin{equation}\label{Ea4-fCii-n2B}
\int_{S^1}\Big(a^{\beta\alpha}_{rs}\xi_{s'}
-a^{\beta\alpha}_{rs'}\xi_{s}\Big)
\big(\xi_r S_{\gamma\beta}(\xi)\big)\,d\sigma(\xi)=0\,\,\mbox{ if }\,\,n=2.
\end{equation}

Then, with $E=\big(E_{\gamma\beta}\big)_{1\leq\gamma,\beta\leq M}$ denoting
the fundamental solution for the system $L$ from Theorem~\ref{FS-prop},
\begin{equation}\label{eq:Asd.FFF.2}
\begin{array}{l}
a^{\beta\alpha}_{rn}\big(\partial_r E_{\gamma\beta}\big)(x',0)=0
\,\,\mbox{ for each }\,\,x'\in{\mathbb{R}}^{n-1}\setminus\{0\},
\\[6pt]
\mbox{for every fixed multi-indices }\,\,\alpha,\gamma\in\{1,\dots,M\}.
\end{array}
\end{equation}
\end{proposition}

\section{Proof of the main results}
\setcounter{equation}{0}
\label{S-3}

We start by establishing an explicit link between
the Poisson kernel and the fundamental solution for systems satisfying \eqref{eq:Asd.FFF}.

\begin{proposition}\label{uniq:double->!Poisson}
Suppose $L$ is a strongly elliptic, second-order, homogeneous, $M\times M$ system,
with constant complex coefficients, and let
$E=\big(E_{\gamma\beta}\big)_{1\leq\gamma,\beta\leq M}$ be the
fundamental solution for $L$ from Theorem~\ref{FS-prop}. Assume
that the vanishing conormal condition \eqref{eq:Asd.FFF} is satisfied.

Then the Poisson kernel for $L$ in $\mathbb{R}^{n}_+$ given in Theorem~\ref{kkjbhV}
is the matrix-valued function $P^L=(P^L_{\gamma\alpha})_{1\leq\gamma,\alpha\leq M}$
with entries given by
\begin{align}\label{yncuds.AAA}
P^L_{\gamma\alpha}(x')=2a^{\beta\alpha}_{rn}(\partial_r E_{\gamma\beta})(x',1),
\qquad\forall\,x'\in{\mathbb{R}}^{n-1},
\end{align}
for each $\gamma,\alpha\in\{1,...,M\}$. As a consequence,
\begin{align}\label{yfg-yLUU}
K^L_{\gamma\alpha}(x',t):=t^{1-n}P^L_{\gamma\alpha}(x'/t)
=2 a^{\beta\alpha}_{rn}(\partial_r E_{\gamma\beta})(x',t),
\qquad\forall\,(x',t)\in{\mathbb{R}}^n_{+},
\end{align}
for each $\gamma,\alpha\in\{1,...,M\}$.
\end{proposition}

\begin{proof}
The task is to show that the matrix-valued function
$P=(P_{\gamma\alpha})_{1\leq\gamma,\alpha\leq M}$ given by
\begin{align}\label{yncuds}
P_{\gamma\alpha}(x'):=2a^{\beta\alpha}_{rn}(\partial_r E_{\gamma\beta})(x',1),
\qquad\forall\,x'\in{\mathbb{R}}^{n-1},
\end{align}
for each $\gamma,\alpha\in\{1,...,M\}$,
coincides with the Poisson kernel for $L$ in $\mathbb{R}^{n}_+$ given in Theorem~\ref{kkjbhV}.
To this end, fix $\gamma,\alpha\in\{1,...,M\}$ arbitrary.
Then, relying on \eqref{yncuds}, \eqref{eq:Asd.FFF}, and the Fundamental
Theorem of Calculus, for every $x'\in\mathbb{R}^{n-1}$ such that $|x'|\geq 1$ we may estimate
\begin{align}\label{yncuds.TT}
\big|P_{\gamma\alpha}(x')\big| &=\Bigg|2a^{\beta\alpha}_{rn}\int_0^1
(\partial_n\partial_r E_{\gamma\beta})(x',t)\,dt\Bigg|
\nonumber\\[4pt]
&\leq C\int_0^1\big|(\nabla^2 E)(x',t)\big|\,dt
\leq C\int_0^1\big|(x',t)\big|^{-n}\,dt
\nonumber\\[4pt]
&\leq C|x'|^{-n}\leq\frac{C}{(1+|x'|^2)^{\frac{n}{2}}},
\end{align}
using 
{\it (5)} in Theorem~\ref{FS-prop}.
Since from \eqref{yncuds} and {\it (1)} in Theorem~\ref{FS-prop}
we have $P\in{\mathscr{C}}^\infty({\mathbb{R}}^{n-1})$,
estimate \eqref{yncuds.TT} eventually shows that $P$ satisfies property {\it (1)}
in Theorem~\ref{kkjbhV}.

Next, by once again relying on the homogeneity property from part {\it (4)}
in Theorem~\ref{FS-prop}, for each $(x',t)\in{\mathbb{R}}^n_{+}$, we may write
\begin{align}\label{yfg-yLLL.BN13}
K_{\gamma\alpha}(x',t)&:=t^{1-n}P_{\gamma\alpha}(x'/t)
=2t^{1-n} a^{\beta\alpha}_{rn}(\partial_r E_{\gamma\beta})(x'/t,t/t)
\nonumber\\[4pt]
&\,\,=2 a^{\beta\alpha}_{rn}(\partial_r E_{\gamma\beta})(x',t).
\end{align}
Thanks to this and part {\it (2)} in Theorem~\ref{FS-prop}, we have
\begin{eqnarray}\label{yfg-yLLL.BN14}
\big(LK_{\cdot\alpha}\big)_{\eta}
=2 a^{\eta\mu}_{kj}a^{\beta\alpha}_{rn}\partial_k\partial_j\partial_r E_{\mu\beta}
=2 a^{\beta\alpha}_{rn}\partial_r (L E_{\cdot\beta})_\eta=0
\,\,\mbox{ in }\,\,{\mathbb{R}}^n_{+},\quad\eta\in\{1,...,M\},
\end{eqnarray}
where $K_{\cdot\alpha}$ and $E_{\cdot\beta}$ denote the column vectors
$(K_{\mu\alpha})_{1\leq\mu\leq M}$ and $(E_{\mu\beta})_{1\leq\mu\leq M}$, respectively.
This proves that $P$ satisfies property {\it (3)} in Theorem~\ref{kkjbhV}.

To proceed, fix $f=(f_\alpha)_{1\leq\alpha\leq M}\in L^p({\mathbb{R}}^{n-1})$,
$1<p<\infty$, and consider the vector-valued function defined at each point
$(x',t)\in{\mathbb{R}}^n_{+}$ by the formula
\begin{align}\label{yfg-yLLL.BN5}
u(x',t) &:=(P_t\ast f)(x')=\int_{\mathbb{R}^{n-1}}
t^{1-n}\,P\Big(\frac{x'-y'}{t}\Big)\,f(y')\,dy'
\nonumber\\[4pt]
&\,\,=\Bigl(\int_{\mathbb{R}^{n-1}}
K_{\gamma\alpha}(x'-y',t)\,f_\alpha(y')\,dy'\Bigr)_{1\leq\gamma\leq M},
\end{align}
where the last equality uses \eqref{yfg-yLLL.BN13}. Hence, if for each
$\gamma,\alpha\in\{1,...,M\}$ we now introduce the integral operator
${\mathcal{T}}_{\gamma\alpha}$ as in \eqref{T-layer} associated to the kernel
$K_{\gamma\alpha}$ from \eqref{yfg-yLLL.BN13}, then \eqref{yfg-yLLL.BN5} may
be simply re-written as
\begin{eqnarray}\label{yfg-yLLL.BN7}
u=\big({\mathcal{T}}_{\gamma\alpha}f_\alpha\big)_{1\leq\gamma\leq M}
\,\,\,\mbox{ in }\,\,\,{\mathbb{R}}^n_{+}.
\end{eqnarray}
Since the kernels $K_{\gamma\alpha}$ satisfy \eqref{ker} (as seen from
\eqref{yfg-yLLL.BN13} and the properties of $E$),
Theorem~\ref{Main-T2} gives that at a.e. $x'\in{\mathbb{R}}^{n-1}$ we have
\begin{eqnarray}\label{yfg-yLLL.BN8}
\Bigl(u\Big|_{\partial{\mathbb{R}}^n_{+}}^{{}n.t.}\Big)(x')
=\Bigl(\tfrac{1}{2i}\,\widehat{K_{\gamma\alpha}}(-{\bf e}_n)f_{\alpha}(x')
+T_{\gamma\alpha}f_{\alpha}(x')\Big)_{1\leq\gamma\leq M},
\end{eqnarray}
where each $T_{\gamma\alpha}$ is associated with $K_{\gamma\alpha}$ as in
\eqref{main-lim}. However, in our case for each $x'\in{\mathbb{R}}^{n-1}$ we have
\begin{eqnarray}\label{yfg-yLLL.BN9}
T_{\gamma\alpha}f_{\alpha}(x')=\lim_{\varepsilon\to 0^{+}}
\int\limits_{\substack{y'\in{\mathbb{R}}^{n-1}\\ |x'-y'|>\varepsilon}}
K_{\gamma\alpha}(x'-y',0)f_{\alpha}(y')\,dy'=0,\qquad\forall\,\gamma\in\{1,...,M\},
\end{eqnarray}
by \eqref{yfg-yLLL.BN13} and \eqref{eq:Asd.FFF}. Moreover, from \eqref{yfg-yLLL.BN13}
and \eqref{E-ftXC} we deduce that if
\begin{equation}\label{eq:Bbvc9h7}
B:=\big(a^{\alpha\beta}_{nn}\big)_{1\leq\alpha,\beta\leq M}\in{\mathbb{C}}^{M\times M}
\end{equation}
then
\begin{eqnarray}\label{yfg-yLLL.BN10}
\widehat{K_{\gamma\alpha}}(-{\bf e}_n)=2i B_{\beta\alpha}\big(B^{-1}\big)_{\gamma\beta}
=2i\delta_{\gamma\alpha}.
\end{eqnarray}
In concert, \eqref{yfg-yLLL.BN8}-\eqref{yfg-yLLL.BN10} imply that
\begin{eqnarray}\label{yfg-yLLL.BN8X}
u\big|_{\partial{\mathbb{R}}^n_{+}}^{{}^{\rm n.t.}}=f
\,\,\,\mbox{ a.e. in }\,\,\,{\mathbb{R}}^{n-1}.
\end{eqnarray}
In light of \eqref{yfg-yLLL.BN5} this shows that
for a.e. $x'\in\mathbb{R}^{n-1}$ we have
\begin{eqnarray}\label{yfg-yLLL.BN11}
\lim_{t\to 0^+}(P_t\ast f)(x')=f(x').
\end{eqnarray}
Fix $j\in\{1,\dots,n\}$ and consider the function $f={\bf 1}_{B_{n-1}(0',1)}{\bf e}_j$
where, generally speaking, $B_{n-1}(x',r)$ denotes the $(n-1)$-dimensional ball of radius $r>0$ centered at
$x'\in{\mathbb{R}}^{n-1}$. Then there exists a point $x'\in B_{n-1}(0',1)$ at which
\eqref{yfg-yLLL.BN11} holds. For such a point $x'$ we then have
\begin{align}\label{yfg-yLLL.BN12}
f(x') &=\lim_{t\to 0^+}(P_t\ast f)(x')
=\lim_{t\to 0^+}\int_{{\mathbb{R}}^{n-1}}P(z')f(x'-tz')\,dz'
\nonumber\\[4pt]
&=\Big(\int_{{\mathbb{R}}^{n-1}}P(z')\,dz'\Big)f(x'),
\end{align}
where the second equality in \eqref{yfg-yLLL.BN12} is obtained via a change
of variables and the last equality in \eqref{yfg-yLLL.BN12} follows by applying
Lebesgue's Dominated Convergence Theorem. Thus, $\int_{{\mathbb{R}}^{n-1}}P(z')\,dz'$
is an $M\times M$ matrix whose action preserves ${\bf e}_j$. Since $j\in\{1,\dots,n\}$ was
arbitrary, this readily implies that $P$ satisfies {\it (2)} in Theorem~\ref{kkjbhV}.
With this in hand, the uniqueness statement in Theorem~\ref{kkjbhV}
finishes the proof of the proposition.
\end{proof}

After this preamble, we are prepared to present the proof of Theorem~\ref{V-Naa.11}.

\vskip 0.08in
\begin{proof}[Proof of Theorem~\ref{V-Naa.11}]
Let $L$ be a system, with complex constant coefficients, as in \eqref{L-def}-\eqref{L-ell.X}.
Denote by $P^L$ the Poisson kernel of $L$ and fix $p\in(1,\infty)$. From \eqref{eq:Cabnm} we know that
the domain of the infinitesimal generator ${\mathbf{A}}$ of the semigroup $T$
from \eqref{eq:Taghb8} is given by
\begin{equation}\label{eq:Cabnm.2}
D({\mathbf{A}})=\Big\{f\in L^p({\mathbb{R}}^{n-1}):\,
\lim_{t\to 0^{+}}\frac{T(t)f-f}{t}\,\,\mbox{ exists in }\,\,L^p({\mathbb{R}}^{n-1})\Big\}.
\end{equation}
The first order of business is proving \eqref{eq:tfc.1-new}.
To this end, let $f\in D({\mathbf{A}})$ be an arbitrary, fixed, function.
In particular, $f\in L^p({\mathbb{R}}^{n-1})$ and we introduce
\begin{equation}\label{eq:FDww}
u(x',t):=\big(T(t)f\big)(x')=(P^L_t\ast f)(x')\,\,\,\mbox{ for }\,\,\,(x',t)\in{\mathbb{R}}^n_{+}.
\end{equation}
Hence, if we now fix some $\kappa>0$, it follows from \eqref{eq:FDww}, \eqref{exTGFVC},
and parts {\it (6)}, {\it (7)} in Theorem~\ref{kkjbhV} that
\begin{equation}\label{KJHab}
\begin{array}{c}
u\in{\mathscr{C}}^\infty({\mathbb{R}}^n_{+}),\quad Lu=0\,\,\mbox{ in }\,\,{\mathbb{R}}^n_{+},
\\[6pt]
{\mathcal{N}}_\kappa u\in L^p({\mathbb{R}}^{n-1})\,\,\mbox{ and }\,\,
u\big|^{{}^{\rm n.t.}}_{\partial{\mathbb{R}}^n_{+}}=f\,\,\mbox{ a.e.~in }\,\,{\mathbb{R}}^{n-1}.
\end{array}
\end{equation}
Also, set
\begin{equation}\label{eq:FDww.2}
g:={\mathbf{A}}f=\lim_{t\to 0^{+}}\frac{T(t)f-f}{t}\in L^p({\mathbb{R}}^{n-1}).
\end{equation}
Then, for each $s\in(0,\infty)$, the properties of the semigroup $T$ from \eqref{eq:Taghb8} permit us
to write
\begin{align}\label{eq:FDww.3}
T(s)g &=T(s)\Big(\lim_{t\to 0^{+}}\frac{T(t)f-f}{t}\Big)
=\lim_{t\to 0^{+}}T(s)\Big(\frac{T(t)f-f}{t}\Big)
\nonumber\\[4pt]
&=\lim_{t\to 0^{+}}\frac{T(t+s)f-T(s)f}{t}
=\lim_{t\to 0^{+}}\frac{u(\cdot,t+s)-u(\cdot,s)}{t}
\,\,\mbox{ in }\,\,L^p({\mathbb{R}}^{n-1}).
\end{align}
As such, for each $s\in(0,\infty)$ fixed,
there exists $\{t_j\}_{j\in{\mathbb{N}}}\subset(0,\infty)$ with the property that
\begin{equation}\label{eq:Jfac}
\lim_{j\to\infty}t_j=0\,\,\mbox{ and }\,\,
\lim_{j\to\infty}\frac{u(\cdot,t_j+s)-u(\cdot,s)}{t_j}=T(s)g
\,\,\mbox{ a.e. in }\,\,{\mathbb{R}}^{n-1}.
\end{equation}
In turn, \eqref{eq:Jfac} forces
\begin{equation}\label{eq:Jfac.2}
(\partial_n u)(\cdot,s)=T(s)g\,\,\mbox{ a.e. in }\,\,{\mathbb{R}}^{n-1},
\,\,\mbox{ for each }\,\,s\in(0,\infty).
\end{equation}
Keeping in mind \eqref{eq:Taghb8} and adjusting notation, we may rephrase this as
\begin{equation}\label{eq:Jfac.2b}
(\partial_n u)(x',t)=(P^L_t\ast g)(x')\,\,\mbox{ for each }\,\,x'\in{\mathbb{R}}^{n-1}
\,\,\mbox{ and }\,\,t\in(0,\infty).
\end{equation}
From \eqref{eq:Jfac.2b} and part {\it (7)} in Theorem~\ref{kkjbhV} we deduce that
\begin{equation}\label{KBvcac65f5}
\mbox{$(\partial_n u)\big|^{{}^{\rm n.t.}}_{\partial{\mathbb{R}}^n_{+}}$ exists
a.e.~in ${\mathbb{R}}^{n-1}$}
\end{equation}
and, in fact,
\begin{equation}\label{eq:JfREDv}
g=(\partial_n u)\big|^{{}^{\rm n.t.}}_{\partial{\mathbb{R}}^n_{+}}
\,\,\mbox{ a.e.~in }\,\,{\mathbb{R}}^{n-1}.
\end{equation}
In concert with part {\it (6)} in Theorem~\ref{kkjbhV},
formula \eqref{eq:Jfac.2b} also implies
\begin{equation}\label{eq:jufds}
{\mathcal{N}}_\kappa(\partial_n u)\leq C{\mathcal{M}}g\,\,\mbox{ in }\,\,{\mathbb{R}}^{n-1}.
\end{equation}
Hence, given that $g\in L^p({\mathbb{R}}^{n-1})$ and ${\mathcal{M}}$ is bounded
on $L^p({\mathbb{R}}^{n-1})$,
\begin{equation}\label{eq:jufds.22}
{\mathcal{N}}_\kappa(\partial_n u)\in L^p({\mathbb{R}}^{n-1}).
\end{equation}

Moving on, fix an arbitrary point $(x',t)\in{\mathbb{R}}^n_{+}$.
Based on \eqref{KJHab} and the interior estimate \eqref{detraz} in the ball
$B\big((x',t),t/2\big)\subset{\mathbb{R}}^n_{+}$,
for each $k\in{\mathbb{N}}_0$ we may estimate
\begin{align}\label{detraz.TTT.UU}
\big|(\nabla^k u)(x',t)\big|
&\leq\frac{C}{t^k}\,\,\aver{B\big((x',t),t/2\big)}|u(y',s)|\,dy'ds
\nonumber\\[4pt]
&\leq\frac{C}{t^{k+n}}\int_{t/2}^{3t/2}\int_{B_{n-1}(x',t/2)}|u(y',s)|\,dy'ds
\nonumber\\[4pt]
&\leq\frac{C}{t^{k}}\,\,\aver{B_{n-1}(x',t/2)}({\mathcal{N}}_\kappa u)(y')\,dy'
\nonumber\\[4pt]
&\leq\frac{C}{t^{k}}\,{\mathcal{M}}\big({\mathcal{N}}_\kappa u\big)(x').
\end{align}
An alternative version of the above estimate which is also going to be useful
for us shortly (once again based on the interior estimate \eqref{detraz}) reads
\begin{align}\label{eq:jufds.23AAA}
\big|(\nabla^k u)(x',t)\big|
&\leq\frac{C}{t^k}\left(\aver{B\big((x',t),t/2\big)}|u(y',s)|^p\,dy'ds\right)^{1/p}
\nonumber\\[4pt]
&\leq\frac{C}{t^{k+n/p}}\left(\int_{t/2}^{3t/2}\int_{B_{n-1}(x',t/2)}|u(y',s)|^p\,dy'ds\right)^{1/p}
\nonumber\\[4pt]
&\leq\frac{C}{t^{k+(n-1)/p}}\left(\int_{B_{n-1}(x',t/2)}({\mathcal{N}}_\kappa u)(y')^p\,dy'\right)^{1/p}
\nonumber\\[4pt]
&\leq C\,t^{-k-(n-1)/p}\,
\|{\mathcal{N}}_\kappa u\|_{L^p({\mathbb{R}}^{n-1})},
\end{align}
for each $(x',t)\in{\mathbb{R}}^n_{+}$ and each $k\in{\mathbb{N}}_0$,
where $C=C(n,L,k,p)\in(0,\infty)$.

To proceed, fix $\varepsilon>0$ and define
\begin{equation}\label{jxfrA-2}
u_\varepsilon:{\mathbb{R}}^n_{+}\to {\mathbb{C}}^M,\quad
u_\varepsilon(x',t):=u(x',t+\varepsilon),\quad\forall\,(x',t)\in{\mathbb{R}}^n_{+}.
\end{equation}
Then, from \eqref{KJHab},
\begin{equation}\label{jxfrA}
\begin{array}{c}
u_\varepsilon\in{\mathscr{C}}^\infty(\overline{{\mathbb{R}}^n_{+}}),
\quad Lu_\varepsilon=0\,\,\mbox{ in }\,\,{\mathbb{R}}^n_{+},\quad
{\mathcal{N}}_\kappa u_\varepsilon\in L^p({\mathbb{R}}^{n-1}),
\\[4pt]
\mbox{ and }\,\,
u_\varepsilon\big|_{\partial{\mathbb{R}}^n_{+}}
=u(\cdot,\varepsilon)\,\,\mbox{ in }\,\,{\mathbb{R}}^{n-1}.
\end{array}
\end{equation}
In addition, from \eqref{detraz.TTT.UU} we see that for each $k\in{\mathbb{N}}_0$
and each $(x',t)\in{\mathbb{R}}^n_{+}$,
\begin{align}\label{elsiWth.UU}
\big|(\nabla^k u_\varepsilon)(x',t)\big|
&=\big|(\nabla^k u)(x',t+\varepsilon)\big|
\nonumber\\[4pt]
&\leq C_{L,n,k}(t+\varepsilon)^{-k}\,
\,{\mathcal{M}}\big({\mathcal{N}}_\kappa u\big)(x')
\nonumber\\[4pt]
&\leq C_{L,n,k}\,\varepsilon^{-k}\,
\,{\mathcal{M}}\big({\mathcal{N}}_\kappa u\big)(x'),
\end{align}
hence,
\begin{align}\label{elsiWth.UU2}
\big|(\nabla^k u_\varepsilon)_{\rm rad}(x')\big|
:=\sup_{t>0}\big|(\nabla^k u_\varepsilon)(x',t)\big|
\leq C_{L,n,k,\varepsilon}\,\,{\mathcal{M}}\big({\mathcal{N}}_\kappa u\big)(x').
\end{align}
For each fixed $x'\in{\mathbb{R}}^{n-1}$ note that if
$(y',s)\in\Gamma_\kappa(x')$ then
\begin{equation}\label{jhgyw}
B_{n-1}(y',s/2)\subset B_{n-1}\big(x',(1/2+\kappa)s\big).
\end{equation}
This and interior estimates then permit us to write (again, for each $k\in{\mathbb{N}}_0$)
\begin{align}\label{elsiWth.UU2b}
{\mathcal{N}}_\kappa(\nabla^k u_\varepsilon)(x') &=\sup_{(y',s)\in\Gamma_\kappa(x')}
\big|(\nabla^k u_\varepsilon)(y',s)\big|
\nonumber\\[4pt]
&\leq C\sup_{(y',s)\in\Gamma_\kappa(x')}\aver{B\big((y',s),s/2\big)}
\big|(\nabla^k u_\varepsilon)(z)\big|\,dz
\nonumber\\[4pt]
&\leq C\sup_{(y',s)\in\Gamma_\kappa(x')} s^{1-n}\int_{B_{n-1}(y',s/2)}
\big|(\nabla^k u_\varepsilon)_{\rm rad}(z')\big|\,dz'
\nonumber\\[4pt]
&\leq C\sup_{(y',s)\in\Gamma_\kappa(x')}\aver{B_{n-1}(x',s/2)}
\big|(\nabla^k u_\varepsilon)_{\rm rad}(z')\big|\,dz'
\nonumber\\[4pt]
&\leq C{\mathcal{M}}\big((\nabla^k u_\varepsilon)_{\rm rad}\big)(x'),
\qquad\forall\,x'\in{\mathbb{R}}^{n-1}.
\end{align}
Combining \eqref{elsiWth.UU2b}, \eqref{elsiWth.UU2}, and the homogeneity and
monotonicity of ${\mathcal{M}}$, we further obtain
\begin{align}\label{elsiWth.UU3}
{\mathcal{N}}_\kappa(\nabla^k u_\varepsilon)(x')\leq C
{\mathcal{M}}\big((\nabla^k u_\varepsilon)_{\rm rad}\big)(x')
\leq C_{L,n,k,\varepsilon}\,
{\mathcal{M}}\Big({\mathcal{M}}\big({\mathcal{N}}_\kappa u\big)\Big)(x')
\end{align}
for each $x'\in{\mathbb{R}}^{n-1}$ and each $k\in{\mathbb{N}}_0$. In concert with the
$L^p$ boundedness of the Hardy-Littlewood operator, \eqref{elsiWth.UU3} implies that
\begin{align}\label{elsiWth.UU3b}
\big\|{\mathcal{N}}_\kappa(\nabla^k u_\varepsilon)\big\|_{L^p({\mathbb{R}}^{n-1})}
\leq C_{L,n,k,\varepsilon}\big\|{\mathcal{N}}_\kappa u\big\|_{L^p({\mathbb{R}}^{n-1})},
\qquad\forall\,k\in{\mathbb{N}}_0.
\end{align}
In addition, from \eqref{eq:jufds.23AAA} we see that for each $(x',t)\in{\mathbb{R}}^n_{+}$
and $k\in{\mathbb{N}}_0$
\begin{equation}\label{elsiWth}
\big|(\nabla^k u_\varepsilon)(x',t)\big|
=\big|(\nabla^k u)(x',t+\varepsilon)\big|
\leq C_{L,n,k,p}(t+\varepsilon)^{-k-(n-1)/p}\,
\|{\mathcal{N}}_\kappa u\|_{L^p({\mathbb{R}}^{n-1})},
\end{equation}
hence,
\begin{equation}\label{eq:jBBV}
\lim_{t\to\infty}\Big[t^k\big|(\nabla^k u_\varepsilon)(x',t)\big|\Big]=0,
\qquad\forall\,k\in{\mathbb{N}}_0,
\end{equation}
uniformly for $x'$ in ${\mathbb{R}}^{n-1}$.

Next, pick $j\in\{1,\dots,n\}$ arbitrary and for each $s>0$ fixed for the time being,
consider the function $w:{\mathbb{R}}^n_{+}\to {\mathbb{C}}^M$ defined by
$w(x',t):=(\partial_j\partial_nu_\varepsilon)(x',t+s)$ for $(x',t)\in{\mathbb{R}}^n_{+}$.
Then \eqref{jxfrA} and \eqref{elsiWth.UU3b} imply that
$w\in{\mathscr{C}}^\infty(\overline{{\mathbb{R}}^n_{+}})$, $Lw=0$ in ${\mathbb{R}}^n_{+}$, and
${\mathcal{N}}_\kappa w\in L^p({\mathbb{R}}^{n-1})$.
Invoking the Fatou-type result in Theorem~\ref{tuFatou} for the function
$w$ then yields the representation
\begin{align}\label{fsrF}
(\partial_j\partial_nu_\varepsilon)(x',t+s)
&=\big(P_t^L\ast(\partial_j\partial_nu_\varepsilon)(\cdot,s)\big)(x')
\nonumber\\[4pt]
&=\int_{{\mathbb{R}}^{n-1}}K^L(x'-y',t)(\partial_j\partial_nu_\varepsilon)(y',s)\,dy'
\end{align}
for each $(x',t)\in{\mathbb{R}}^n_{+}$ and each $s\in(0,\infty)$.
Fix now an arbitrary function $h\in{\mathscr{C}}^\infty_0({\mathbb{R}}^{n-1})$.
Applying first $\frac{d}{dt}$ to both sides of identity \eqref{fsrF}, then letting
$s=t$, and then integrating on ${\mathbb{R}}^{n-1}\times(0,\infty)$
with respect to the measure $h(x')t\,dx'\,dt$, yields
\begin{align}\label{GEsc}
& \int_0^\infty\int_{{\mathbb{R}}^{n-1}}(\partial_n^2\partial_j u_\varepsilon)(x',2t)h(x')t\,dx'\,dt
\\[4pt]
&=\int_0^\infty\int_{{\mathbb{R}}^{n-1}}\int_{{\mathbb{R}}^{n-1}}
(\partial_n K^L)(x'-y',t)(\partial_j\partial_n u_\varepsilon)(y',t)h(x')t\,dy'\,dx'\,dt.
\nonumber
\end{align}
Note that
\begin{equation}\label{eq:yTVV}
(\partial_n^2\partial_j u_\varepsilon)(x',2t)
=\frac{1}{4}\Big(\frac{d}{dt}\Big)^2\big[(\partial_j u_\varepsilon)(x',2t)\big],
\qquad\forall\,(x',t)\in{\mathbb{R}}^n_{+},
\end{equation}
so Fubini's Theorem, integration by parts, and the Fundamental Theorem of Calculus
in the variable $t$ yield, upon recalling \eqref{eq:jBBV}, that
\begin{align}\label{GEsc-2}
\int_0^\infty\int_{{\mathbb{R}}^{n-1}}(\partial_n^2\partial_j u_\varepsilon)(x',2t)h(x')t\,dx'\,dt
&=-\frac{1}{4}\int_{{\mathbb{R}}^{n-1}}\int_0^\infty
\frac{d}{dt}\big[(\partial_j u_\varepsilon)(x',2t)\big]h(x')\,dt\,dx'
\nonumber\\[4pt]
&=\frac{1}{4}\int_{{\mathbb{R}}^{n-1}}(\partial_j u_\varepsilon)(x',0)h(x')\,dx'.
\end{align}
As a preamble to dealing with the integral in the right-hand side of \eqref{GEsc},
we make the following observation. Recall the area-function from \eqref{ga.12-hat}
and let $v_{n-1}$ be the volume of the unit ball in ${\mathbb{R}}^{n-1}$.
Also, let $p'\in(1,\infty)$ be such that $1/p+1/p'=1$. Then, given two measurable
functions $v_1,v_2$ defined in ${\mathbb{R}}^n_{+}$, by \eqref{ga.12-hat},
Fubini's theorem, and H\"older's inequality we obtain
\begin{align}\label{uasgh}
v_{n-1}\kappa^{n-1}\int_{{\mathbb{R}}^n_{+}} & |v_1(y',t)||v_2(y',t)|\,t\,dy'\,dt
\nonumber\\[4pt]
&=\int_{{\mathbb{R}}^n_{+}}\Bigg(\int_{\mathbb{R}^{n-1}}
{\bf 1}_{|x'-y'|<\kappa t}\,dx'\Bigg)|v_1(y',t)||v_2(y',t)|\,\frac{dy'\,dt}{t^{n-2}}
\nonumber\\[4pt]
&=\int_{\mathbb{R}^{n-1}}\Bigg(\int_{\Gamma_{\kappa}(x')}
|v_1(y',t)||v_2(y',t)|\,\frac{dy'\,dt}{t^{n-2}}\Bigg)dx'
\nonumber\\[4pt]
&\leq \int_{\mathbb{R}^{n-1}}({\mathcal{A}}_{\kappa}v_1)(x')
({\mathcal{A}}_{\kappa}v_2)(x')\,dx'
\nonumber\\[4pt]
&\leq\big\|{\mathcal{A}}_{\kappa}v_1\big\|_{L^p({\mathbb{R}}^{n-1})}
\big\|{\mathcal{A}}_{\kappa}v_2\big\|_{L^{p'}({\mathbb{R}}^{n-1})}.
\end{align}

Returning to the mainstream discussion, for each $\ell\in\{1,\dots,n\}$ define
\begin{align}\label{GEsc-bdt}
(\Theta^\ell h)(x',t):=\Big(
\int_{{\mathbb{R}}^{n-1}}\theta^\ell_{\alpha\beta}(x',t;y')h(y')\,dy'\Big)_{1\leq\alpha,\beta\leq M},
\qquad\forall\,(x',t)\in{\mathbb{R}}^n_{+},
\end{align}
where, for each $\alpha,\beta\in\{1,\dots,M\}$,
\begin{equation}\label{eq:KhAFV}
\theta^\ell_{\alpha\beta}(x',t;y'):=\big(\partial_\ell K^L_{\alpha\beta}\big)(y'-x',t)\quad
\mbox{for $x',y'\in\mathbb{R}^{n-1}$ and $t>0$}.
\end{equation}
In this regard, we first observe that \eqref{eq:Kest} implies
\begin{equation}\label{est-theta-K}
|\theta^\ell_{\alpha\beta}(x',t;y')|\leq |\nabla K^L_{\alpha\beta}(y'-x',t)|
\leq C|(x'-y',t)|^{-n}
\end{equation}
as well as
\begin{equation}\label{est-theta-K-nabla}
|\nabla_{y'}\theta^\ell_{\alpha\beta}(x',t;y')|\leq C|\nabla^2 K^L_{\alpha\beta}(y'-x',t)|
\leq C|(x'-y',t)|^{-n-1},
\end{equation}
for all $\alpha,\beta\in\{1,\dots,M\}$ and $\ell\in\{1,\dots,n\}$.
These properties are in agreement with \eqref{SFE-est-theta}. On the other hand,
by $(2)$ in Theorem~\ref{kkjbhV} we have the cancellation property
\begin{align}\label{est-theta-vanish}
\int_{\mathbb{R}^{n-1}}\theta^n_{\alpha\beta}(x',t;y')\,dy'
&=\int_{\mathbb{R}^{n-1}}\big(\partial_n K^L_{\alpha\beta}\big)(y'-x',t)\,dy'
\nonumber\\[4pt]
&=\frac{d}{dt}\int_{\mathbb{R}^{n-1}}K^L_{\alpha\beta}(y',t)\,dy'
=\frac{d}{dt}\int_{\mathbb{R}^{n-1}}(P^L_{\alpha\beta})_t(y')\,dy'
\nonumber\\[4pt]
&=\frac{d}{dt}\int_{\mathbb{R}^{n-1}}P^L_{\alpha\beta}(y')\,dy'
=0,
\end{align}
and, similarly, when $\ell\in\{1,\dots,n-1\}$,
\begin{align}\label{est-theta-vanish-ell}
\int_{\mathbb{R}^{n-1}}\theta^\ell_{\alpha\beta}(x',t;y')\,dy'
&=\int_{\mathbb{R}^{n-1}}\big(\partial_\ell K^L_{\alpha\beta}\big)(y'-x',t)\,dy'
\nonumber\\[4pt]
&=-\frac{\partial}{\partial_{x_\ell}}\int_{\mathbb{R}^{n-1}}K^L_{\alpha\beta}(y',t)\,dy'
=0.
\end{align}
These computations allow us to apply Theorem~\ref{prop:SFE} and write
\begin{equation}\label{eq:nbdu}
\big\|{\mathcal{A}}_{\kappa}(\Theta^\ell\eta)\big\|_{L^p(\mathbb{R}^{n-1})}
\leq C_{\kappa,L,p}\|\eta\|_{L^p(\mathbb{R}^{n-1})},
\quad\forall\,\eta\in L^p(\mathbb{R}^{n-1}),\,\,\forall\,\ell\in\{1,\dots,n\}.\quad
\end{equation}

To proceed, note that for each $\ell\in\{1,\dots,n\}$ we have
\begin{align}\label{GEsc-d}
&\int_0^\infty\int_{{\mathbb{R}}^{n-1}}\int_{{\mathbb{R}}^{n-1}}
(\partial_\ell K^L)(x'-y',t)(\partial_j\partial_n u_\varepsilon)(y',t)h(x')t\,dy'\,dx'\,dt
\nonumber\\[4pt]
&=\int_0^\infty\int_{{\mathbb{R}}^{n-1}}(\Theta^\ell h)(y',t)
(\partial_j\partial_n u_\varepsilon)(y',t)t\,dy'\,dt.
\end{align}
From \eqref{GEsc}, \eqref{GEsc-2}, and \eqref{GEsc-d} (with $\ell=n$) it follows that
\begin{equation}\label{GEsc-2dg}
\frac{1}{4}\int_{{\mathbb{R}}^{n-1}}(\partial_j u_\varepsilon)(x',0)h(x')\,dx'
=\int_{{\mathbb{R}}^n_{+}}(\Theta^n h)(y',t)(\partial_j\partial_n u_\varepsilon)(y',t)
t\,dy'\,dt.
\end{equation}
Upon invoking \eqref{uasgh} with $v_1:=\partial_j\partial_n u_\varepsilon$ and $v_2:=\Theta^n h$,
as well as Theorem~\ref{prop:SFE}, identity \eqref{GEsc-2dg} further yields
\begin{align}\label{GEsc-2dgf}
\Big|\int_{{\mathbb{R}}^{n-1}}(\partial_j u_\varepsilon)(x',0)h(x')\,dx'\Big|
&\leq\frac{4}{v_{n-1}\kappa^{n-1}}
\big\|{\mathcal{A}}_{\kappa}(\partial_j\partial_n u_\varepsilon)\big\|_{L^p({\mathbb{R}}^{n-1})}
\big\|{\mathcal{A}}_{\kappa}(\Theta^n h)\big\|_{L^{p'}({\mathbb{R}}^{n-1})}
\nonumber\\[4pt]
&\leq C_{n,\kappa,L,p}
\big\|{\mathcal{A}}_{\kappa}(\partial_j\partial_n u_\varepsilon)\big\|_{L^p({\mathbb{R}}^{n-1})}
\|h\|_{L^{p'}({\mathbb{R}}^{n-1})}.
\end{align}
Since $h\in{\mathscr{C}}^\infty_0({\mathbb{R}}^{n-1})$ is arbitrary,
\eqref{GEsc-2dgf} and Riesz's Representation Theorem imply
\begin{equation}\label{jdywtg}
\big\|(\partial_j u_\varepsilon)(\cdot,0)\big\|_{L^p({\mathbb{R}}^{n-1})}
\leq C_{n,\kappa,L,p}
\big\|{\mathcal{A}}_{\kappa}(\partial_j\partial_n u_\varepsilon)\big\|_{L^p({\mathbb{R}}^{n-1})}.
\end{equation}
On the other hand, from \eqref{jxfrA}, \eqref{elsiWth.UU3b}, and Theorem~\ref{tuFatou}
we obtain
\begin{equation}\label{eq:Temp1}
(\partial_j u_\varepsilon)(x',t)=\big(P^L_t\ast(\partial_j u_\varepsilon)(\cdot,0)\big)(x'),
\qquad\forall\,(x',t)\in{\mathbb{R}}^n_{+},
\end{equation}
as well as
\begin{align}\label{eq:Temp2}
(\partial_j \partial_n u_\varepsilon)(x',t)
&=\partial_j\big(P^L_t\ast(\partial_n u_\varepsilon)(\cdot,0)\big)(x')
\nonumber\\[4pt]
&=\big(\Theta^j (\partial_n u_\varepsilon)(\cdot,0)\big)(x',t),
\qquad\forall\,(x',t)\in{\mathbb{R}}^n_{+},
\end{align}
taking \eqref{eq:Gvav7g5}, \eqref{GEsc-bdt}-\eqref{eq:KhAFV} into account. Consequently,
for each $j\in\{1,\dots,n\}$,
\begin{align}\label{eq:iFD}
\|{\mathcal{N}}_{\kappa}(\partial_j u_\varepsilon)\|_{L^p({\mathbb{R}}^{n-1})} &\leq C
\big\|{\mathcal{M}}\big((\partial_j u_\varepsilon)(\cdot,0)\big)\big\|_{L^p({\mathbb{R}}^{n-1})}
\leq C\big\|(\partial_j u_\varepsilon)(\cdot,0)\big\|_{L^p({\mathbb{R}}^{n-1})}
\nonumber\\[6pt]
&\leq C\|{\mathcal{A}}_{\kappa}(\partial_j\partial_n u_\varepsilon)\|_{L^p({\mathbb{R}}^{n-1})}
\nonumber\\[6pt]
&\leq C\|{\mathcal{A}}_{\kappa}\big(\Theta^j (\partial_n u_\varepsilon)(\cdot,0)\big)
\|_{L^p({\mathbb{R}}^{n-1})}
\nonumber\\[6pt]
&\leq C\|(\partial_n u_\varepsilon)(\cdot,0)\|_{L^p({\mathbb{R}}^{n-1})}
=C\|(\partial_n u)(\cdot,\varepsilon)\|_{L^p({\mathbb{R}}^{n-1})}
\nonumber\\[6pt]
&\leq C\|{\mathcal{N}}_{\kappa}(\partial_n u)\|_{L^p({\mathbb{R}}^{n-1})}.
\end{align}
For the first inequality in \eqref{eq:iFD} we used \eqref{eq:Temp1} and part {\it (6)}
in Theorem~\ref{kkjbhV}, the second inequality is based on the $L^p$ boundedness of
the Hardy-Littlewood maximal operator, the third is just \eqref{jdywtg},
the fourth is just \eqref{eq:Temp2}, the fifth is a consequence of \eqref{eq:nbdu},
the sixth follows from \eqref{jxfrA-2}, while the last one is seen from \eqref{NT-Fct}.
Next, note that
\begin{align}\label{exTGF-DFF}
{\mathcal{N}}_{\kappa}(\partial_j u_\varepsilon)(x')
&=\sup_{|x'-y'|<\kappa t}\big|(\partial_j u_\varepsilon)(y',t)\big|
=\sup_{(y',t)\in\Gamma_\kappa(x')}\big|(\partial_j u)(y',t+\varepsilon)\big|
\nonumber\\[6pt]
&=\sup_{(y',t)\in\Gamma_\kappa(x')+\varepsilon{\bf e}_n}\big|(\partial_j u)(y',t)\big|
\nearrow\sup_{(y',t)\in\Gamma_\kappa(x')}\big|(\partial_j u)(y',t)\big|
\,\,\mbox{ as }\,\,\varepsilon\searrow 0,
\end{align}
hence
\begin{align}\label{exTGF-DFF.2}
{\mathcal{N}}_{\kappa}(\partial_j u_\varepsilon)(x')
\nearrow{\mathcal{N}}_{\kappa}(\partial_j u)(x')\,\,\mbox{ as }\,\,\varepsilon\searrow 0,
\end{align}
for each $x'\in{\mathbb{R}}^{n-1}$. Combining \eqref{exTGF-DFF.2} and \eqref{eq:iFD}
and relying on Lebesgue's Monotone Convergence Theorem then yields
\begin{align}\label{eq:iFD.556}
\|{\mathcal{N}}_{\kappa}(\partial_j u)\|_{L^p({\mathbb{R}}^{n-1})}
=\lim_{\varepsilon\searrow 0}
\|{\mathcal{N}}_{\kappa}(\partial_j u_\varepsilon)\|_{L^p({\mathbb{R}}^{n-1})}
\leq C\|{\mathcal{N}}_{\kappa}(\partial_n u)\|_{L^p({\mathbb{R}}^{n-1})},
\end{align}
for some finite constant $C>0$ independent of $u$.
This shows that \eqref{eq:jufds.22} improves to
\begin{equation}\label{eq:jufds.26}
{\mathcal{N}}_\kappa(\nabla u)\in L^p({\mathbb{R}}^{n-1}).
\end{equation}
As a consequence of \eqref{eq:jufds.26} and
Theorem~\ref{tuFatou} we have
\begin{equation}\label{eufds.26fff}
(\nabla u)\big|^{{}^{\rm n.t.}}_{\partial{\mathbb{R}}^n_{+}}
\,\,\mbox{ exists and belongs to }\,\,L^p({\mathbb{R}}^{n-1}).
\end{equation}

Next, fix some two arbitrary indices, $j\in\{1,\dots,n-1\}$ and $\alpha\in\{1,\dots,M\}$, along with some
scalar-valued test function $\psi\in{\mathscr{C}}^\infty_0({\mathbb{R}}^{n-1})$. Extend $\psi$
to some $\widetilde{\psi}\in{\mathscr{C}}^\infty_0({\mathbb{R}}^{n})$ and, with $u_\alpha$ denoting
the $\alpha$-th component of $u$, consider the vector field
\begin{equation}\label{eq:Rgfab}
\vec{F}:=\partial_j\big(\widetilde{\psi}\,u_\alpha\big){\bf e}_n
-\partial_n\big(\widetilde{\psi}\,u_\alpha\big){\bf e}_j\in{\mathscr{C}}^\infty({\mathbb{R}}^n_{+},{\mathbb{C}}^n).
\end{equation}
From the design of $\vec{F}$ we see that
\begin{equation}\label{eq:ndjdf}
\vec{F}\in L^1_{\rm loc}({\mathbb{R}}^n_{+},{\mathbb{C}}^n),\qquad
{\rm div}\,\vec{F}=0\,\,\mbox{ in }\,\,{\mathbb{R}}^n_{+}.
\end{equation}
Also, from \eqref{KJHab}, \eqref{eq:jufds.26}, and \eqref{eq:Rgfab} we obtain
\begin{equation}\label{eq:ndjHG}
\mathcal{N}_\kappa\vec{F}\leq\big\|\widetilde{\psi}\big\|_{L^\infty({\mathbb{R}}^n)}
{\bf 1}_{K}\,{\mathcal{N}}_\kappa(\nabla u)
+\big\|\nabla\widetilde{\psi}\big\|_{L^\infty({\mathbb{R}}^n)}
{\bf 1}_{K}\,{\mathcal{N}}_\kappa u\in L^1({\mathbb{R}}^{n-1}),
\end{equation}
where $K$ is the compact subset of ${\mathbb{R}}^{n-1}$ given by
\begin{equation}\label{eq:Nva}
K:=\big\{x'\in{\mathbb{R}}^{n-1}:\,\overline{\Gamma_\kappa(x')}
\cap{\rm supp}\,\widetilde{\psi}\not=\emptyset\big\}.
\end{equation}
Keeping in mind that $u_\alpha\big|^{{}^{\rm n.t.}}_{\partial{\mathbb{R}}^n_{+}}=f_\alpha$,
the $\alpha$-th component of $f$ (cf. \eqref{KJHab}), it follows from \eqref{eq:Rgfab} and
\eqref{eufds.26fff} that $\vec{F}\big|^{{}^{\rm n.t.}}_{\partial{\mathbb{R}}^n_{+}}$
exists and is given by
\begin{align}\label{eq:ndjdf.2}
\vec{F}\big|^{{}^{\rm n.t.}}_{\partial{\mathbb{R}}^n_{+}}
&=\Big\{\Big((\partial_ju_\alpha)\big|^{{}^{\rm n.t.}}_{\partial{\mathbb{R}}^n_{+}}\Big)\psi
+f_\alpha\,\partial_j\psi\Big\}\,{\bf e}_n
\nonumber\\[4pt]
&\quad +\Big\{\Big((\partial_n u_\alpha)\big|^{{}^{\rm n.t.}}_{\partial{\mathbb{R}}^n_{+}}\Big)\psi
+f_\alpha\,\big(\partial_n\widetilde{\psi}\,\big)\big|_{\partial{\mathbb{R}}^n_{+}}
\Big\}\,{\bf e}_j\,\,\mbox{ a.e. in }\,\,{\mathbb{R}}^{n-1}.
\end{align}
Granted \eqref{eq:ndjdf}, \eqref{eq:ndjHG}, and \eqref{eq:ndjdf.2},
Theorem~\ref{theor:div-thm} then gives
\begin{equation}\label{eq:IIIa}
\int_{{\mathbb{R}}^{n-1}}f_\alpha\,\partial_j\psi\,d{\mathscr{L}}^{n-1}
=-\int_{{\mathbb{R}}^{n-1}}\Big((\partial_ju_\alpha)\big|^{{}^{\rm n.t.}}_{\partial{\mathbb{R}}^n_{+}}\Big)\psi\,d{\mathscr{L}}^{n-1},
\qquad\forall\,\psi\in{\mathscr{C}}^\infty_0({\mathbb{R}}^{n-1}).
\end{equation}
Together with \eqref{eufds.26fff}, this ultimately proves that,
in the sense of distributions in ${\mathbb{R}}^{n-1}$,
\begin{equation}\label{eq:IIIb}
\partial_j f=(\partial_ju)\big|^{{}^{\rm n.t.}}_{\partial{\mathbb{R}}^n_{+}}
\in L^p({\mathbb{R}}^{n-1}),\,\,\,\mbox{ for each }\,\,j\in\{1,\dots,n-1\}.
\end{equation}
Consequently, $f\in L^p_1({\mathbb{R}}^{n-1})$. When combined with
\eqref{KJHab} and \eqref{eq:jufds.26}, this shows that $u$ is a solution of
the Regularity problem \eqref{Dir-BVvku} with boundary datum $f$.
The argument so far proves the left-to-right inclusion in \eqref{eq:tfc.1-new}.

Let us also note here that from the well-posedness of the $L^p$-Dirichlet boundary value problem
(cf. Theorem~\ref{Theorem-NiceDP}) as well as \eqref{KJHab} and \eqref{eq:jufds.26},
it follows that $u$ given by \eqref{eq:FDww} is the unique solution of the Regularity
problem \eqref{Dir-BVvku} with boundary datum $f$.
With this in hand, \eqref{eq:FDww.2} and \eqref{KBvcac65f5}-\eqref{eq:JfREDv}
then justify \eqref{eq:hJB} and \eqref{mJBVV}.

To prove the right-to-left inclusion in \eqref{eq:tfc.1-new}, consider a function
$f\in L^p_1({\mathbb{R}}^{n-1})$ for which the Regularity problem \eqref{Dir-BVvku}
with boundary datum $f$ has a solution $u$. From Theorem~\ref{Theorem-NiceDP} it
follows that, necessarily, $u(x',t)=(P^L_t\ast f)(x')=\big(T(t)f\big)(x')$
for all $(x',t)\in{\mathbb{R}}^n_{+}$.
Since the boundary condition in \eqref{Dir-BVvku} entails
$\lim\limits_{t\to 0^{+}}u(\cdot,t)=f$ a.e.~in ${\mathbb{R}}^{n-1}$, we may write
\begin{equation}\label{eq:hds66}
\frac{T(t)f-f}{t}=\frac{u(\cdot,t)-f}{t}
=\meanint_{\!\!\!0}^{\,t}\,\,(\partial_n u)(\cdot,s)\,ds
\quad\mbox{a.e.~in ${\mathbb{R}}^{n-1}$}.
\end{equation}
In addition,
\begin{equation}\label{eq:hds67}
\Big|\meanint_{\!\!\!0}^{\,t}\,\,(\partial_n u)(\cdot,s)\,ds\Big|
\leq{\mathcal{N}}_\kappa(\partial_n u)\in L^p({\mathbb{R}}^{n-1}),\qquad\forall\,t>0.
\end{equation}
Furthermore, Theorem~\ref{tuFatou} gives that there exists $N\subset{\mathbb{R}}^{n-1}$
with ${\mathscr{L}}^{n-1}(N)=0$ and such that
\begin{eqnarray}\label{TaJBb}
\Big((\partial_n u)\big|^{{}^{\rm n.t.}}_{\partial{\mathbb{R}}^n_{+}}\Big)(x')
=\lim_{\Gamma_\kappa(x')\ni y\to(x',0)}(\partial_n u)(y)
\,\,\mbox{ exists for each }\,\,x'\in{\mathbb{R}}^{n-1}\setminus N.
\end{eqnarray}
In turn, this readily entails
\begin{equation}\label{eq:hds68}
\lim_{t\to 0^{+}}\meanint_{\!\!\!0}^{\,t}\,\,(\partial_n u)(x',s)\,ds
=\Big((\partial_n u)\big|^{{}^{\rm n.t.}}_{\partial{\mathbb{R}}^n_{+}}\Big)(x')\,\,\,
\mbox{ for each }\,\,x'\in{\mathbb{R}}^{n-1}\setminus N.
\end{equation}
From \eqref{eq:hds66}-\eqref{eq:hds68} and Lebesgue's Dominated
Convergence Theorem it follows that
\begin{equation}\label{eq:hds69}
\lim_{t\to 0^{+}}\frac{T(t)f-f}{t}
=(\partial_n u)\big|^{{}^{\rm n.t.}}_{\partial{\mathbb{R}}^n_{+}}
\,\,\,\mbox{ in }\,\,L^p({\mathbb{R}}^{n-1}).
\end{equation}
The bottom line is that $f\in D({\mathbf{A}})$, finishing
the proof of the right-to-left inclusion in \eqref{eq:tfc.1-new}.
At this stage, \eqref{eq:tfc.1-new} is established.

Consider now the claim made in \eqref{eq:tfc.1-new.RRR}. Pick an arbitrary
$f\in L^p_1({\mathbb{R}}^{n-1})$ and define $u$ as in \eqref{eq:FDww}.
Also, for each $\varepsilon>0$ introduce $u_\varepsilon$ as in \eqref{jxfrA-2}
and set
\begin{equation}\label{eq:vrcc}
f_\varepsilon:=u_\varepsilon\big|_{\partial{\mathbb{R}}^n_{+}}=u(\cdot,\varepsilon).
\end{equation}
By \eqref{jxfrA}, this is a well-defined function in ${\mathbb{R}}^{n-1}$ and, since
\begin{equation}\label{eq:pht6}
\mbox{$|f_\varepsilon|\leq{\mathcal{N}}_\kappa u$\,\,\, pointwise in\,\,\,${\mathbb{R}}^{n-1}$},
\end{equation}
it follows from \eqref{KJHab} that $f_\varepsilon\in L^p({\mathbb{R}}^{n-1})$.
Moreover, for each $j\in\{1,\dots,n-1\}$ we have
$\partial_j f_\varepsilon=(\partial_j u)(\cdot,\varepsilon)$ which, in view of
\eqref{elsiWth.UU3b} with $k=1$, forces $\partial_j f_\varepsilon\in L^p({\mathbb{R}}^{n-1})$.
The upshot of this analysis is that $f_\varepsilon\in L^p_1({\mathbb{R}}^{n-1})$.
In addition, from \eqref{jxfrA}, \eqref{elsiWth.UU3b} with $k=1$, and \eqref{eq:vrcc}
we see that $u_\varepsilon$ solves $(R_p)$ with boundary datum $f_\varepsilon$.
As such, \eqref{eq:tfc.1-new.RRR} will follow from \eqref{eq:tfc.1-new} as soon
as we show that
\begin{equation}\label{eq:KIgt5c}
f_\varepsilon\to f\,\,\mbox{ in }\,\,L^p_1({\mathbb{R}}^{n-1})
\,\,\mbox{ as }\,\,\varepsilon\to 0^{+}.
\end{equation}
With this in mind, we first observe that
$f_\varepsilon=u(\cdot,\varepsilon)\to u\big|^{{}^{\rm n.t.}}_{\partial{\mathbb{R}}^n_{+}}=f$
pointwise a.e.~in ${\mathbb{R}}^{n-1}$ as $\varepsilon\to 0^{+}$, by \eqref{eq:vrcc}
and \eqref{KJHab}. Together with \eqref{eq:pht6} and Lebesgue's Dominated Convergence Theorem
this proves that $f_\varepsilon\to f$ in $L^p({\mathbb{R}}^{n-1})$ as $\varepsilon\to 0^{+}$.
Granted this, there remains to observe that for each $j\in\{1,\dots,n-1\}$ we have
\begin{align}\label{i6tr}
\partial_j f_\varepsilon &=(\partial_j u)(\cdot,\varepsilon)
=\partial_j\big[P^L_\varepsilon\ast f\big]
\nonumber\\[4pt]
&=P^L_\varepsilon\ast(\partial_j f)\to \partial_j f
\,\,\mbox{ in }\,\,L^p({\mathbb{R}}^{n-1})
\,\,\mbox{ as }\,\,\varepsilon\to 0^{+}
\end{align}
where, for the last line, we have relied on parts {\it (6)}-{\it (7)} of Theorem~\ref{kkjbhV},
Lebesgue's Dominated Convergence Theorem, and the $L^p$-boundedness of the Hardy-Littlewood
maximal operator. This ultimately shows that \eqref{eq:KIgt5c} holds, finishing the proof of
\eqref{eq:tfc.1-new.RRR}.

Our next goal is to show that
\begin{equation}\label{eq:RPWP}
\mbox{\eqref{eq:Asd.FFF} implies that $(R_p)$ is well-posed}.
\end{equation}
For the remainder of the proof make the additional assumption that the system $L$
has the property that \eqref{eq:Asd.FFF} holds. Pick some
$f\in L^p_1({\mathbb{R}}^{n-1})$ and define $u$ as in \eqref{eq:FDww}.
In particular, $u$ satisfies \eqref{KJHab} and we also claim that there exists
a finite constant $C>0$, independent of $f$, such that
\begin{equation}\label{eq:hds65}
\begin{array}{c}
{\mathcal{N}}_\kappa(\partial_j u)\in L^p({\mathbb{R}}^{n-1})\,\,\mbox{ and }\,\,
\big\|{\mathcal{N}}_\kappa(\partial_j u)\big\|_{L^p({\mathbb{R}}^{n-1})}
\leq C\|f\|_{L^p_1({\mathbb{R}}^{n-1})}
\\[4pt]
\mbox{for each }\,\,\,j\in\{1,\dots,n\}.
\end{array}
\end{equation}
To justify this claim, let $E=\big(E_{\gamma\beta}\big)_{1\leq\gamma,\beta\leq M}$
be the fundamental solution for $L$ from Theorem~\ref{FS-prop}. Invoking \eqref{yfg-yLUU}
in Proposition~\ref{uniq:double->!Poisson}, then \eqref{fs-GLOB}, \eqref{E-Trans}
and then integration by parts, for each $\gamma\in\{1,...,M\}$ and
$(x',t)\in{\mathbb{R}}^n_{+}$ we may write
\begin{align}\label{yncuds-2}
(\partial_n u_\gamma)(x',t)
&=\frac{d}{dt}\Big[\big((P^L_{\gamma\alpha})_t\ast f_\alpha\big)(x')\Big]
=\frac{d}{dt}\int_{{\mathbb{R}}^{n-1}}K^L_{\gamma\alpha}(x'-y',t)f_\alpha(y')\,dy'
\nonumber\\[4pt]
&=2\int_{{\mathbb{R}}^{n-1}}a^{\beta\alpha}_{rn}(\partial_r\partial_n E_{\gamma\beta})(x'-y',t)
f_\alpha(y')\,dy'
\nonumber\\[4pt]
&=-2\sum\limits_{s=1}^{n-1}
\int_{{\mathbb{R}}^{n-1}}a^{\beta\alpha}_{rs}(\partial_r\partial_s E_{\gamma\beta})(x'-y',t)
f_\alpha(y')\,dy'
\nonumber\\[4pt]
&=-2\sum\limits_{s=1}^{n-1}
\int_{{\mathbb{R}}^{n-1}}a^{\beta\alpha}_{rs}(\partial_r E_{\gamma\beta})(x'-y',t)
(\partial_sf_\alpha)(y')\,dy'
\end{align}
(note that the summation in $s$ is only from $s=1$ to $s=n-1$, thus the
summation convention may not be used for this sub-index in \eqref{yncuds-2}).
We intend to apply Theorem~\ref{Main-T2} to \eqref{yncuds-2} in order to
take care of the claims made in \eqref{eq:hds65} in the case when $j=n$.
To do so, for each $s\in\{1,\dots,n-1\}$, and each $\alpha,\gamma\in\{1,\dots,M\}$,
define
\begin{equation}\label{eq:bdue}
K_{\gamma\alpha}^s(x):=a^{\beta\alpha}_{rs}(\partial_r E_{\gamma\beta})(x),
\qquad\forall\,x\in{\mathbb{R}}^n\setminus\{0\},
\end{equation}
and denote by ${\mathcal{T}}_{\gamma\alpha}^s$ the integral operator
as in \eqref{T-layer} associated to the kernel $K_{\gamma\alpha}^s$.
In this notation, \eqref{yncuds-2} becomes
\begin{eqnarray}\label{bsyus56}
\partial_n u
=\left(-2\sum\limits_{s=1}^{n-1}{\mathcal{T}}_{\gamma\alpha}^s(\partial_sf_\alpha)
\right)_{1\leq\gamma\leq M}\,\,\,\mbox{ in }\,\,\,{\mathbb{R}}^n_{+}.
\end{eqnarray}
Since the kernels $K_{\gamma\alpha}^s$ satisfy \eqref{ker} (as seen from
\eqref{eq:bdue} and the properties of $E$), Theorem~\ref{Main-T2} gives
that ${\mathcal{N}}_\kappa(\partial_n u)\in L^p({\mathbb{R}}^{n-1})$ and
$\big\|{\mathcal{N}}_\kappa(\partial_n u)\big\|_{L^p({\mathbb{R}}^{n-1})}
\leq C\|f\|_{L^p_1({\mathbb{R}}^{n-1})}$, for some $C\in(0,\infty)$ independent of $f$.
In the scenario when $j\in\{1,\dots,n-1\}$, the claims in \eqref{eq:hds65}
are more directly seen by writing
\begin{align}\label{eq:efgg}
(\partial_j u)(x',t) & =\frac{\partial}{\partial x_j}
\Big[\big(P^L_t\ast f\big)(x')\Big]
=\frac{\partial}{\partial x_j}\int_{{\mathbb{R}}^{n-1}}K^L(x'-y',t)f(y')\,dy'
\nonumber\\[4pt]
&=-\int_{{\mathbb{R}}^{n-1}}\partial_{y_j}\big[K^L(x'-y',t)\big]f(y')\,dy'
=\int_{{\mathbb{R}}^{n-1}}K^L(x'-y',t)(\partial_jf)(y')\,dy'
\nonumber\\[4pt]
&=\big(P^L_t\ast(\partial_jf)\big)(x'),\qquad
\forall\,(x',t)\in{\mathbb{R}}^n_{+},
\end{align}
then invoking part {\it (6)} in Theorem~\ref{kkjbhV}. Having proved \eqref{eq:hds65},
it follows that $u$ solves $(R_p)$ for the boundary datum $f$, and obeys natural
estimates. Since uniqueness for $(R_p)$ also holds (implied by uniqueness for $(D_p)$;
cf. Theorem~\ref{Theorem-NiceDP}), we conclude that \eqref{eq:RPWP} holds.
In light of \eqref{eq:tfc.2BB}, this further implies that
$D({\mathbf{A}})=L^p_1({\mathbb{R}}^{n-1})$ in the class of systems we are
currently considering. At this stage, there remains to establish \eqref{mdiab}, a task
to which we now turn. In the context of \eqref{bsyus56}, the jump-formula from
Theorem~\ref{Main-T2} (see also \cite[Theorem~11.11, pp.\,366-367]{DM}) gives that,
for each $\gamma\in\{1,\dots,M\}$ and at a.e. $x'\in{\mathbb{R}}^{n-1}$, we have
\begin{align}\label{nhdsiswu}
\Bigl((\partial_n u_\gamma)\Big|_{\partial{\mathbb{R}}^n_{+}}^{{}^{\rm n.t.}}\Big)(x')
&=
-2\sum\limits_{s=1}^{n-1}
\tfrac{1}{2i}\,\widehat{K_{\gamma\alpha}^s}(-{\bf e}_n)(\partial_sf_{\alpha})(x')
-
2\sum\limits_{s=1}^{n-1}T_{\gamma\alpha}^s(\partial_sf_{\alpha})(x')
\nonumber\\[4pt]
&=
-\sum\limits_{s=1}^{n-1}
(B^{-1})_{\gamma\beta}\,a^{\beta\alpha}_{ns}\,(\partial_sf_{\alpha})(x')
-2\sum\limits_{s=1}^{n-1}T_{\gamma\alpha}^s(\partial_sf_{\alpha})(x').
\end{align}
Above, each $T_{\gamma\alpha}^s$ is the principal value operator associated
with $K_{\gamma\alpha}^s$ as in \eqref{main-lim} and we have set
$B:=\big(a^{\alpha\beta}_{nn}\big)_{1\leq\alpha,\beta\leq M}\in{\mathbb{C}}^{M\times M}$.
For the last equality in \eqref{nhdsiswu} we used \eqref{eq:bdue} and \eqref{E-ftXC}.
Now \eqref{mdiab} follows from \eqref{mJBVV} and \eqref{nhdsiswu}.
The proof of Theorem~\ref{V-Naa.11} is therefore complete.
\end{proof}

We now turn to the proof of Corollary~\ref{yreeCCC}.

\vskip 0.08in
\begin{proof}[Proof of Corollary~\ref{yreeCCC}]
Suppose $L={\rm div}{A}\,\nabla$ is an elliptic scalar operator
with complex coefficients, ${A}=(a_{rs})_{1\leq r,s\leq n}\in{\mathbb{C}}^{n\times n}$.
Then the symmetric part of the coefficient matrix ${A}$, defined by
${A}_{\rm sym}:=\frac{1}{2}({A}+{A}^t)$, satisfies conditions
\eqref{Ea4-fCii-n3}-\eqref{Ea4-fCii-n2B} (see \cite[Remark~3.8, p.132]{EE-MMMM}).
From Proposition~\ref{UaalpIKL-PP} and Theorem~\ref{V-Naa.11} it is then immediate that
the infinitesimal generator for the Poisson semigroup in $L^p({\mathbb{R}}^{n-1})$
associated with $L$ as in Theorem~\ref{VCXga} has domain $L^p_1({\mathbb{R}}^{n-1})$.

Proposition~\ref{UaalpIKL-PP} also guarantees that ${A}_{\rm sym}$ does the job in
\eqref{eq:Asd.FFF}. As such, given some $f\in L^p_1({\mathbb{R}}^{n-1})$,
starting with \eqref{mdiab} in which ${A}_{\rm sym}$ is used in the writing of $L$,
for a.e. $x'\in{\mathbb{R}}^{n-1}$ we obtain
\begin{align}\label{mdiab-scalarhgst}
({\mathbf{A}}f)(x')&
=
-\sum\limits_{s=1}^{n-1}\frac{a_{ns}+a_{sn}}{2a_{nn}}\,(\partial_sf)(x')
\nonumber\\[4pt]
&\qquad
-2\sum\limits_{s=1}^{n-1}({A}_{\rm sym})_{rs}\lim_{\varepsilon\to0^{+}}
\int\limits_{\substack{y'\in{\mathbb{R}}^{n-1}\\ |x'-y'|>\varepsilon}}
(\partial_r E)(x'-y',0)(\partial_sf)(y')\,dy',
\end{align}
where $E$ is the fundamental solution for the operator $L$ from \eqref{YTcxb-ytSH}.
Furthermore, based on the explicit description in \eqref{YTcxb-ytSH}, for each
$x=(x_1,\dots,x_n)\in{\mathbb{R}}^n\setminus\{0\}$ and each $s\in\{1,\dots,n-1\}$, we write
\begin{align}\label{mditt}
({A}_{\rm sym})_{rs}(\partial_r E)(x)
&=({A}_{\rm sym})_{rs}\frac{1}{\omega_{n-1}\sqrt{{\rm det}\,({A}_{\rm sym})}}
\Big[\big(({A}_{\rm sym})^{-1}x\big)\cdot x\Big]^{\frac{-n}{2}}
\big(({A}_{\rm sym})^{-1}x\big)_r
\nonumber\\[4pt]
&=\frac{x_s}{\omega_{n-1}\sqrt{{\rm det}\,({A}_{\rm sym})}
\Big[\big(({A}_{\rm sym})^{-1}x\big)\cdot x\Big]^{\frac{n}{2}}}.
\end{align}
Now \eqref{mdiab-scalar} follows by combining \eqref{mdiab-scalarhgst} and \eqref{mditt}.
Of course, \eqref{mdiab-Laplacian} is a particular case of \eqref{mdiab-scalar}, while
\eqref{mdiab-Lap:LL} is seen from \eqref{mdiab-Laplacian.2} by working on the Fourier transform side:
\begin{equation}\label{eq:Tree}
\widehat{{\mathbf A}f}(\xi')
=
-\sum_{s=1}^{n-1}\widehat{R_s(\partial_sf)}(\xi')
=
-\sum_{s=1}^{n-1}-i\frac{\xi_s}{|\xi'|}\big(i\xi_s\widehat{f}(\xi')\big)
=
-|\xi'|\widehat{f}(\xi'),
\end{equation}
in agreement with the Fourier transform of $-\sqrt{-\Delta_{n-1}}f$.

Next, we consider the case when $L$ is the Lam\'e system \eqref{TYd-YG-76g} with
Lam\'e moduli as in \eqref{Yfhv-8yg}. Recall from \cite{MaMiMi}, \cite{MaMiMiMi}
that the coefficient tensor with entries given by
\begin{equation}\label{BUIg-17XX}
\begin{array}{c}
a^{\beta\alpha}_{rs}:=\mu\delta_{rs}\delta_{\alpha\beta}
+\frac{(\lambda+\mu)(2\mu+\lambda)}{3\mu+\lambda}\delta_{r\beta}\delta_{s\alpha}
+\frac{\mu(\lambda+\mu)}{3\mu+\lambda}\delta_{r\alpha}\delta_{s\beta},
\\[8pt]
\quad\mbox{for each }\,\,\,\alpha,\beta,r,s\in\{1,\dots,n\},
\end{array}
\end{equation}
satisfies conditions \eqref{Ea4-fCii-n3}-\eqref{Ea4-fCii-n2B}.
Hence, by Proposition~\ref{UaalpIKL-PP}, this coefficient tensor also does the job in
\eqref{eq:Asd.FFF}. Granted this, Theorem~\ref{V-Naa.11} applies and gives that
the infinitesimal generator for the Poisson semigroup in $L^p({\mathbb{R}}^{n-1})$
associated with the Lam\'e system as in Theorem~\ref{VCXga} has domain $L^p_1({\mathbb{R}}^{n-1})$.

Moving on, observe that (cf. \cite{DM} for details), when specialized
to the case of the Lam\'e system, the general fundamental solution described in
Theorem~\ref{FS-prop} becomes the matrix $E=(E_{\alpha\beta})_{1\leq\alpha,\beta\leq n}$
whose $(\alpha,\beta)$ entry is defined at each
$x=(x_1,...,x_n)\in{\mathbb{R}}^{n}\setminus\{0\}$ according to
\begin{eqnarray}\label{Lame-6}
E_{\alpha\beta}(x):=\left\{
\begin{array}{ll}
{\displaystyle{\frac{-1}{2\mu(2\mu+\lambda)\omega_{n-1}}
\left[\frac{3\mu+\lambda}{n-2}\frac{\delta_{\alpha\beta}}{|x|^{n-2}}
+\frac{(\mu+\lambda)x_\alpha x_\beta}{|x|^{n}}\right],}} & \mbox{ if }\,n\geq 3,
\\[18pt]
{\displaystyle{\frac{1}{4\pi\mu(2\mu+\lambda)}
\left[(3\mu+\lambda)\delta_{\alpha\beta}{\rm ln}\,|x|
-\frac{(\mu+\lambda)x_\alpha x_\beta}{|x|^{2}}\right],}} & \mbox{ if }\,n=2.
\end{array}
\right.
\end{eqnarray}
Fix now some $f=(f_\alpha)_{1\leq\alpha\leq n}\in L^p_1({\mathbb{R}}^{n-1})$.
Using the expression for the coefficient tensors from \eqref{BUIg-17XX} we obtain
\begin{equation}\label{eq:sq3}
\Big(\big[\big(a^{\sigma\tau}_{nn}\big)_{1\leq\sigma,\tau\leq M}\big]^{-1}
\Big)_{\gamma\beta}=\frac{1}{\mu}\delta_{\gamma\beta}
-\frac{\mu+\lambda}{\mu(2\mu+\lambda)}\delta_{n\gamma}\delta_{n\beta},
\qquad\forall\,\gamma,\beta\in\{1,\dots,n\},
\end{equation}
so that
\begin{align}\label{mdiab-csf}
&
\left(\sum\limits_{s=1}^{n-1}
\Big(\big[\big(a^{\sigma\tau}_{nn}\big)_{1\leq\sigma,\tau\leq M}\big]^{-1}
\Big)_{\gamma\beta}\,a^{\beta\alpha}_{ns}\,(\partial_sf_{\alpha})(x')\right)_{1\leq\gamma\leq n}
\nonumber\\[4pt]
&\qquad
=\left(\sum_{s=1}^{n-1}\frac{\mu+\lambda}{3\mu+\lambda}\left[
\delta_{\gamma n}\delta_{s\alpha}+\delta_{n\alpha}\delta_{\gamma s}\right]
(\partial_sf_\alpha)(x')\right)_{1\leq\gamma\leq n}
\nonumber\\[4pt]
&\qquad
=\frac{\mu+\lambda}{3\mu+\lambda}
\left((\partial_1f_n)(x'),(\partial_2f_n)(x'),\dots,(\partial_{n-1}f_n)(x'),
\sum_{s=1}^{n-1}(\partial_sf_s)(x')\right).
\end{align}
We are left with computing the second sum in \eqref{mdiab} in the current setting.
To do so, use \eqref{Lame-6} to write for each $\beta,\gamma,r\in\{1,\dots,n\}$,
\begin{align}\label{n-uwuqi}
\partial_r E_{\gamma\beta}(x)&=\frac{-1}{2\mu(2\mu+\lambda)\omega_{n-1}}
\Big[-(3\mu+\lambda)\delta_{\gamma\beta}\frac{x_r}{|x|^n}
+(\mu+\lambda)\delta_{r\gamma}\frac{x_\beta}{|x|^n}
\\[10pt]
&\hskip 1.50in
+(\mu+\lambda)\delta_{r\beta}\frac{x_\gamma}{|x|^n}
-n(\mu+\lambda)\frac{x_\gamma x_r x_\beta}{|x|^{n+2}}\Big],
\nonumber
\end{align}
for every $x\in{\mathbb{R}}^n\setminus\{0\}$. Then a direct computation based on
\eqref{BUIg-17XX} and \eqref{n-uwuqi} yields
\begin{align}\label{mdibdG}
&2\sum\limits_{s=1}^{n-1}a^{\beta\alpha}_{rs}
\lim_{\varepsilon\to 0^{+}}
\int\limits_{\substack{y'\in{\mathbb{R}}^{n-1}\\ |x'-y'|>\varepsilon}}
(\partial_r E_{\gamma\beta})(x'-y',0)(\partial_sf_\alpha)(y')\,dy'
\nonumber\\[4pt]
&\qquad
=\frac{4\mu}{(3\mu+\lambda)\omega_{n-1}}\sum\limits_{s=1}^{n-1}
\lim_{\varepsilon\to 0^{+}}
\int\limits_{\substack{y'\in{\mathbb{R}}^{n-1}\\ |x'-y'|>\varepsilon}}
\frac{x_s-y_s}{|x'-y'|^n}\,(\partial_sf_\gamma)(y')\,dy'
\\[4pt]
&\qquad\quad
+\frac{2n(\mu+\lambda)}{(3\mu+\lambda)\omega_{n-1}}\sum\limits_{s=1}^{n-1}
\lim_{\varepsilon\to 0^{+}}
\int\limits_{\substack{y'\in{\mathbb{R}}^{n-1}\\ |x'-y'|>\varepsilon}}
\frac{(x_s-y_s)(x'-y',0)_\alpha(x'-y',0)_\gamma}{|x'-y'|^{n+2}}\,
(\partial_sf_\alpha)(y')\,dy'
\nonumber
\end{align}
for each $\gamma\in\{1,\dots,n\}$ and a.e. $x'=(x_1,\dots,x_{n-1})\in{\mathbb{R}}^{n-1}$.
Now \eqref{mdiab-Lame} follows from \eqref{mdiab}, \eqref{mdiab-csf}, and \eqref{mdibdG}.
\end{proof}

\section{Further results}
\setcounter{equation}{0}
\label{S-4}

As is apparent from the statement of Theorem~\ref{V-Naa.11}, the solvability of the
Regularity problem $(R_p)$ (formulated in \eqref{Dir-BVvku}) interfaces tightly with
the infinitesimal generator of the Poisson semigroup considered in Theorem~\ref{VCXga}.
In this regard, we have already seen in \eqref{eq:RPWP} that $(R_p)$ is well-posed
for any system $L$ as in \eqref{L-def}-\eqref{L-ell.X} with the property that
\eqref{eq:Asd.FFF} holds. Here we wish to further improve on this result, by considering
a weaker condition in lieu of \eqref{eq:Asd.FFF} which we now describe. Let
\begin{equation}\label{RRR-UpHs.33}
{\mathbb{R}}^{n}_{-}:=\big\{x=(x',x_n)\in
{\mathbb{R}}^{n}={\mathbb{R}}^{n-1}\times{\mathbb{R}}:\,x_n<0\big\}
\end{equation}
denote the lower-half space. Given a system $L$ as in \eqref{L-def}-\eqref{L-ell.X},
let $P^L$ be the Poisson kernel for $L$ in $\mathbb{R}^{n}_{+}$ as in Theorem~\ref{kkjbhV}.
Then, for each $j\in\{1,\dots,n-1\}$ introduce the {\tt conjugate} {\tt Poisson} {\tt kernels}
\begin{equation}\label{eq:QWd7}
Q^L_j(x'):=x_jP^L(x'),\qquad\forall\,x'=(x_1,\dots,x_{n-1})\in{\mathbb{R}}^{n-1},
\end{equation}
and, as usual, set $(Q^L_j)_t(x'):=t^{1-n}Q^L_j(x'/t)$ for each $x'\in{\mathbb{R}}^{n-1}$ and $t>0$.
Finally, for each $j\in\{1,\dots,n-1\}$ define the {\tt conjugate} {\tt kernel} {\tt function}
$K_j\in{\mathscr{C}}^\infty({\mathbb{R}}^n_{+}\cup{\mathbb{R}}^n_{-})$ by extending
the mapping ${\mathbb{R}}^n_{+}\ni(x',t)\mapsto (Q^L_j)_t(x')$ to ${\mathbb{R}}^n_{-}$ so that
the resulting function is odd in ${\mathbb{R}}^n_{+}\cup{\mathbb{R}}^n_{-}$. That is,
for each $j\in\{1,\dots,n-1\}$ we set
\begin{equation}\label{eq:QWd8}
K_j(x):=
\left\{
\begin{array}{ll}
(Q^L_j)_{t}(x') & \mbox{ for }\,x=(x',t)\in{\mathbb{R}}^n_{+},
\\[8pt]
-(Q^L_j)_{-t}(-x') & \mbox{ for }\,x=(x',t)\in{\mathbb{R}}^n_{-}.
\end{array}
\right.
\end{equation}
Then, in place of \eqref{eq:Asd.FFF}, we shall impose the regularity condition that
\begin{equation}\label{eq:QWd9}
\parbox{11.50cm}
{each $K_j$ extends to a function in ${\mathscr{C}}^\infty({\mathbb{R}}^n\setminus\{0\})$
(in which scenario, we will continue to use the same symbol $K_j$ for the said extension).}
\end{equation}
Note that $K_j$ is smooth both in ${\mathbb{R}}^n_{+}$ and ${\mathbb{R}}^n_{-}$,
so what \eqref{eq:QWd9} really demands is that these restrictions smoothly transit across
$\partial{\mathbb{R}}^n_{+}\setminus\{0\}$. As the examples below illustrate, this
condition is automatically satisfied for some natural classes of operators.

\vskip 0.08in
\noindent{\bf Example~1: Scalar second order elliptic equations.}
Assume that
\begin{equation}\label{YUjhv-753}
A=(a_{rs})_{r,s}\in{\mathbb{C}}^{n\times n}
\,\,\mbox{ satisfies }\,\,
\inf_{\xi=(\xi_j)_{{}_{1\leq j\leq n}}\in S^{n-1}}{\rm Re}\,\bigl[a_{rs}\xi_r\xi_s\bigr]>0,
\end{equation}
and denote by $A_{\rm sym}:=\tfrac{1}{2}\big(A+A^\top\big)$ the symmetric part of $A$.
From \cite{EE-MMMM} we know that the Poisson kernel in ${\mathbb{R}}^n_{+}$ for the operator
\begin{equation}\label{eq:LSCA}
L={\rm div}A\nabla
\end{equation}
is the function
\begin{equation}\label{Uahab8a-hab}
P^L(x')=\frac{2}{\omega_{n-1}\sqrt{{\rm det}\,(A_{\rm sym})}}
\frac{1}{\big[\big((A_{\rm sym})^{-1}(x',1)\big)\,\cdot\,(x',1)\big]^{\frac{n}{2}}},
\qquad\forall\,x'\in{\mathbb{R}}^{n-1}.
\end{equation}
It is reassuring to observe that, in the case when $A=I$ (i.e., when $L$ is the Laplacian),
\eqref{Uahab8a-hab} reduces precisely to the classical harmonic Poisson kernel
\begin{eqnarray}\label{Uah-TTT}
P^{\Delta}(x')=\frac{2}{\omega_{n-1}}\frac{1}{\big(1+|x'|^2\big)^{\frac{n}{2}}},
\qquad\forall\,x'\in{\mathbb{R}}^{n-1}.
\end{eqnarray}

Starting with \eqref{Uahab8a-hab}, a simple calculation reveals that the conjugate
kernel functions for the operator $L={\rm div}A\nabla$, originally defined
in ${\mathbb{R}}^n_{+}\cup{\mathbb{R}}^n_{-}$, extend to ${\mathbb{R}}^{n}\setminus\{0\}$
as ${\mathscr{C}}^\infty$ functions by the formula
\begin{equation}\label{UahJBB}
K_j(x)=\frac{2}{\omega_{n-1}\sqrt{{\rm det}\,(A_{\rm sym})}}
\frac{x_j}{\big[\big((A_{\rm sym})^{-1}x\big)\,\cdot x\big]^{\frac{n}{2}}},
\qquad\forall\,x=(x_1,\dots,x_n)\in{\mathbb{R}}^{n}\setminus\{0\},
\end{equation}
for each $j\in\{1,\dots,n\}$. As such, condition \eqref{eq:QWd9} is satisfied in the
class of operators $L$ as in \eqref{eq:LSCA} with coefficients as in \eqref{YUjhv-753}.

\vskip 0.08in
\noindent{\bf Example~2: The Lam\'e system of elasticity.}
Consider the Lam\'e operator in ${\mathbb{R}}^n$, written as in \eqref{TYd-YG-76g}
for Lam\'e moduli satisfying \eqref{Yfhv-8yg}. Then, as shown in \cite{EE-MMMM},
its Poisson kernel in ${\mathbb{R}}^n_{+}$ is the matrix-valued function
$P^{\,{\rm Lame}}=\big(P^{\,{\rm Lame}}_{\alpha\beta}\big)_{1\leq\alpha,\beta\leq n}:
{\mathbb{R}}^{n-1}\to{\mathbb{R}}^{n\times n}$ whose entries are given
for each $\alpha,\beta\in\{1,...,n\}$ and $x'\in{\mathbb{R}}^{n-1}$ by
\begin{equation}\label{Lame-9RR.L}
P^{\,{\rm Lame}}_{\alpha\beta}(x')=\frac{4\mu}{3\mu+\lambda}
\frac{\delta_{\alpha\beta}}{\omega_{n-1}}\frac{1}{(1+|x'|^2)^{\frac{n}{2}}}
+\frac{\mu+\lambda}{3\mu+\lambda}\frac{2n}{\omega_{n-1}}
\frac{(x',1)_\alpha(x',1)_\beta}{(1+|x'|^2)^{\frac{n+2}{2}}}.
\end{equation}
Starting this time from \eqref{Lame-9RR.L}, via a direct calculation it follows that the conjugate
kernel functions for the Lam\'e operator extend from ${\mathbb{R}}^n_{+}\cup{\mathbb{R}}^n_{-}$
to ${\mathbb{R}}^{n}\setminus\{0\}$ as ${\mathscr{C}}^\infty$ functions by the formula
\begin{align}\label{Uamndys}
&K_j(x)=\left(\frac{4\mu}{3\mu+\lambda}
\frac{\delta_{\alpha\beta}}{\omega_{n-1}}\frac{x_j}{|x|^n}
+\frac{\mu+\lambda}{3\mu+\lambda}\frac{2n}{\omega_{n-1}}
\frac{x_jx_\alpha x_\beta}{|x|^{n+2}}\right)_{1\leq\alpha,\beta\leq n}
\\[8pt]
&\forall\,x=(x_1,\dots,x_n)\in{\mathbb{R}}^{n}\setminus\{0\},
\qquad\forall\,j\in\{1,\dots,n-1\}.
\nonumber
\end{align}
Hence, condition \eqref{eq:QWd9} is satisfied for the Lam\'e operator \eqref{TYd-YG-76g}
in ${\mathbb{R}}^n$ with Lam\'e moduli as in \eqref{Yfhv-8yg}.

\vskip 0.10in

Below we show that, as advertised in the preamble to this section,
the vanishing conormal condition \eqref{eq:Asd.FFF} is stronger than
the regularity condition \eqref{eq:QWd9}.

\begin{proposition}\label{V-Naa.11.PPP}
Let $L$ be a strongly elliptic, second-order, homogeneous, $M\times M$ system,
with constant complex coefficients.
Denote by $E=\big(E_{\gamma\beta}\big)_{1\leq\gamma,\beta\leq M}$
the canonical fundamental solution for $L$ from Theorem~\ref{FS-prop}, and
make the additional assumption that $L$ may be written as $L=\partial_r A_{rs}\partial_s$
for some family of complex matrices
$A_{rs}=\bigl(a_{rs}^{\,\alpha\beta}\bigr)_{1\leq\alpha,\beta\leq M}$,
$1\leq r,s\leq n$, satisfying
\begin{equation}\label{eq:Asd.FFF.PPP}
\begin{array}{l}
a^{\beta\alpha}_{rn}\big(\partial_r E_{\gamma\beta}\big)(x',0)=0
\,\,\mbox{ for each }\,\,x'\in{\mathbb{R}}^{n-1}\setminus\{0\},
\\[6pt]
\mbox{for every fixed multi-indices }\,\,\alpha,\gamma\in\{1,\dots,M\}.
\end{array}
\end{equation}
Then condition \eqref{eq:QWd9} holds.
\end{proposition}

\begin{proof}
Define
\begin{equation}\label{eq:Kytrf.a}
\begin{array}{c}
{\mathcal{O}}_1:=\big\{x=(x',x_n)\in{\mathbb{R}}^n:\,x_n\not=0\big\},
\\[6pt]
{\mathcal{O}}_2:=\big\{x=(x',x_n)\in{\mathbb{R}}^n:\,x'\not=0'\big\},
\end{array}
\end{equation}
and, for each $j\in\{1,\dots,n-1\}$, consider the functions
\begin{equation}\label{eq:Kytrf}
\begin{array}{c}
K^{(1)}_j:{\mathcal{O}}_1\longrightarrow{\mathbb{C}}^{M\times M}
\\[6pt]
K^{(1)}_j(x):=\displaystyle\Big(\frac{x_j}{x_n}
2a^{\beta\alpha}_{rn}\big(\partial_r E_{\gamma\beta}\big)(x)\Big)_{1\leq\gamma,\alpha\leq M},
\qquad\forall\,x=(x_k)_{1\leq k\leq n}\in{\mathcal{O}}_1,
\end{array}
\end{equation}
and
\begin{equation}\label{eq:Kytrf.2}
\begin{array}{c}
K^{(2)}_j:{\mathcal{O}}_2\longrightarrow{\mathbb{C}}^{M\times M}
\\[6pt]
K^{(2)}_j(x):=\left(\displaystyle\int_0^1 x_j
2a^{\beta\alpha}_{rn}\big(\partial_n\partial_r E_{\gamma\beta}\big)(x',\theta x_n)\,d\theta
\right)_{1\leq\gamma,\alpha\leq M},
\qquad\forall\,x=(x',x_n)\in{\mathcal{O}}_2.
\end{array}
\end{equation}
Then ${\mathcal{O}}_1,{\mathcal{O}}_2$ are open subsets of ${\mathbb{R}}^n$ satisfying
\begin{equation}\label{eq:itrR6g}
{\mathcal{O}}_1\cup{\mathcal{O}}_2={\mathbb{R}}^n\setminus\{0\},\quad
-{\mathcal{O}}_1={\mathcal{O}}_1,\quad
-{\mathcal{O}}_2={\mathcal{O}}_2,
\end{equation}
while Theorem~\ref{FS-prop} guarantees that $K^{(1)}_j\in{\mathscr{C}}^\infty({\mathcal{O}}_1)$,
$K^{(2)}_j\in{\mathscr{C}}^\infty({\mathcal{O}}_2)$, and $K^{(1)}_j,K^{(2)}_j$ are odd.
Furthermore, thanks to \eqref{eq:Asd.FFF.PPP} and the Fundamental Theorem of Calculus,
we have $K^{(2)}_j=K^{(1)}_j$ on ${\mathcal{O}}_1\cap{\mathcal{O}}_2$. Collectively,
these observations imply that, for each $j\in\{1,\dots,n-1\}$, the function
\begin{equation}\label{eq:Phgf7bt5}
K_j:{\mathbb{R}}^n\setminus\{0\}\longrightarrow{\mathbb{C}}^{M\times M},
\qquad
K_j(x):=\left\{
\begin{array}{ll}
K^{(1)}_j(x) & \mbox{ for }\,\,x\in{\mathcal{O}}_1,
\\[6pt]
K^{(2)}_j(x) & \mbox{ for }\,\,x\in{\mathcal{O}}_2,
\end{array}
\right.
\end{equation}
satisfies
\begin{equation}\label{eq:jygtt5}
K_j\in{\mathscr{C}}^\infty({\mathbb{R}}^n\setminus\{0\})
\,\,\mbox{ and }\,\,K_j\,\,\mbox{ is odd}.
\end{equation}
In addition, having fixed $j\in\{1,\dots,n-1\}$,
for each $x'\in{\mathbb{R}}^{n-1}$ and $t>0$ we have
\begin{align}\label{y5556}
(Q^L_j)_t(x') &=t^{1-n}Q^L_j(x'/t)=t^{1-n}\frac{x_j}{t}P^L(x'/t)
\nonumber\\[4pt]
&=\Big(t^{1-n}\frac{x_j}{t}2a^{\beta\alpha}_{rn}(\partial_r E_{\gamma\beta})(x'/t,1)
\Big)_{1\leq\gamma,\alpha\leq M}
\nonumber\\[4pt]
&=\Big(\frac{x_j}{t}2a^{\beta\alpha}_{rn}(\partial_r E_{\gamma\beta})(x',t)
\Big)_{1\leq\gamma,\alpha\leq M}
\nonumber\\[4pt]
&=K^{(1)}_j(x',t)=K_j(x',t),
\end{align}
by \eqref{eq:QWd7}, \eqref{yncuds.AAA}, part {\it (4)} in Theorem~\ref{FS-prop}, \eqref{eq:Kytrf},
and \eqref{eq:Phgf7bt5}. In turn, \eqref{eq:jygtt5} and \eqref{y5556} show that
\eqref{eq:QWd9} holds, finishing the proof.
\end{proof}

In light of Proposition~\ref{V-Naa.11.PPP}, the well-posedness result in the next
proposition is an improvement over \eqref{eq:RPWP}.

\begin{proposition}\label{jdsgau}
Let $L$ be a system with complex coefficients as in \eqref{L-def}-\eqref{L-ell.X} and
consider the Poisson kernel $P^L$ for $L$ in $\mathbb{R}^{n}_{+}$ as in Theorem~\ref{kkjbhV}.
If condition \eqref{eq:QWd9} holds, then for each $p\in(1,\infty)$ the Regularity problem
$(R_p)$ for $L$ in ${\mathbb{R}}^n_{+}$, as formulated in \eqref{Dir-BVvku}, is well-posed.
\end{proposition}

\begin{proof}
Fix $p\in(1,\infty)$ and pick an arbitrary $f\in L^p_1({\mathbb{R}}^{n-1})$.
Then we claim that the function $u(x',t):=(P_t^L\ast f)(x')$ for $(x',t)\in{\mathbb{R}}^n_{+}$
is a solution of $(R_p)$ for boundary datum $f$. As before, the only thing to verify is
\eqref{eq:hds65}. With this goal in mind, for each $x=(x',t)\in{\mathbb{R}}^n_{+}$ we write
\begin{align}\label{urree}
(\partial_n u)(x) &=\frac{d}{dt}\big[(P_t^L\ast f)(x')\big]
\nonumber\\[4pt]
&=\frac{d}{dt}\left[\int_{{\mathbb{R}}^{n-1}}t^{1-n}P^L\big((x'-y')/t\big)f(y')\,dy'\right]
\nonumber\\[4pt]
&=\frac{d}{dt}\left[\int_{{\mathbb{R}}^{n-1}}P^L(z')f(x'-tz')\,dz'\right]
\nonumber\\[4pt]
&=-\sum_{j=1}^{n-1}\int_{{\mathbb{R}}^{n-1}}z_jP^L(z')(\partial_j f)(x'-tz')\,dz'
\nonumber\\[4pt]
&=-\sum_{j=1}^{n-1}\int_{{\mathbb{R}}^{n-1}}Q^L_j(z')(\partial_j f)(x'-tz')\,dz'
\nonumber\\[4pt]
&=-\sum_{j=1}^{n-1}\int_{{\mathbb{R}}^{n-1}}t^{1-n}Q^L_j\big((x'-y')/t\big)(\partial_j f)(y')\,dy'
\nonumber\\[4pt]
&=-\sum_{j=1}^{n-1}\int_{{\mathbb{R}}^{n-1}}\big(Q^L_j\big)_t(x'-y')(\partial_j f)(y')\,dy'
\nonumber\\[4pt]
&=-\sum_{j=1}^{n-1}\int_{{\mathbb{R}}^{n-1}}K_j\big(x-(y',0)\big)(\partial_j f)(y')\,dy'.
\end{align}
Recall that \eqref{eq:QWd9} is assumed to hold.
Hence, if for each $j\in\{1,\dots,n-1\}$ we denote by ${\mathcal{T}}_j$ the integral operator
associated as in \eqref{T-layer}  to the kernel $K_j\in{\mathscr{C}}^\infty({\mathbb{R}}^n\setminus\{0\})$,
then \eqref{urree} becomes
\begin{eqnarray}\label{bsyus56.iii}
\partial_n u=-\sum\limits_{j=1}^{n-1}{\mathcal{T}}_j(\partial_jf)
\,\,\,\mbox{ in }\,\,\,{\mathbb{R}}^n_{+}.
\end{eqnarray}
Given that the kernels $K_j$ satisfy all conditions in \eqref{ker} (as is evident from
\eqref{eq:QWd8} and \eqref{eq:QWd9}), Theorem~\ref{Main-T2} then implies
that ${\mathcal{N}}(\partial_n u)\in L^p({\mathbb{R}}^{n-1})$ and
$\big\|{\mathcal{N}}(\partial_n u)\big\|_{L^p({\mathbb{R}}^{n-1})}
\leq C\|f\|_{L^p_1({\mathbb{R}}^{n-1})}$, for some $C\in(0,\infty)$ independent of $f$.
That similar properties also hold for partial derivatives $\partial_k u$ with $k\in\{1,\dots,n-1\}$
is seen directly from \eqref{eq:efgg} and part {\it (6)} in Theorem~\ref{kkjbhV}.
This proves that $(R_p)$ has a solution satisfying natural estimates. Since uniqueness
is contained in Theorem~\ref{Theorem-NiceDP}, it follows that $(R_p)$ is well-posed.
\end{proof}

Having established Proposition~\ref{jdsgau}, we may now proceed to state and prove the
main result in this section.

\begin{theorem}\label{V-Naa.11.CCC}
Assume that $L$ is a strongly elliptic, second-order, homogeneous, $M\times M$ system,
with constant complex coefficients, with the property that \eqref{eq:QWd9} holds.
Fix $p\in(1,\infty)$ and consider the $C_0$-semigroup $T=\{T(t)\}_{t\geq 0}$ on
$L^p({\mathbb{R}}^{n-1})$ associated with $L$ as in Theorem~\ref{VCXga}.
Finally, let ${\mathbf{A}}$ be the infinitesimal generator of $T$, with domain $D({\mathbf{A}})$.
Then
\begin{equation}\label{eq:tfc.1-new.CCC}
D({\mathbf{A}})=L^p_1({\mathbb{R}}^{n-1})
\end{equation}
and for each $f\in L^p_1({\mathbb{R}}^{n-1})$ one has
\begin{align}\label{mdiab.CCC}
{\mathbf{A}}f(x') &=-\sum\limits_{j=1}^{n-1}\frac{1}{2i}\widehat{K_j}(-{\mathbf e}_n)(\partial_jf)(x')
\nonumber\\[4pt]
&\quad -\sum\limits_{j=1}^{n-1}\lim_{\varepsilon\to 0^{+}}
\int\limits_{\substack{y'\in{\mathbb{R}}^{n-1}\\ |x'-y'|>\varepsilon}}
K_j(x'-y',0)(\partial_jf)(y')\,dy'
\end{align}
for a.e. $x'\in{\mathbb{R}}^{n-1}$.
\end{theorem}

\begin{proof}
The identification of the domain of the infinitesimal generator from
\eqref{eq:tfc.1-new.CCC} is a consequence of \eqref{eq:tfc.2BB} and Proposition~\ref{jdsgau},
while formula \eqref{mdiab.CCC} follows from \eqref{mJBVV}, \eqref{bsyus56.iii}, and \eqref{main-jump}.
\end{proof}

\section{Iterations of the infinitesimal generator}
\setcounter{equation}{0}
\label{S-5}

Having identified the infinitesimal generator ${\mathbf{A}}$ of the Poisson semigroup
from Theorem~\ref{VCXga}, here the goal is to describe ${\mathbf{A}}^k$ for arbitrary
$k\in{\mathbb{N}}$. Interestingly, this is intimately connected with a higher order
Regularity problem for the system in question, which we first describe.

For each $p\in(1,\infty)$ and $k\in\mathbb{N}_0$ denote by $L^p_k(\mathbb{R}^{n-1})$
the classical Sobolev space of order $k$ in $\mathbb{R}^{n-1}$, consisting of
functions from $L^p(\mathbb{R}^{n-1})$ whose distributional derivatives up to order
$k$ are in $L^p(\mathbb{R}^{n-1})$. This becomes a Banach space when equipped with
the natural norm
\begin{equation}\label{Lnanb}
\|f\|_{L^p_k(\mathbb{R}^{n-1})}:=\sum_{|\alpha'|\leq k}\|\partial^{\alpha'}\! f\|_{L^p(\mathbb{R}^{n-1})},
\qquad\forall\,f\in L^p_k(\mathbb{R}^{n-1}).
\end{equation}

Assume $L$ is a system as in \eqref{L-def}-\eqref{L-ell.X}, and fix some $p\in(1,\infty)$.
Given $k\in{\mathbb{N}}_0$, we formulate the $k$-th order Regularity problem
for $L$ in $\mathbb{R}^{n}_+$ as follows:
\begin{equation}\label{Dir-BVP-p:l}
(R^k_p)\,\,\left\{
\begin{array}{l}
u\in{\mathscr{C}}^\infty(\mathbb{R}^{n}_{+}),
\\[4pt]
Lu=0\,\,\mbox{ in }\,\,\mathbb{R}^{n}_{+},
\\[8pt]
{\mathcal{N}}(\nabla^\ell u)\in L^p(\mathbb{R}^{n-1})
\,\,\mbox{ for }\,\,\ell\in\{0,1,...,k\},
\\[4pt]
u\bigl|_{\partial\mathbb{R}^{n}_{+}}^{{}^{n.t.}}=f\in L^p_k(\mathbb{R}^{n-1}).
\end{array}
\right.
\end{equation}
Of course, $(R^k_p)$ reduces to $(D_p)$ when $k=0$ and to $(R_p)$ when $k=1$.
In regard to \eqref{Dir-BVP-p:l} we note the following well-posedness result,
recently proved in \cite{EE-MMMM}.

\begin{theorem}\label{them:Dir-l}
Let $L$ be a strongly elliptic, second-order, homogeneous, $M\times M$ system,
with constant complex coefficients, with the property that
\eqref{Ea4-fCii-n3}-\eqref{Ea4-fCii-n2B} hold.

Then for each $p\in(1,\infty)$ and $k\in{\mathbb{N}}_0$
the $k$-th order Regularity problem $(R^k_p)$ for $L$ in $\mathbb{R}^{n}_+$
has a unique solution, which is actually given by
\begin{equation}\label{eqn-Dir-l:u}
u(x',t)=(P^L_t*f)(x'),\qquad\forall\,(x',t)\in{\mathbb{R}}^n_{+},
\end{equation}
where $P^L$ is the Poisson kernel for $L$ in $\mathbb{R}^{n}_+$ from
Theorem~\ref{kkjbhV}. Furthermore, there exists a
constant $C=C(n,p,L,k)\in(0,\infty)$ with the property that
\begin{equation}\label{Dir-BVP-p2F.ii:l}
\sum_{\ell=0}^k \big\|{\mathcal{N}}(\nabla^\ell u)\big\|_{L^p(\mathbb{R}^{n-1})}
\leq C\|f\|_{L^p_k(\mathbb{R}^{n-1})}.
\end{equation}
\end{theorem}

It has also been noted in \cite{EE-MMMM} that \eqref{Ea4-fCii-n3}-\eqref{Ea4-fCii-n2B} hold
both in the class of scalar elliptic operators with complex constant coefficients of the form
\eqref{eq:Kbag}, as well as for the Lam\'e system \eqref{TYd-YG-76g} with Lam\'e moduli as
in \eqref{Yfhv-8yg}. In particular, the well-posedness result described in Theorem~\ref{them:Dir-l}
is valid for these categories of operators.

The main result in this section is the theorem below, establishing a direct link between
the domain of powers of the infinitesimal generator of the Poisson semigroup from
Theorem~\ref{VCXga} and the solvability of the higher order Regularity problem \eqref{Dir-BVP-p:l}.

\begin{theorem}\label{V-Naa.11.HOR}
Let $L$ be a strongly elliptic, second-order, homogeneous, $M\times M$ system,
with constant complex coefficients. Fix some $p\in(1,\infty)$ and consider the
$C_0$-semigroup $T=\{T(t)\}_{t\geq 0}$ on $L^p({\mathbb{R}}^{n-1})$ associated
with $L$ as in Theorem~\ref{VCXga}. Denote by ${\mathbf{A}}$ the infinitesimal
generator of $T$ and pick some $k\in{\mathbb{N}}$. Then
\begin{equation}\label{eq:tfc.1-new.RRR-LAC}
\mbox{$D({\mathbf{A}}^k)$ is a dense linear subspace of $L^p_k({\mathbb{R}}^{n-1})$}
\end{equation}
and, in fact,
\begin{equation}\label{eq:tfc.1-new.HOR}
D({\mathbf{A}}^k)=\big\{f\in L^p_k({\mathbb{R}}^{n-1}):\,
\mbox{$(R^k_p)$ with boundary datum $f$ is solvable}\big\}.
\end{equation}
In particular,
\begin{equation}\label{eq:tfc.2BB.HOR}
D({\mathbf{A}}^k)=L^p_k({\mathbb{R}}^{n-1})
\,\Longleftrightarrow\,\,\mbox{$(R^k_p)$ is solvable
for arbitrary data in $L^p_k({\mathbb{R}}^{n-1})$}.
\end{equation}
Moreover, given any $f\in D({\mathbf{A}}^k)$, if $u$ solves $(R^k_p)$ with boundary
datum $f$ then
\begin{equation}\label{eq:hJB.HOR}
(\partial^k_n u)\big|^{{}^{\rm n.t.}}_{\partial{\mathbb{R}}^n_{+}}\,\,
\mbox{ exists a.e.~in ${\mathbb{R}}^{n-1}$}\,\,\mbox{ and }\,\,
{\mathbf{A}}^k f=(\partial^k_n u)\big|^{{}^{\rm n.t.}}_{\partial{\mathbb{R}}^n_{+}}.
\end{equation}

As a corollary, one has $D({\mathbf{A}}^k)=L^p_k({\mathbb{R}}^{n-1})$ whenever
\eqref{Ea4-fCii-n3}-\eqref{Ea4-fCii-n2B} hold, which is the case both for
scalar operators of the form \eqref{eq:Kbag} {\rm (}with constant complex coefficients
satisfying \eqref{sec-or-aEEE}{\rm )} and for the Lam\'e system \eqref{TYd-YG-76g}
with Lam\'e moduli as in \eqref{Yfhv-8yg}.
\end{theorem}

\begin{proof}
We start by proving the left-to-right inclusion in \eqref{eq:tfc.1-new.HOR}.
To this end, let $f\in D({\mathbf{A}}^k)$ be an arbitrary, fixed, function.
In particular, $f\in L^p({\mathbb{R}}^{n-1})$ and we introduce
\begin{equation}\label{eq:FDww.YY}
u(x',t):=\big(T(t)f\big)(x')=(P^L_t\ast f)(x')\,\,\,\mbox{ for }\,\,\,(x',t)\in{\mathbb{R}}^n_{+}.
\end{equation}
As before, this choice entails
\begin{equation}\label{KJHab.ttr}
u\in{\mathscr{C}}^\infty({\mathbb{R}}^n_{+}),\quad Lu=0\,\,\mbox{ in }\,\,{\mathbb{R}}^n_{+},\quad
{\mathcal{N}}_\kappa u\in L^p({\mathbb{R}}^{n-1}),\quad
u\big|^{{}^{\rm n.t.}}_{\partial{\mathbb{R}}^n_{+}}=f\,\,\mbox{ a.e.~in }\,\,{\mathbb{R}}^{n-1}.
\end{equation}
Then the identity in \eqref{eq:ADCV.ttt.79} yields that for each $\ell\in\{1,\dots,k\}$
\begin{align}\label{eq:FDww.YY2}
(\partial^\ell_n u)(x',t) &=
\frac{d^\ell}{dt^\ell}\big[u(x',t)\big]=\big(T(t){\mathbf{A}}^\ell f\big)(x')
\nonumber\\[4pt]
&=(P^L_t\ast{\mathbf{A}}^\ell f)(x')\,\,\,\mbox{ for }\,\,\,(x',t)\in{\mathbb{R}}^n_{+}.
\end{align}
Hence, since for each $\ell\in\{1,\dots,k\}$ we have ${\mathbf{A}}^\ell f\in L^p({\mathbb{R}}^{n-1})$,
this implies
\begin{equation}\label{eq:jufds.22Tuj}
{\mathcal{N}}_\kappa(\partial^\ell_n u)\in L^p({\mathbb{R}}^{n-1})\,\,\,
\mbox{for each }\,\,\ell\in\{1,\dots,k\}.
\end{equation}

It is implicit in \eqref{eq:iFD} that for each function satisfying
\begin{equation}\label{jxfrA.UTE}
\begin{array}{c}
v:{\mathbb{R}}^n_{+}\longrightarrow{\mathbb{C}}^M,\qquad
v\in{\mathscr{C}}^\infty(\overline{{\mathbb{R}}^n_{+}}),
\qquad Lv=0\,\,\mbox{ in }\,\,{\mathbb{R}}^n_{+},
\\[4pt]
{\mathcal{N}}_\kappa(\nabla^\ell v)\in L^p({\mathbb{R}}^{n-1})
\,\,\mbox{ and }\,\,\lim\limits_{t\to\infty}\Big[t^\ell\big|(\nabla^\ell v)(x',t)\big|\Big]=0
\,\,\mbox{ for }\,\,\ell\in\{1,2,3\},
\end{array}
\end{equation}
uniformly for $x'$ in compact subsets of ${\mathbb{R}}^{n-1}$,
one has
\begin{align}\label{eq:iFD.uyt}
\|{\mathcal{N}}_{\kappa}(\partial_j v)\|_{L^p({\mathbb{R}}^{n-1})}
\leq C\|{\mathcal{N}}_{\kappa}(\partial_n v)\|_{L^p({\mathbb{R}}^{n-1})},
\quad\forall\,j\in\{1,\dots,n\},
\end{align}
where the constant $C\in(0,\infty)$ is independent of $v$.
To proceed, fix $\varepsilon>0$ and, as in \eqref{jxfrA-2}, define
\begin{equation}\label{jxfrA-2nnn}
u_\varepsilon:{\mathbb{R}}^n_{+}\to {\mathbb{C}}^M,\quad
u_\varepsilon(x',t):=u(x',t+\varepsilon),\quad\forall\,(x',t)\in{\mathbb{R}}^n_{+}.
\end{equation}
Then for each $\ell\in{\mathbb{N}}_0$ we have
(cf. \eqref{jxfrA}, \eqref{elsiWth.UU3b}, \eqref{eq:jBBV})
\begin{equation}\label{jxfrAeded}
\begin{array}{c}
u_\varepsilon\in{\mathscr{C}}^\infty(\overline{{\mathbb{R}}^n_{+}}),
\quad Lu_\varepsilon=0\,\,\mbox{ in }\,\,{\mathbb{R}}^n_{+},\quad
{\mathcal{N}}_\kappa(\nabla^\ell u_\varepsilon)\in L^p({\mathbb{R}}^{n-1}),
\\[4pt]
\mbox{and }\,\,\lim\limits_{t\to\infty}
\Big[t^\ell\big|(\nabla^\ell u_\varepsilon)(x',t)\big|\Big]=0,
\,\,\mbox{ uniformly for $x'\in{\mathbb{R}}^{n-1}$}.
\end{array}
\end{equation}
Based on iterations of \eqref{jxfrA.UTE}-\eqref{eq:iFD.uyt}, \eqref{jxfrAeded}
we deduce that for each $\ell\in{\mathbb{N}}_0$
\begin{align}\label{eq:iFD.uyt.ad.YY}
\|{\mathcal{N}}_{\kappa}(\nabla^\ell u_\varepsilon)\|_{L^p({\mathbb{R}}^{n-1})}
\leq C\|{\mathcal{N}}_{\kappa}(\partial^\ell_n u_\varepsilon)\|_{L^p({\mathbb{R}}^{n-1})},
\end{align}
where the constant $C\in(0,\infty)$ does not depend on $u$ or $\varepsilon$.
In turn, from \eqref{eq:iFD.uyt.ad.YY}, \eqref{eq:jufds.22Tuj},
and the third condition in \eqref{KJHab.ttr}, we obtain
that, for the same constant $C$ as above,
\begin{align}\label{eq:iFD.uyt.ad}
\|{\mathcal{N}}_{\kappa}(\nabla^\ell u_\varepsilon)\|_{L^p({\mathbb{R}}^{n-1})}
\leq C\|{\mathcal{N}}_{\kappa}(\partial^\ell_n u)\|_{L^p({\mathbb{R}}^{n-1})}<\infty,
\qquad\forall\,\ell\in\{0,1,\dots,k\}.
\end{align}
With this in hand and arguing as in \eqref{exTGF-DFF}-\eqref{eq:iFD.556} we then conclude
that
\begin{equation}\label{eq:jufds.26.TRnb}
{\mathcal{N}}_\kappa(\nabla^\ell u)\in L^p({\mathbb{R}}^{n-1})\,\,\,
\mbox{for each }\,\,\ell\in\{0,1,\dots,k\}.
\end{equation}
As a consequence of \eqref{eq:jufds.26.TRnb} and
Theorem~\ref{tuFatou} we have
\begin{equation}\label{eufds.26fff.ii}
(\nabla^\ell u)\big|^{{}^{\rm n.t.}}_{\partial{\mathbb{R}}^n_{+}}
\,\,\mbox{ exists and belongs to }\,\,L^p({\mathbb{R}}^{n-1}),
\,\,\mbox{for each }\,\,\ell\in\{0,1,\dots,k\}.
\end{equation}
Next, an inspection of the manner in which \eqref{eq:IIIb} has been proved reveals that
\begin{equation}\label{eq:IIIb.yfr}
\begin{array}{c}
\partial_j\Big(v\big|^{{}^{\rm n.t.}}_{\partial{\mathbb{R}}^n_{+}}\Big)
=(\partial_jv)\big|^{{}^{\rm n.t.}}_{\partial{\mathbb{R}}^n_{+}}
\,\,\mbox{ in }\,\,{\mathcal{D}}'({\mathbb{R}}^{n-1}),\,\,\,\mbox{ for each }\,\,j\in\{1,\dots,n-1\},
\\[6pt]
\mbox{whenever the function $v\in{\mathscr{C}}^1({\mathbb{R}}^n_{+})$ satisfies}
\\[4pt]
{\mathcal{N}}_\kappa(\nabla^\ell v)\in L^1_{\rm loc}({\mathbb{R}}^{n-1})
\,\,\mbox{ and }\,\,(\nabla^\ell v)\big|^{{}^{\rm n.t.}}_{\partial{\mathbb{R}}^n_{+}}
\,\,\mbox{ exists, }\,\,\mbox{ for }\,\,\ell\in\{0,1\}.
\end{array}
\end{equation}
Keeping this in mind, as well as \eqref{eq:jufds.26.TRnb}-\eqref{eufds.26fff.ii} and
the last condition in \eqref{KJHab.ttr},
we may conclude that, in the sense of distributions in ${\mathbb{R}}^{n-1}$,
\begin{equation}\label{eq:IIIbcc}
\partial^{\alpha'}\!f=(\partial^{\alpha'}u)\big|^{{}^{\rm n.t.}}_{\partial{\mathbb{R}}^n_{+}}
\in L^p({\mathbb{R}}^{n-1}),\,\,\,\mbox{ for each }\,\,\alpha'\in{\mathbb{N}}_0^{n-1}
\,\,\mbox{ with }\,\,|\alpha'|\leq k.
\end{equation}
Hence $f\in L^p_k({\mathbb{R}}^{n-1})$ which, together with
\eqref{eq:jufds.26.TRnb} and \eqref{KJHab.ttr}, shows that $u$ is a solution of
the higher order Regularity problem $(R^k_p)$ with boundary datum $f$.
This establishes the left-to-right inclusion in \eqref{eq:tfc.1-new.HOR}.

As regards the right-to-left inclusion in \eqref{eq:tfc.1-new.HOR}, we propose to prove
by induction on $k\in{\mathbb{N}}$ that
\begin{equation}\label{eq:tfc.1-new.RRRaaa}
D({\mathbf{A}}^k)\supseteq\big\{f\in L^p_k({\mathbb{R}}^{n-1}):\,
\mbox{$(R^k_p)$ with boundary datum $f$ is solvable}\big\}
\end{equation}
and
\begin{equation}\label{mJBVV.RRR}
\parbox{11.50cm}
{for each $f\in L^p_k({\mathbb{R}}^{n-1})$ with the property
that $(R^k_p)$ with boundary datum $f$ is solvable we have
${\mathbf{A}}^k f=(\partial^k_n u)\big|^{{}^{\rm n.t.}}_{\partial{\mathbb{R}}^n_{+}}$
where $u$ is the (unique) solution of $(R^k_p)$ with boundary datum $f$.}
\end{equation}
Of course, this also yields \eqref{eq:hJB.HOR} (bearing in mind Theorem~\ref{tuFatou}).

To this end, note that the case $k=1$ of \eqref{eq:tfc.1-new.RRRaaa}-\eqref{mJBVV.RRR}
is contained in Theorem~\ref{V-Naa.11}. Assume next that \eqref{eq:tfc.1-new.RRRaaa}-\eqref{mJBVV.RRR}
hold for some $k\in{\mathbb{N}}$, with the goal of establishing analogous results
with $k$ replaced by $k+1$. Fix $f\in L^p_{k+1}({\mathbb{R}}^{n-1})$ with
the property that $(R^{k+1}_p)$ with boundary datum $f$ has a solution $u$.
Since $u$ also solves $(R^{k}_p)$ with boundary datum $f$, the induction hypothesis
implies that
\begin{equation}\label{eq:pyv6}
f\in D({\mathbf{A}}^k)\,\,\mbox{ and }\,\,
{\mathbf{A}}^k f=(\partial^k_n u)\big|^{{}^{\rm n.t.}}_{\partial{\mathbb{R}}^n_{+}}.
\end{equation}
As such, there remains to prove that
\begin{equation}\label{eq:pyv7}
{\mathbf{A}}^k f\in D({\mathbf{A}})
\,\,\mbox{ and }\,\,
{\mathbf{A}}\big({\mathbf{A}}^k f\big)
=(\partial^{k+1}_n u)\big|^{{}^{\rm n.t.}}_{\partial{\mathbb{R}}^n_{+}}.
\end{equation}
Since, by design, ${\mathcal{N}}_{\kappa}(\nabla^{\ell} u)\in L^p({\mathbb{R}}^{n-1})$ and
$L(\nabla^{\ell} u)=0$ in ${\mathbb{R}}^n_{+}$ for each $\ell\in\{0,1,\dots,k+1\}$,
the Fatou result recorded in Theorem~\ref{tuFatou} gives that
$(\nabla^{\ell} u)\big|^{{}^{\rm n.t.}}_{\partial{\mathbb{R}}^n_{+}}$ exists
a.e.~in ${\mathbb{R}}^{n-1}$ and belongs to $L^p(\mathbb{R}^{n-1})$
for each $\ell\in\{0,1,\dots,k+1\}$.
In concert with the observation of a general nature made in \eqref{eq:IIIb.yfr},
this permits us to conclude that
\begin{equation}\label{eq:i6f5}
(\partial^k_n u)\big|^{{}^{\rm n.t.}}_{\partial{\mathbb{R}}^n_{+}}
\in L^p_1({\mathbb{R}}^{n-1}).
\end{equation}
Moreover, the function $v(x',t):=(\partial^k_n u)(x',t)$ for $(x',t)\in{\mathbb{R}}^n_{+}$
solves $(R_p)$ with boundary datum
$(\partial^k_n u)\big|^{{}^{\rm n.t.}}_{\partial{\mathbb{R}}^n_{+}}$.
From this, \eqref{eq:i6f5}, \eqref{eq:pyv6}, as well as \eqref{eq:tfc.1-new} and \eqref{mJBVV}
in Theorem~\ref{V-Naa.11}, we then deduce that ${\mathbf{A}}^k f\in D({\mathbf{A}})$ and
\begin{equation}\label{eq:pyv8}
{\mathbf{A}}\big({\mathbf{A}}^k f\big)
=(\partial_n v)\big|^{{}^{\rm n.t.}}_{\partial{\mathbb{R}}^n_{+}}
=(\partial^{k+1}_n u)\big|^{{}^{\rm n.t.}}_{\partial{\mathbb{R}}^n_{+}}.
\end{equation}
This yields \eqref{eq:pyv7} and completes the proof of \eqref{eq:tfc.1-new.RRRaaa}-\eqref{mJBVV.RRR}.
Next, the density result in \eqref{eq:tfc.1-new.RRR-LAC} is justified in a very similar manner
to the proof of \eqref{eq:tfc.1-new.RRR}, with natural alterations (taking into account what we
have proved so far). At this stage, there remains to observe that the very last claim in the
statement of the theorem is a consequence of \eqref{eq:tfc.2BB.HOR} and Theorem~\ref{them:Dir-l}
(cf. also the comment following its statement).
\end{proof}

We conclude with a comment designed to shed light on the remarkable formula \eqref{mdiab-Lap:LL}
in a more general setting, which we first describe. Concretely, assume that
$L$ is a strongly elliptic $M\times M$ system, with constant complex coefficients,
which has the `block' structure
\begin{equation}\label{eq:Lnay}
L=I_{M\times M}\partial_n^2+L'
\end{equation}
where $I_{M\times M}$ is the $M\times M$ identity matrix and $L'$ has the form
\begin{equation}\label{eq:Lnay2}
L'=\sum_{r,s=1}^{n-1}B_{rs}\partial_r\partial_s,
\end{equation}
for some complex matrices $B_{rs}=\bigl(b_{rs}^{\,\alpha\beta}\bigr)_{1\leq\alpha,\beta\leq M}$ with
$r,s\in\{1,\dots,n-1\}$ (each of which turns out to be strongly elliptic).
As in the past, fix some $p\in(1,\infty)$ and consider the infinitesimal
generator ${\mathbf{A}}$ of the Poisson semigroup associated with $L$ as in Theorem~\ref{VCXga}.
In this context, we claim that (compare with \eqref{mdiab-Lap:LL} in the case when
$L=\Delta=\partial^2_n+\Delta_{n-1}$)
\begin{equation}\label{eq:thb}
{\mathbf{A}}=-\sqrt{-L'},\quad\mbox{ in the sense that }\,\,{\mathbf{A}}^2f=-L'f\,\,
\mbox{ for each }\,\,f\in D({\mathbf{A}}^2).
\end{equation}
In this vein, we wish to note that since by \eqref{eq:tfc.1-new.HOR} we have
$D({\mathbf{A}}^2)\subseteq L^p_2({\mathbb{R}}^{n-1})$,
the action of $L'$ on $f$ is meaningful and $L'f\in L^p({\mathbb{R}}^{n-1})$.

Concerning the proof of \eqref{eq:thb}, start by fixing some $f\in D({\mathbf{A}}^2)$
then define $u$ as in \eqref{eq:FDww.YY}. From \eqref{eq:tfc.1-new.HOR} (with $k=2$) and
Theorem~\ref{Theorem-NiceDP} it follows that
\begin{equation}\label{eq:iyt8742}
\mbox{$u$ solves $(R^2_p)$ with boundary datum $f$}.
\end{equation}
We then have
\begin{align}\label{kgrtrutv5}
{\mathbf{A}}^2f =(\partial_n^2 u)\big|^{{}^{\rm n.t.}}_{\partial{\mathbb{R}}^n_{+}}
=(-L'u)\big|^{{}^{\rm n.t.}}_{\partial{\mathbb{R}}^n_{+}}
=-L'\Big(u\big|^{{}^{\rm n.t.}}_{\partial{\mathbb{R}}^n_{+}}\Big)=-L'f.
\end{align}
The first equality comes from \eqref{eq:iyt8742} and \eqref{eq:hJB.HOR} (with $k=2$).
The second equality employs \eqref{eq:Lnay} and the fact that $Lu=0$ in ${\mathbb{R}}^n_{+}$.
The third equality uses \eqref{eq:Lnay2}, \eqref{eq:IIIb.yfr}, and \eqref{eq:iyt8742}.
Finally, the last equality is implicit in \eqref{eq:iyt8742}. This finishes the proof of \eqref{eq:thb}.

\section{The infinitesimal generator for higher-order operators in Lipschitz domains}
\setcounter{equation}{0}
\label{S-6}

Here the main goal is to prove Theorem~\ref{yenbcxgu}, dealing with the nature of the infinitesimal
generator of the semigroup naturally associated with higher-order operators in graph Lipschitz domains.
This requires a few preliminaries and we begin by considering a graph Lipschitz domain in ${\mathbb{R}}^n$,
i.e., a region of the form
\begin{equation}\label{eq:p8h}
\Omega:=\big\{x=(x',x_n)\in{\mathbb{R}}^{n-1}\times{\mathbb{R}}:\,x_n>\varphi(x')\big\},
\end{equation}
where $\varphi:{\mathbb{R}}^{n-1}\to{\mathbb{R}}$ is a Lipschitz function.
The Lebesgue scale $L^p(\partial\Omega)$, $0<p\leq\infty$, is considered with respect to
the surface measure $\sigma$  on $\partial\Omega$. We shall also need $L^p$-based Sobolev
spaces of order one on $\partial\Omega$, namely
\begin{equation}\label{eq:}
L^p_1(\partial\Omega):=\big\{f\in L^p(\partial\Omega):\,
f(\cdot,\varphi(\cdot))\in L^p_1({\mathbb{R}}^{n-1})\big\},
\end{equation}
for each $p\in(1,\infty)$. A very useful alternative description of these Sobolev spaces is
as follows. Let $\nu=(\nu_1,\dots,\nu_n)$ be the outward unit normal to $\Omega$ and
consider the first-order tangential derivative operators
$\partial_{\tau_{jk}}$ acting on a compactly supported function
$\psi$ of class ${\mathscr{C}}^1$ in a neighborhood of $\partial\Omega$ by
\begin{eqnarray}\label{def-TAU}
\partial_{\tau_{jk}}\psi:=\nu_j(\partial_k\psi)\Bigl|_{\partial\Omega}
-\nu_k(\partial_j\psi)\Bigl|_{\partial\Omega},\qquad j,k=1,\dots,n.
\end{eqnarray}
Then $L^p_1(\partial\Omega)$ may be viewed as the collection of functions $f\in L^p(\partial\Omega)$
such that there exists a constant $C\in(0,\infty)$ with the property that for each $j,k\in\{1,\dots,n\}$
\begin{equation}\label{A.8yt6}
\Bigg|\int_{\partial\Omega}(\partial_{\tau_{jk}}\psi)f\,d\sigma\Bigg|
\leq C\big\|\psi|_{\partial\Omega}\big\|_{L^{p'}(\partial\Omega)},
\qquad\forall\,\psi\in{\mathscr{C}}_0^1({\mathbb{R}}^n),
\end{equation}
where $p'\in(1,\infty)$ is the H\"older conjugate exponent of $p$. On account
of Riesz's representation theorem it follows that $f\in L^p_1(\partial\Omega)$ if and
only if $f\in L^p(\partial\Omega)$ and for each $j,k\in\{1,\dots,n\}$ there exists
$f_{jk}\in L^p(\partial\Omega)$ satisfying
\begin{equation}\label{A.8}
\int_{\partial\Omega}(\partial_{\tau_{jk}}\psi)f\,d\sigma
=-\int_{\partial\Omega}\psi f_{jk}\,d\sigma,
\qquad\forall\,\psi\in{\mathscr{C}}_0^1({\mathbb{R}}^n).
\end{equation}
In such a case, we write $\partial_{\tau_{jk}}f:=f_{jk}$, so
\begin{eqnarray}\label{Lp1-2}
L^p_{1}(\partial\Omega)=\Big\{f\in L^p(\partial\Omega):\,
\partial_{\tau_{jk}}f\in L^p(\partial\Omega),\,\,\,\,j,k=1,\dots,n\Big\}.
\end{eqnarray}
This becomes a Banach space when equipped with the natural norm
\begin{eqnarray}\label{Lp1-3}
\|f\|_{L^p_{1}(\partial\Omega)}:=\|f\|_{L^p(\partial\Omega)}
+\sum_{j,k=1}^{n}\|\partial_{\tau_{jk}}f\|_{L^p(\partial\Omega)}.
\end{eqnarray}
See \cite{IMM} for more on this topic.

Moving on, we continue to assume that $\Omega$ is as in \eqref{eq:p8h}, and
fix $\kappa=\kappa(\Omega)>0$ sufficiently large. In this context, given
$u:\Omega\to{\mathbb{C}}$ for each $x=(x',\varphi(x'))\in\partial\Omega$ define
the nontangential trace, nontangential maximal function, and the area function
of $u$, respectively, by
\begin{align}\label{eq:WW7h65r.A}
\Big(u\big|^{{}^{\rm n.t.}}_{\partial\Omega}\Big)(x) &:=
\lim_{\varphi(x'){\mathbf{e}}_n+\Gamma_\kappa(x')\ni y\to x}u(y),
\\[12pt]
\big({\mathcal{N}}_\kappa u\big)(x) & :=\sup_{y\in\varphi(x'){\mathbf{e}}_n+\Gamma_\kappa(x')}|u(y)|,
\label{eq:WW7h65r.B}
\\[12pt]
\big({\mathcal{A}}_\kappa u\big)(x) & :=\Bigg(\int_{\varphi(x'){\mathbf{e}}_n+\Gamma_\kappa(x')}
\frac{|u(y',t)|^2}{(t-\varphi(x'))^{n-2}}\,dy'\,dt\Bigg)^{1/2}.
\label{eq:WW7h65r.C}
\end{align}
Note that \eqref{eq:WW7h65r.A}-\eqref{eq:WW7h65r.C} reduce to earlier definitions
in the case when $\Omega={\mathbb{R}}^n_{+}$ (i.e., for $\varphi=0$).
As in the past, we shall suppress the dependence on the parameter $\kappa>0$ whenever irrelevant.

Going further, fix $n,m\in{\mathbb{N}}$ with $n\geq 2$, $p\in(1,\infty)$ and denote by
$\gamma$ multi-indices in ${\mathbb{N}}_0^n$. For a family
$\dot{f}=\{f_\gamma\}_{|\gamma|\leq m-1}$ of functions in $L^p(\partial\Omega)$ consider
the compatibility condition
\begin{equation}\label{eq:hdie}
\partial_{\tau_{jk}}f_\gamma=\nu_jf_{\gamma+{\mathbf{e}}_k}-\nu_kf_{\gamma+{\mathbf{e}}_j},
\quad\forall\,\gamma\in{\mathbb{N}}_0^n,\,\,|\gamma|\leq m-2,\,\,\,
\forall\,j,k\in\{1,\dots,n\}
\end{equation}
(a condition assumed to be vacuous if $m=1$).
We then define the space of $L^p$-Whitney arrays as
\begin{equation}\label{eq:jdi}
\dot{L}^p_{m-1}(\partial\Omega):=\Big\{
\dot{f}=\big\{f_\gamma\big\}_{|\gamma|\leq m-1}:\,
f_\gamma\in L^p(\partial\Omega)\mbox{ for }|\gamma|\leq m-1
\mbox{ and satisfy \eqref{eq:hdie}}\Big\},
\end{equation}
and equip this space with the natural norm
\begin{equation}\label{eq:NNau8}
\|\dot{f}\|_{\dot{L}^p_{m-1}(\partial\Omega)}:=\sum_{|\gamma|\leq m-1}\|f_\gamma\|_{L^p(\partial\Omega)},
\qquad\forall\,\dot{f}=\big\{f_\gamma\big\}_{|\gamma|\leq m-1}\in \dot{L}^p_{m-1}(\partial\Omega).
\end{equation}

Fix some $M\in\mathbb{N}$ and consider an $M\times M$
system ${\mathfrak{L}}$ of homogeneous differential operators of order $2m$ as in \eqref{op-Liii}.
Also, suppose $\Omega$ is a graph Lipschitz domain in ${\mathbb{R}}^n$. Then the Dirichlet problem
for the system ${\mathfrak{L}}$ corresponding to the boundary datum
$\dot{f}=\big\{f_\gamma\big\}_{|\gamma|\leq m-1}\in
\dot{L}^p_{m-1}(\partial\Omega)$, where $p\in(1,\infty)$, is formulated as
\begin{equation}\label{Dir-BVP-higher}
(D_{m,p})\,\,
\left\{
\begin{array}{l}
u\in{\mathscr{C}}^\infty(\Omega),
\\[4pt]
{\mathfrak{L}}u=0\,\,\mbox{ in }\,\,\Omega,
\\[4pt]
\mathcal{N}(\nabla^\ell u)\in L^p(\partial\Omega),
\quad\ell\in\{0,1,\dots,m-1\},
\\[6pt]
(\partial^\gamma u)\big|_{\partial\Omega}^{{}^{\rm n.t.}}=f_\gamma\in L^p(\partial\Omega),
\quad |\gamma|\leq m-1.
\end{array}
\right.
\end{equation}
Let us also define the space of $L^p_1$-Whitney arrays, for $p\in(1,\infty)$, by setting
\begin{equation}\label{eq:jdi2}
\dot{L}^p_{1,m-1}(\partial\Omega):=\Big\{
\dot{f}=\big\{f_\gamma\big\}_{|\gamma|\leq m-1}\in
\dot{L}^p_{m-1}(\partial\Omega):\,
f_\gamma\in L^p_1(\partial\Omega)\mbox{ for all }\gamma\Big\},
\end{equation}
and endow this space with the natural norm
\begin{equation}\label{eq:NNau9}
\|\dot{f}\|_{\dot{L}^p_{1,m-1}(\partial\Omega)}:=\sum_{|\gamma|\leq m-1}\|f_\gamma\|_{L^p_1(\partial\Omega)},
\qquad\forall\,\dot{f}=\big\{f_\gamma\big\}_{|\gamma|\leq m-1}\in \dot{L}^p_{1,m-1}(\partial\Omega).
\end{equation}

Regarding the Regularity problem for the system ${\mathfrak{L}}$ as in \eqref{op-Liii}, given the
boundary datum $\dot{f}=\big\{f_\gamma\big\}_{|\gamma|\leq m-1}\in
\dot{L}^p_{1,m-1}(\partial\Omega)$, its formulation is
\begin{equation}\label{Reg-BVP-higher}
(R_{m,p})\,\,
\left\{
\begin{array}{l}
u\in{\mathscr{C}}^\infty(\Omega),
\\[4pt]
{\mathfrak{L}}u=0\,\,\mbox{ in }\,\,\Omega,
\\[4pt]
\mathcal{N}(\nabla^\ell u)\in L^p(\partial\Omega),
\quad\ell\in\{0,1,\dots,m\},
\\[6pt]
(\partial^\gamma u)\big|_{\partial\Omega}^{{}^{\rm n.t.}}=f_\gamma\in L^p_1(\partial\Omega),
\quad |\gamma|\leq m-1.
\end{array}
\right.
\end{equation}

The stage has been set to formulate the main result in this section.

\begin{theorem}\label{yenbcxgu}
Fix $n,m,M\in\mathbb{N}$ with $n\geq 2$, and consider an $M\times M$
system of homogeneous differential operators of order $2m$,
\begin{equation}\label{op-Liii-4}
{\mathfrak{L}}:=\sum_{|\alpha|=2m}{A}_{\alpha}\partial^\alpha,
\end{equation}
with constant complex matrix coefficients ${A}_{\alpha}\in{\mathbb{C}}^{M\times M}$,
for $\alpha\in{\mathbb{N}}_0^n$ with $|\alpha|=2m$.
Assume that ${\mathfrak{L}}$ satisfies the weak ellipticity condition
\begin{equation}\label{LH.w-4}
{\rm det}\Bigg(\sum_{|\alpha|=2m}\xi^\alpha{A}_{\alpha}\Bigg)\not=0,
\qquad\forall\,\xi\in{\mathbb{R}}^n\setminus\{0\}.
\end{equation}
Fix $p\in(1,\infty)$, $\kappa>0$, and assume that $\Omega$ is a graph Lipschitz domain
in ${\mathbb{R}}^n$ for which the following properties hold:
\begin{enumerate}
\item[(i)] The Dirichlet problem $(D_{m,p})$ is well-posed in $\Omega$.
\item[(ii)] If $v\in{\mathscr{C}}^\infty(\Omega)$ is a function satisfying ${\mathfrak{L}}v=0$ in $\Omega$
and ${\mathcal{N}}_\kappa(\nabla^\ell v)\in L^p(\partial\Omega)$ for each $\ell\in\{0,1,\dots,m\}$ and
$(\nabla^\ell v)\big|^{{}^{\rm n.t.}}_{\partial\Omega}$ exists $\sigma$-a.e. on $\partial\Omega$
for each $\ell\in\{0,1,\dots,m-1\}$, then $(\partial^{\gamma+{\mathbf{e}}_n}v)\big|^{{}^{\rm n.t.}}_{\partial\Omega}$
also exists $\sigma$-a.e. on $\partial\Omega$ for every $\gamma\in{\mathbb{N}}_0^n$ with $|\gamma|=m-1$.
\item[(iii)] There exists $C\in(0,\infty)$ such that
if $v\in{\mathscr{C}}^\infty(\Omega)$ is a function satisfying ${\mathfrak{L}}v=0$ in $\Omega$ then
\begin{equation}\label{eq:ihgnur}
\|{\mathcal{A}}_\kappa(\nabla^{m}v)\|_{L^p(\partial\Omega)}
\leq C\|{\mathcal{N}}_\kappa(\nabla^{m-1}v)\|_{L^p(\partial\Omega)}
\end{equation}
and also
\begin{equation}\label{eq:ihgnur.2}
\|{\mathcal{N}}_\kappa(\nabla^{m-1}v)\|_{L^p(\partial\Omega)}\leq
C\|{\mathcal{A}}_\kappa(\nabla^{m}v)\|_{L^p(\partial\Omega)}
\end{equation}
under the additional assumption that
$\lim\limits_{t\to\infty}|(\nabla^\ell v)(x+t{\mathbf{e}}_n)|=0$ for each $x\in\partial\Omega$ and
each $\ell\in\{0,1,\dots,m-1\}$.
\end{enumerate}

Then for each $t>0$, the mapping
\begin{equation}\label{eq:ndusy}
T(t):\dot{L}^p_{m-1}(\partial\Omega)\longrightarrow
\dot{L}^p_{m-1}(\partial\Omega)
\end{equation}
defined by
\begin{equation}\label{eq:ndusy-2}
\big(T(t)\dot{f}\big)(x):=\Big\{(\partial^\gamma u_{\dot{f}})
(x+t{\mathbf{e}}_n)\Big\}_{|\gamma|\leq m-1},\qquad
\forall\,\dot{f}\in\dot{L}^p_{m-1}(\partial\Omega),\quad
\forall\,x\in\partial\Omega,
\end{equation}
where $u_{\dot{f}}$ solves $(D_{m,p})$ for the boundary datum $\dot{f}$,
is well-defined, linear, bounded and
\begin{equation}\label{eq:bdys}
\sup_{t>0}\|T(t)\|_{{\mathcal{L}}\big(\dot{L}^p_{m-1}(\partial\Omega)\big)}<\infty.
\end{equation}
In addition, $\big\{T(t)\big\}_{t>0}$ is a $C_0$-semigroup on the Banach space
$\dot{L}^p_{m-1}(\partial\Omega)$ whose infinitesimal generator
${\mathbf{A}}:D({\mathbf{A}})\to\dot{L}^p_{m-1}(\partial\Omega)$ is given by
\begin{equation}\label{eq:jdid}
D({\mathbf{A}})=\big\{\dot{f}\in
\dot{L}^p_{1,m-1}(\partial\Omega):\,(R_{m,p})\,\,\mbox{ is solvable for the boundary datum }
\dot{f}\big\}
\end{equation}
and
\begin{equation}\label{eq:jdid-a}
{\mathbf{A}}\dot{f}
=\Big\{\big(\partial^{\gamma+{\mathbf{e}}_n}u_{\dot{f}}\big)\Big|^{{}^{\rm n.t.}}_{\partial\Omega}
\Big\}_{|\gamma|\leq m-1},\qquad\forall\,\dot{f}\in D({\mathbf{A}}).
\end{equation}
\end{theorem}

We first deal with the preliminary result below, in the spirit of \cite[Lemma~2.5.1, p.\,213]{St70}.

\begin{lemma}\label{L-Stein}
Fix $0<\kappa<\kappa'$ along with some point $x_0'\in{\mathbb{R}}^{n-1}$, and
let $u\in{\mathscr{C}}^\infty\big(\Gamma_{\kappa'}(x_0')\big)$ be a ${\mathbb{C}}^M$-valued
function. Assume that $\lim\limits_{t\rightarrow\infty}|\nabla u(x+t{\bf e}_n)|=0$
for each $x\in\Gamma_{\kappa'}(x_0')$, and that ${\mathfrak{L}}u=0$ in $\Gamma_{\kappa'}(x_0')$
for some system ${\mathfrak{L}}$ as in Theorem~\ref{yenbcxgu}. Then there exists some
finite constant $C=C_{{\mathfrak{L}},n,\kappa,\kappa'}>0$ with the property that
\begin{align}\label{q-ste1}
\int\limits_{\Gamma_\kappa(x_0')}|\nabla u(x',x_n)|^2 & \big[|x'-x_0'|+x_n\big]^{2-n}\,dx'dx_n
\nonumber\\[4pt]
&\leq C\int\limits_{\Gamma_{\kappa'}(x_0')}|\partial_n u(x',x_n)|^2\big[|x'-x_0'|+x_n\big]^{2-n}\,dx'dx_n.
\end{align}
\end{lemma}

\begin{proof}
Making a translation, there is no loss of generality in assuming that $x_0'=0'$.
Fix some $j\in\{1,\dots,n-1\}$ and first notice that for each fixed
$x=(x',x_n)\in\Gamma_\kappa(0')$ the Fundamental Theorem of
Calculus and the decay of $\nabla u$ at infinity allow us to write
\begin{equation}\label{NW-eq1}
\partial_j u(x)=-\int_{x_n}^\infty\partial_j\partial_n u(x',t)\,dt.
\end{equation}
Also, using interior estimates (cf. Theorem~\ref{ker-sbav}), for every $t>x_n$ we have
\begin{equation}\label{NW-eq2}
|\partial_j\partial_n u(x',t)|\leq\frac{C}{t}
\Big(\aver{B((x',t),\eta\,t)}|\partial_n u(y)|^2\,dy\Big)^{\frac12},
\end{equation}
where $\eta\in(0,1)$ is chosen to be small enough that
$B((x',t),\eta\,t) \subset\Gamma_{\kappa'}(0')$ for all $t>x_n$.
To proceed, for each $t>0$ define
\begin{eqnarray}\label{q-ste2}
D_t:=\big\{(y',y_n)\in\Gamma_{\kappa'}(0'):\,|y_n-t|<\eta t\big\}.
\end{eqnarray}
Since $B((x',t),\eta\,t)\subset D_t$ for every $t>x_n$ it follows that
\begin{equation}\label{NW-eq3}
|\partial_j\partial_n u(x',t)|\leq C\,t^{-(n+2)/2}F_t^{1/2},\quad\mbox{ for each }\,\,t>x_n,
\end{equation}
where we have set
\begin{eqnarray}\label{q-ste3}
F_t:=\int_{D_t}|\partial_n u(y)|^2\,dy,\quad\mbox{ for each }\,\,t>0.
\end{eqnarray}
Consider the spherical cap $S_{\kappa}:=S^{n-1}\cap\Gamma_\kappa(0')$
and, given a generic function $v$ defined in $\Gamma_\kappa(0')$,
for each $\theta\in S_\kappa$ define $v_\theta(s):=v(s\theta)$ where $s\in(0,\infty)$.
Then combining \eqref{NW-eq1}, \eqref{NW-eq3}, and taking into account the geometry of
$\Gamma_\kappa(0')$, we see that there exists a finite constant $c=c(\kappa)>0$ such that
\begin{equation}\label{q-ste4}
\big|(\partial_j u)_\theta(s)\big|\leq C\int_{cs}^\infty t^{-(n+2)/2}F_t^{1/2}\,dt
\end{equation}
On the one hand, this and Hardy's inequality (cf., e.g., \cite[A.4, p.\,272]{St70})
yield
\begin{equation}\label{NW-eq4}
\int_0^\infty\big|(\partial_j u)_\theta(s)\big|^2\,s\,ds\leq C\int_0^\infty t^{1-n}F_t\,dt.
\end{equation}
On the other hand, Fubini's Theorem gives
\begin{align}\label{NW-eq5}
\int_0^\infty t^{1-n} F_t\,dt &=\int_0^\infty t^{1-n}\Big(\int_{D_t}|\partial_n u(y)|^2\,dy\Big)\,dt
\nonumber\\[6pt]
&=\int_{\Gamma_{\kappa'}(0')}|\partial_n u(y)|^2\Big(\int_0^\infty t^{1-n}{\bf 1}_{D_t}(y)\,dt\Big)\,dy
\nonumber\\[6pt]
&\leq\int_{\Gamma_{\kappa'}(0')}|\partial_n u(y)|^2
\Big(\int_{y_n/(1+\eta)}^{y_n/(1-\eta)} t^{1-n}\,dt\Big)\,dy
\nonumber\\[6pt]
&\leq C\int_{\Gamma_{\kappa'}(0')}|\partial_n u(y)|^2\,y_n^{2-n}\,dy.
\end{align}
Then using \eqref{NW-eq4} and \eqref{NW-eq5} and polar coordinates, we obtain
\begin{align}\label{q-ste5}
\int_{\Gamma_\kappa(0')}|\partial_j u(x)|^2\,x_n^{2-n}\,dx
&\leq C\int_{S_\kappa}\int_0^\infty\big|(\partial_j u)_\theta(s)\big|^2\,s\,ds\,d\theta
\nonumber\\[4pt]
& \leq C\int_{\Gamma_{\kappa'}(0')}|\partial_n u(x)|^2\,x_n^{2-n}\,dx.
\end{align}
Hence,
\begin{eqnarray}\label{q-ste1rf}
\int\limits_{\Gamma_\kappa(0')}|\nabla u(x)|^2 x_n^{2-n}\,dx
\leq C_{L,n,\kappa,\kappa'}\int\limits_{\Gamma_{\kappa'}(0')}|\partial_n u(x)|^2 x_n^{2-n}\,dx,
\end{eqnarray}
and the version of \eqref{q-ste1} corresponding to $x_0'=0'$
now readily follows from this by observing that $|x'|+x_n\approx x_n$ uniformly for
$x=(x',x_n)\in\Gamma_{\kappa}(0')$ and $x=(x',x_n)\in\Gamma_{\kappa'}(0')$.
\end{proof}

After this preamble, we are ready to present the proof of Theorem~\ref{yenbcxgu}.

\vskip 0.08in
\begin{proof}[Proof of Theorem~\ref{yenbcxgu}]
The fact that the mapping in \eqref{eq:ndusy}-\eqref{eq:ndusy-2} is well-defined, linear and
bounded, as well as estimate \eqref{eq:bdys},
follow from the well-posedness of $(D_{m,p})$ and \eqref{def-TAU}.
To prove that $\big\{T(t)\big\}_{t>0}$ is a $C_0$-semigroup on $\dot{L}^p_{m-1}(\partial\Omega)$,
fix some $\dot{f}\in\dot{L}^p_{m-1}(\partial\Omega)$ and consider $t_1,t_2\in(0,\infty)$ arbitrary.
Also, set
\begin{equation}\label{eq:jdu}
\dot{g}:=\big\{(\partial^\gamma u_{\dot{f}})(\cdot+t_1{\mathbf{e}}_n)\big|_{\partial\Omega}
\big\}_{|\gamma|\leq m-1},
\end{equation}
where $u_{\dot{f}}$ is the solution to $(D_{m,p})$ for the boundary datum $\dot{f}$.
Then, by definition, $T(t_1)\dot{f}=\dot{g}$. Since $u_{\dot{f}}(\cdot+t_1{\mathbf{e}}_n)$
as a function defined in $\Omega$ is a solution to $(D_{m,p})$ with boundary datum $\dot{g}$, and
$(D_{m,p})$ is well-posed, it follows that $u_{\dot{g}}=u_{\dot{f}}(\cdot+t_1{\mathbf{e}}_n)$ in $\Omega$.
Hence,
\begin{align}\label{eq:MKG}
T(t_2)T(t_1)\dot{f}=T(t_2)\dot{g}
&=\big\{(\partial^\gamma u_{\dot{g}})(\cdot+t_2{\mathbf{e}}_n)\big|_{\partial\Omega}
\big\}_{|\gamma|\leq m-1}
\nonumber\\[4pt]
&=\big\{(\partial^\gamma u_{\dot{f}})\big((\cdot+t_2{\mathbf{e}}_n)+t_1{\mathbf{e}}_n\big)\big|_{\partial\Omega}
\big\}_{|\gamma|\leq m-1}
\nonumber\\[4pt]
&=\big\{(\partial^\gamma u_{\dot{f}})(\cdot+(t_1+t_2){\mathbf{e}}_n)\big|_{\partial\Omega}
\big\}_{|\gamma|\leq m-1}
\nonumber\\[4pt]
&=T(t_1+t_2)\dot{f},
\end{align}
as wanted.

Next, we prove \eqref{eq:jdid}. Suppose
$\dot{f}\in D({\mathbf{A}})\subseteq\dot{L}^p_{m-1}(\partial\Omega)$ and set
$\dot{g}:={\mathbf{A}}\dot{f}\in\dot{L}^p_{m-1}(\partial\Omega)$. As a preamble,
we note that by reasoning in a similar manner as in \eqref{detraz.TTT.UU}, based on the interior estimates
from Theorem~\ref{ker-sbav}, for each $v\in{\mathscr{C}}^\infty(\Omega)$
satisfying ${\mathfrak{L}}v=0$ in $\Omega$ and each $k\in{\mathbb{N}}_0$ there exists
a finite constant $C=C(\Omega,{\mathfrak{L}},\kappa,k)>0$ such that
\begin{equation}\label{deUYG}
\big|(\nabla^k v)(x+t{\mathbf{e}}_n)\big|
\leq\frac{C}{t^{k}}\,{\mathcal{M}}_{\partial\Omega}\big({\mathcal{N}}_\kappa v\big)(x),
\qquad\forall\,x\in\partial\Omega,\quad\forall\,t>0,
\end{equation}
where ${\mathcal{M}}_{\partial\Omega}$ is the Hardy-Littlewood
maximal operator on $\partial\Omega$. Next, pick an arbitrary
$\gamma\in{\mathbb{N}}_0^n$ with $|\gamma|\leq m-1$. We claim that
for each $t>0$ fixed we have
\begin{equation}\label{hcduT.AAA}
\frac{d}{dt}\Big[(\partial^{\gamma}u_{\dot{f}})(\cdot+t{\mathbf{e}}_n)\big|_{\partial\Omega}\Big]
=(\partial^{\gamma+{\mathbf{e}}_n}u_{\dot{f}})
(\cdot+t{\mathbf{e}}_n)\big|_{\partial\Omega}
\quad\mbox{ in }\,\,L^p(\partial\Omega).
\end{equation}
Indeed, using the one-dimensional Mean Value Theorem matters are readily reduced
to showing that if for each $x\in\partial\Omega$ and each $s\in{\mathbb{R}}$
some number $\xi_{x,s}$ has been chosen in between $0$ and $s$ then
\begin{equation}\label{hcduT.Abb}
\lim_{s\to 0}(\partial^{\gamma+{\mathbf{e}}_n}u_{\dot{f}})
\big(\cdot+(t+\xi_{\cdot,s}){\mathbf{e}}_n\big)\big|_{\partial\Omega}
=(\partial^{\gamma+{\mathbf{e}}_n}u_{\dot{f}})
(\cdot+t{\mathbf{e}}_n)\big|_{\partial\Omega}
\quad\mbox{ in }\,\,L^p(\partial\Omega).
\end{equation}
To justify this, note that, as seen from \eqref{deUYG} (used here with
$v:=\partial^{\gamma}u_{\dot{f}}$ and $k=1$), there exists a finite
constant $C=C(\Omega,{\mathfrak{L}},\kappa)>0$ such that whenever $\xi\in(-t/2,t/2)$ we have
\begin{equation}\label{deUYG.222}
\big|(\partial^{\gamma+{\mathbf{e}}_n}u_{\dot{f}})
\big(x+(t+\xi){\mathbf{e}}_n\big)\big|
\leq\frac{C}{t}\,{\mathcal{M}}_{\partial\Omega}\big({\mathcal{N}}_\kappa
(\partial^{\gamma}u_{\dot{f}})\big)(x),
\qquad\forall\,x\in\partial\Omega.
\end{equation}
Then \eqref{hcduT.Abb} follows
from \eqref{deUYG.222} and Lebesgue's Dominated Convergence Theorem, keeping in mind that
$(\partial^{\gamma+{\mathbf{e}}_n}u_{\dot{f}})(\cdot+t{\mathbf{e}}_n)$
is continuous on $\overline{\Omega}$, ${\mathcal{N}}_\kappa(\partial^{\gamma}u_{\dot{f}})
\in L^p(\partial\Omega)$, and that ${\mathcal{M}}_{\partial\Omega}$
is bounded on $L^p(\partial\Omega)$. This proves \eqref{hcduT.AAA}.

At this stage, from \eqref{hcduT.AAA}, \eqref{eq:ADCV.ttt},
and \eqref{eq:ndusy-2} we conclude that, for each $t>0$ fixed, the following sequence
of equalities holds in $\dot{L}^p_{m-1}(\partial\Omega)$:
\begin{align}\label{hcduT}
\Big\{(\partial^{\gamma+{\mathbf{e}}_n}u_{\dot{f}})
(\cdot+t{\mathbf{e}}_n)\big|_{\partial\Omega}\Big\}_{|\gamma|\leq m-1}
&=\Big\{\frac{d}{dt}\Big[
(\partial^{\gamma}u_{\dot{f}})(\cdot+t{\mathbf{e}}_n)\big|_{\partial\Omega}\Big]
\Big\}_{|\gamma|\leq m-1}
\nonumber\\[4pt]
&=\frac{d}{dt}\big[T(t)\dot{f}\big]=T(t){\mathbf{A}}\dot{f}=T(t)\dot{g}
\nonumber\\[4pt]
&=\Big\{(\partial^{\gamma}u_{\dot{g}})(\cdot+t{\mathbf{e}}_n)\big|_{\partial\Omega}
\Big\}_{|\gamma|\leq m-1}.
\end{align}
In turn, from \eqref{hcduT} and the well-posedness of $(D_{m,p})$ we deduce that
\begin{equation}\label{hcduT-re}
\partial_n u_{\dot{f}}=u_{\dot{g}}\,\,\,\mbox{ in }\,\,\Omega.
\end{equation}
Fix now $\gamma\in{\mathbb{N}}_0^n$ with $|\gamma|=m-1$. Since from the well-posedness
of $(D_{m,p})$ we know that ${\mathcal{N}}_\kappa(\partial^{\gamma} u_{\dot{g}})\in L^p(\partial\Omega)$,
in light of \eqref{hcduT-re} we have
\begin{equation}\label{eq:ncus}
{\mathcal{N}}_\kappa(\partial^{\gamma+{\mathbf{e}}_n} u_{\dot{f}})\in L^p(\partial\Omega).
\end{equation}
In addition, by reasoning as in \eqref{eq:jufds.23AAA} we see that derivatives of
$u_{\dot{f}}$ decay at infinity in the precise sense that for each $x\in\partial\Omega$,
each $t>0$, and each $k\in{\mathbb{N}}_0$,
\begin{align}\label{uht5f5}
\big|(\nabla^k u_{\dot{f}})(x+t{\mathbf{e}}_n)\big|\leq C\,t^{-k-(n-1)/p}\,
\|{\mathcal{N}}_\kappa u_{\dot{f}}\|_{L^p(\partial\Omega)},
\end{align}
where $C=C(\Omega,{\mathfrak{L}},n,k,\kappa,p)\in(0,\infty)$.
Then, having fixed some $\kappa'>\kappa$, for each $j\in\{1,\dots,n\}$ we may write
\begin{align}\label{kdisQ}
\big\|{\mathcal{N}}_\kappa\big(\partial^{\gamma+{\mathbf{e}}_j} u_{\dot{f}}\big)\big\|
_{L^p(\partial\Omega)}
& \leq\big\|{\mathcal{N}}_\kappa\big(\nabla^{m-1}(\partial_j u_{\dot{f}})\big)\big\|
_{L^p(\partial\Omega)}
\nonumber\\[4pt]
&\leq C\big\|{\mathcal{A}}_\kappa\big(\nabla^{m}(\partial_j u_{\dot{f}})\big)\big\|
_{L^p(\partial\Omega)}
\nonumber\\[4pt]
&\leq C\big\|{\mathcal{A}}_{\kappa'}\big(\nabla^{m}(\partial_n u_{\dot{f}})\big)\big\|
_{L^p(\partial\Omega)}
\nonumber\\[4pt]
&\leq C\big\|{\mathcal{N}}_{\kappa'}\big(\nabla^{m-1}(\partial_n u_{\dot{f}})\big)\big\|
_{L^p(\partial\Omega)}
\nonumber\\[4pt]
&\leq C\big\|{\mathcal{N}}_{\kappa}\big(\nabla^{m-1}(\partial_n u_{\dot{f}})\big)\big\|
_{L^p(\partial\Omega)}<+\infty.
\end{align}
Above, the first inequality is obvious, the second is a consequence of
\eqref{eq:ihgnur.2} in assumption {\it (iii)}
in the statement and \eqref{uht5f5}, the third follows from Lemma~\ref{L-Stein}
(used with $u:=\nabla^m u_{\dot{f}}$) and \eqref{uht5f5} (used with $k:=m+1$),
the fourth uses \eqref{eq:ihgnur} in assumption {\it (iii)}, the fifth relies on
\cite[Proposition~2.1, p.\,22]{IMM}, while the sixth is implied by \eqref{eq:ncus}.

Since $\gamma\in{\mathbb{N}}_0^n$ with $|\gamma|=m-1$ and $j\in\{1,\dots,n\}$ are arbitrary, we
conclude from \eqref{kdisQ} that
\begin{equation}\label{eq:dusFD.it43}
{\mathcal{N}}_\kappa(\nabla^m u_{\dot{f}})\in L^p(\partial\Omega).
\end{equation}
In concert with the fact that $u_{\dot{f}}$ solves $(D_{m,p})$ with datum $\dot{f}$, this further
implies that
\begin{equation}\label{eq:dusFD}
{\mathcal{N}}_\kappa(\nabla^\ell u_{\dot{f}})\in L^p(\partial\Omega),\qquad
\forall\,\ell\in\{0,1,\dots,m\}.
\end{equation}

Moving on, we claim that $\dot{f}=\{f_\gamma\}_{|\gamma|\leq m-1}\in \dot{L}^p_{m-1}(\partial\Omega)$
actually belongs to $\dot{L}^p_{1,m-1}(\partial\Omega)$. By \eqref{eq:jdi2} we need to show that
$f_\gamma\in L^p_1(\partial\Omega)$ for each $\gamma$. When $|\gamma|\leq m-2$, this is seen
directly from \eqref{Lp1-2} and \eqref{eq:hdie}-\eqref{eq:jdi}. Suppose now that
$\gamma\in{\mathbb{N}}_0^n$ satisfies $|\gamma|=m-1$ and pick an arbitrary
$\psi\in{\mathscr{C}}_0^1({\mathbb{R}}^n)$. Then, with $p'\in(1,\infty)$ denoting
the H\"older conjugate exponent of $p$, for each $j,k\in\{1,\dots,n\}$ we may write
\begin{align}\label{A.8yiuyty}
\Bigg|\int_{\partial\Omega}(\partial_{\tau_{jk}}\psi)f_\gamma\,d\sigma\Bigg|
&=\Bigg|\int_{\partial\Omega}(\partial_{\tau_{jk}}\psi)
\big(\partial^\gamma u_{\dot{f}}\big)\big|^{{}^{\rm n.t.}}_{\partial\Omega}\,d\sigma\Bigg|
\nonumber\\[4pt]
&=\Bigg|\lim_{t\to 0^{+}}\int_{\partial\Omega}(\partial_{\tau_{jk}}\psi)
\big(\partial^\gamma u_{\dot{f}}\big)(\cdot+t{\mathbf{e}}_n)\,d\sigma\Bigg|
\nonumber\\[4pt]
&=\Bigg|\lim_{t\to 0^{+}}\int_{\partial\Omega}\psi\Big(
\nu_k\big(\partial^{\gamma+{\mathbf{e}}_j} u_{\dot{f}}\big)(\cdot+t{\mathbf{e}}_n)
-\nu_j\big(\partial^{\gamma+{\mathbf{e}}_k} u_{\dot{f}}\big)(\cdot+t{\mathbf{e}}_n)\Big)\,d\sigma\Bigg|
\nonumber\\[4pt]
&\leq 2\int_{\partial\Omega}|\psi|
{\mathcal{N}}_\kappa\big(\nabla^m u_{\dot{f}}\big)\,d\sigma
\nonumber\\[4pt]
&\leq 2\big\|{\mathcal{N}}_\kappa\big(\nabla^m u_{\dot{f}}\big)\big\|_{L^p(\partial\Omega)}
\big\|\psi|_{\partial\Omega}\big\|_{L^{p'}(\partial\Omega)}.
\end{align}
In \eqref{A.8yiuyty}, the first equality follows from the definition of $u_{\dot{f}}$, while
the second equality uses Lebesgue's Dominated Convergence Theorem and \eqref{eq:dusFD}.
To justify the third equality in \eqref{A.8yiuyty}, consider the vector field
\begin{align}\label{eq:3gtg}
\vec{F} &:=\Big\{(\partial_k\psi)\big(\partial^\gamma u_{\dot{f}}\big)(\cdot+t{\mathbf{e}}_n)
+\psi\big(\partial^{\gamma+{\mathbf{e}}_k} u_{\dot{f}}\big)(\cdot+t{\mathbf{e}}_n)\Big\}{\mathbf{e}}_j
\nonumber\\[6pt]
&\quad\,\,
-\Big\{(\partial_j\psi)\big(\partial^\gamma u_{\dot{f}}\big)(\cdot+t{\mathbf{e}}_n)
+\psi\big(\partial^{\gamma+{\mathbf{e}}_j}
u_{\dot{f}}\big)(\cdot+t{\mathbf{e}}_n)\Big\}{\mathbf{e}}_k
\,\,\,\mbox{ in }\,\,\Omega,
\end{align}
which is continuous in $\Omega$, hence belongs to $L^1_{\rm loc}(\Omega)$, and satisfies
${\rm div}\,\vec{F}=0$ in $\Omega$. Also, ${\mathcal{N}}_\kappa(\vec{F})$ is a
compactly supported function in $L^p(\partial\Omega)$, hence
${\mathcal{N}}_\kappa(\vec{F})\in L^1(\partial\Omega)$. Moreover,
$\vec{F}\big|^{{}^{\rm n.t.}}_{\partial\Omega}$ exists $\sigma$-a.e.
on $\partial\Omega$ and
\begin{align}\label{eq:juytr}
\nu\cdot\Big(\vec{F}\big|^{{}^{\rm n.t.}}_{\partial\Omega}\Big)
&=(\partial_{\tau_{jk}}\psi)\big(\partial^\gamma u_{\dot{f}}\big)(\cdot+t{\mathbf{e}}_n)
\nonumber\\[4pt]
&\quad
-\psi\Big(\nu_k\big(\partial^{\gamma+{\mathbf{e}}_j} u_{\dot{f}}\big)(\cdot+t{\mathbf{e}}_n)
-\nu_j\big(\partial^{\gamma+{\mathbf{e}}_k} u_{\dot{f}}\big)(\cdot+t{\mathbf{e}}_n)\Big).
\end{align}
Then the desired equality is obtained by applying the version of the Divergence Theorem
recorded in \cite[Proposition~2.3, p.\,28]{IMM}. The last step in \eqref{A.8yiuyty}
involves H\"older's inequality. This establishes \eqref{A.8yiuyty} which, in concert with
\eqref{A.8yt6}, proves that $f_\gamma\in L^p_1(\partial\Omega)$ for each
$\gamma\in{\mathbb{N}}_0^n$ with $|\gamma|=m-1$. Hence, $\dot{f}\in\dot{L}^p_{1,m-1}(\partial\Omega)$
and we may conclude that $u_{\dot{f}}$ solves $(R_{m,p})$ for the boundary datum $\dot{f}$.
This completes the proof of the left-to-right inclusion in \eqref{eq:jdid}.

Our next goal is to prove the right-to-left inclusion in \eqref{eq:jdid} and also
establish formula \eqref{eq:jdid-a}. To this end, fix some
$\dot{f}\in\dot{L}^p_{1,m-1}(\partial\Omega)\subset\dot{L}^p_{m-1}(\partial\Omega)$ with the
property that $(R_{m,p})$ is solvable for the boundary datum $\dot{f}$ and let $u_{\dot{f}}$
be the corresponding solution. Bearing in mind the Fatou type property assumed in part {\it (ii)}
of the statement, by reasoning in a similar fashion to \eqref{eq:hds66}-\eqref{eq:hds69}
we obtain that
\begin{equation}\label{FrGd9}
\lim_{t\to 0^{+}}\frac{T(t)\dot{f}-\dot{f}}{t}
=\Big\{
\big(\partial^{\gamma+{\mathbf{e}}_n}u_{\dot{f}}\big)\Big|^{{}^{\rm n.t.}}_{\partial\Omega}
\Big\}_{|\gamma|\leq m-1}
\quad\mbox{ component-wise in }\,\,L^p(\partial\Omega).
\end{equation}
To show that actually the above limit may, in fact, be considered in the space
$\dot{L}^p_{m-1}(\partial\Omega)$ we need to ensure that
\begin{equation}\label{FrGjYmn}
\Big\{\big(\partial^{\gamma+{\mathbf{e}}_n}u_{\dot{f}}\big)\Big|^{{}^{\rm n.t.}}_{\partial\Omega}
\Big\}_{|\gamma|\leq m-1}\in\dot{L}^p_{m-1}(\partial\Omega).
\end{equation}
In turn, this comes down to verifying the compatibility conditions \eqref{eq:hdie} for the components of the
array in \eqref{FrGjYmn}. A useful result in this regard, obtained as in the justification
of the third equality in \eqref{A.8yiuyty}, (compare with \cite[Proposition~2.7, p.\,36]{IMM})
is that
\begin{equation}\label{eq:53dce}
\parbox{11.00cm}
{if $v\in{\mathscr{C}}^1(\Omega)$ has
${\mathcal{N}}v,\,{\mathcal{N}}(\nabla v)\in L^p(\partial\Omega)$
and $v\big|^{{}^{\rm n.t.}}_{\partial\Omega}$, $(\nabla v)\big|^{{}^{\rm n.t.}}_{\partial\Omega}$
exist $\sigma$-a.e. on $\partial\Omega$, then
$v\big|^{{}^{\rm n.t.}}_{\partial\Omega}\in L^p_1(\partial\Omega)$
and for each $j,k\in\{1,\dots,n\}$ there holds
$\partial_{\tau_{jk}}\big(v\big|^{{}^{\rm n.t.}}_{\partial\Omega}\big)
=\nu_j(\partial_k v)\big|^{{}^{\rm n.t.}}_{\partial\Omega}
-\nu_k(\partial_j v)\big|^{{}^{\rm n.t.}}_{\partial\Omega}$ on $\partial\Omega$.}
\end{equation}
Then \eqref{FrGjYmn} follows by applying \eqref{eq:53dce} with
$v:=\partial^{\gamma+{\mathbf{e}}_n}u_{\dot{f}}$ for $\gamma\in{\mathbb{N}}_0^n$
with $|\gamma|\leq m-2$. Having established this, we may finally conclude
from \eqref{FrGd9}-\eqref{FrGjYmn} that
$\dot{f}\in D({\mathbf{A}})$ and \eqref{eq:jdid-a} holds.
\end{proof}

We conclude by recording the following significant consequence of our previous theorem
which, in particular, is applicable to the polyharmonic operator $\Delta^m$ for each $m\in{\mathbb{N}}$
(hence also to the Laplacian, in which case we recover Dahlberg's result from \cite{Dah}),
as well as to iterations of the Lam\'e system.

\begin{corollary}\label{yenbcxgu.CCC}
Fix $n,m,M\in\mathbb{N}$ with $n\geq 2$, and consider an $M\times M$
system ${\mathfrak{L}}$ of homogeneous differential operators of order $2m$,
with constant, real, symmetric, matrix coefficients. That is,
${\mathfrak{L}}=\sum_{|\alpha|=2m}{A}_{\alpha}\partial^\alpha$ with
${A}_{\alpha}\in{\mathbb{R}}^{M\times M}$ satisfying
${A}_{\alpha}={A}_{\alpha}^\top$ for each $\alpha\in{\mathbb{N}}_0^n$ with $|\alpha|=2m$.
Assume that ${\mathfrak{L}}$ is strongly elliptic, in the sense that there exists $c>0$ so that
\begin{equation}\label{LH.wigff}
{\rm Re}\Bigg(\sum_{|\alpha|=2m}\xi^\alpha{A}_{\alpha}\eta\cdot\overline\eta\Bigg)\geq c|\xi|^{2m}|\eta|^2,
\qquad\forall\,\xi\in{\mathbb{R}}^n,\quad\forall\,\eta\in{\mathbb{C}}^M.
\end{equation}
Then for each graph Lipschitz domain $\Omega$ in ${\mathbb{R}}^n$ there exists $p_o\in[1,2)$,
depending only on $\Omega$ and ${\mathfrak{L}}$, with the following significance.

Whenever $p\in(p_o,p_o')$ where $1/p_o+1/p_o'=1$ it follows that the Dirichlet problem $(D_{m,p})$
for ${\mathfrak{L}}$ is well-posed in $\Omega$ and the mapping $T(t)$ defined for each $t>0$
as in \eqref{eq:ndusy}-\eqref{eq:ndusy-2} is well-defined, linear, bounded and satisfies
\eqref{eq:bdys}. In fact, $\big\{T(t)\big\}_{t>0}$ is a $C_0$-semigroup on the Banach space
$\dot{L}^p_{m-1}(\partial\Omega)$ whose infinitesimal generator
${\mathbf{A}}$ has domain $D({\mathbf{A}})=\dot{L}^p_{1,m-1}(\partial\Omega)$ and acts as in
\eqref{eq:jdid-a}.
\end{corollary}

\begin{proof}
This follows from Theorem~\ref{yenbcxgu} taking into account the well-posedness
results in \cite{PV} (which also imply a suitable Fatou type theorem) and the
comparability of the nontangential maximal function and the area function
for such systems from \cite{DKPV}.
\end{proof}

\end{document}